\documentclass[11pt]{amsart}   	        
\usepackage[dvipsnames]{xcolor}   		             		
\usepackage{graphicx}				
\usepackage{amssymb,mathrsfs}
\usepackage{amssymb,amsthm,amsmath,stmaryrd}
\SetSymbolFont{stmry}{bold}{U}{stmry}{m}{n}
\usepackage{tikz}
\usepackage{tikz-cd}
\usepackage{accents,upgreek,enumerate}
\usepackage[all]{xy}
\usepackage{scalerel}
\usepackage{stackengine}

\numberwithin{equation}{section}

\newcommand{\CC}{\mathbb{C}}
\newcommand{\EE}{\mathbb{E}}
\newcommand{\FF}{\mathbb{F}}
\newcommand{\GG}{\mathbb{G}}

\newcommand{\PP}{\mathbb{P}}
\newcommand{\QQ}{\mathbb{Q}}

\newcommand{\WW}{\mathbb{W}}
\newcommand{\ZZ}{\mathbb{Z}}

\newcommand{\cal}{\mathcal}

\def\cA{{\cal A}}
\def\cB{{\cal B}}

\def\cE{{\cal E}}
\def\cF{{\cal F}}
\def\cG{{\cal G}}
\def\cH{{\cal H}}

\def\cM{{\cal M}}

\def\cO{{\cal O}}

\def\cV{{\cal V}}

\def\cI{{\cal I}}
\def\cJ{{\cal J}}

\def\fD{\mathfrak{D}}

\def\and{\quad{\rm and}\quad}
\def\lra{\longrightarrow }

\def\hra{\hookrightarrow}

\def\dra{\dashrightarrow}

\def\commra{\circlearrowright}

 \DeclareMathOperator{\Ext}{Ext}

\newtheorem{prop}{Proposition}[section]
\newtheorem{theo}[prop]{Theorem}
\newtheorem{lemm}[prop]{Lemma}

\newtheorem{rema}[prop]{Remark}
\newtheorem{exam}[prop]{Example}

\newtheorem{defi}[prop]{Definition}

\newtheorem{defiprop}[prop]{Definition-Proposition}

\def\beq{\begin{equation}}
\def\eeq{\end{equation}}

\def\dual{^{\vee}}
\def\Ob{{\mathcal Ob}}
\def\virt{^{\mathrm{vir}} }

\def\Spec{\mathrm{Spec} }

\def\vphi{\varphi}

\def\surra{\twoheadrightarrow}

\def\hra{\hookrightarrow}

\def\bar{\overline}
\def\rou{\partial}

\def\wtil{\widetilde}
\def\what{\widehat}

\def\cseq{^{\bullet}}
\def\hseq{_{\bullet}}
\def\fN{\mathfrak{N}}
\def\coker{\mathrm{coker}}

\def\deg{\mathrm{deg}}
\def\dim{\mathrm{dim}}

\def\Spec{\mathrm{Spec}}

\newcommand{\bs}[1]{\boldsymbol{#1}}

\newcommand{\Gysin}[1]{0_{#1}^!}

\newcommand{\bdst}[1]{h^1/h^0(#1)}

\newcommand\reallywidetilde[1]{\ThisStyle{%
  \setbox0=\hbox{$\SavedStyle#1$}%
  \stackengine{-.1\LMpt}{$\SavedStyle#1$}{%
    \stretchto{\scaleto{\SavedStyle\mkern.2mu\sim}{.5467\wd0}}{.7\ht0}%
  }{O}{c}{F}{T}{S}%
}}

\def\wwtil#1{$%
  \reallywidetilde{#1}\,
$\par}

\title[Intersection theory of coherent sheaf stacks]{Intersection theories of coherent sheaf stacks and virtual pull-backs via semi-perfect obstruction theories}
\author{Sanghyeon Lee}
\address{Korea Institute for Advanced Study, 85 Hoegiro, Dongdaemun-gu,
Seoul, Republic of Korea}
\email{sanghyeon@kias.re.kr}

\thanks{}

\date{}

\begin{document}

\begin{abstract} In this paper, we construct proper pushforwards and flat pullbacks in Chow groups of coherent sheaf stacks over a Deligne-Mumford(DM) stack.

When there is a relative semi-perfect obstruction theory for a DM-type morphism $X \to Y$, $X$ is a DM stack and $Y$ is a DM stack or a smooth Artin stack, we define a virtual pull-back as a bivariant class. This is an analogue of virtual pull-backs defined by Manolache.
\end{abstract}

\maketitle
\tableofcontents

\section{Introduction}
The definition of a semi-perfect obstruction theory was first introduced by H. -L. Chang and J. Li in \cite{CL11semi}. It is a generalization of a perfect obstruction theory, which is defined in \cite{BF97, LT98}. We note that many parts of this paper contain review and modification of the results in \cite{CL11semi}.

In this paper, we will work over an algebraically closed base field $k$ with $\mathrm{char}(k)=0$. Moreover we assume that all algebraic stacks are finite type over $k$. 

For a representable morphism $f : X \to Y$ of stacks, where $X$ is a DM stack and $Y$ is a DM stack or smooth Artin stack, a semi-perfect obstruction theory $\phi$ over $X/Y$ is determined by the following data :

\begin{itemize}
\item[-]
An \'etale open cover $\{ X_{\alpha} \}_{\alpha \in \cI}$ of $X$.
\item[-]
For each $\alpha \in \cI$, a perfect obstruction theory $\phi_{\alpha} : E_{\alpha} \to L_{X_{\alpha}/Y}$
\item[-] A transition isomorphism $\psi_{\alpha \beta} : h^1(E_{\alpha}\dual)|_{X_{\beta}} \to h^1(E_{\beta}\dual)|_{X_{\alpha}}$, such that $\psi_{\alpha \beta}$ is a $\nu$-equivalence.
\item[-] A coherent sheaf $Ob(\phi)$ such that $\cE|_{X_{\alpha}} \cong Ob(\phi_{\alpha})=:h^1(E_{\alpha}\dual)$.
\end{itemize}
The notion of a $\nu$-equivalence will be explained in Section \ref{sec:semiperf}. Briefly, it means a morphism between two obstruction spaces which preserves elements in obstruction spaces parametrizing obstructions to infinitesimal lifting problems.

In our paper, review the definition of the Chow group of coherent sheaf stacks appeared in \cite{CL11semi}. We use slightly different definition of boundary maps and Chow groups for some technical reason. This will be appeared in Section \ref{sec:cycle},\ref{sec:rationaleq}.

We first define proper pushforwards and flat pull-backs between two cycle groups of coherent sheaf stacks, and prove that they commutes with boundary maps, hence these morphisms are defined in Chow groups. This will be appeared in Section \ref{sec:properpush},\ref{sec:flatpullback} and \ref{sec:rationaleq}.

We review the definition of Gysin homomorphism of coherent sheaf stacks in \cite{CL11semi} and re-describe it in a modified way. We prove the compatiblity with Gysin homomorphism and proper pushforwards, flat pullbacks. This will be appeared in Section \ref{sec:gysin}.

Next, similar to the definition of virtual pull-backs defined by C. Manolache in \cite{Man08}, we define a semi-virtual pull-backs when there are compatible semi-perfect obstruction theories. This will be appeared in Section \ref{sec:semivirtpullback-subsection}.

Then we proof various functorial properties of semi-virtual pull-backs, which are appeared in \cite{Man08}. We prove that semi-virtual pullbacks and proper pushforwards, flat pullbacks are compatible. Also we prove functoriality property(Proposition \ref{functoriality}). As a result, we show that semi-virtual pull-backs are defined as a bivariant class.

Finally we introduce some examples where semi-perfect obstruction theories appear. In \cite{CL11semi}, Chang and Li considered a moduli space $\fD^L_{S}$ which parametrizes $E \in D^b(S)$ where $S$ is a smooth projective Calabi-Yau threefold, such that $\det(E)\cong L$ for some fixed line bundle $L$ on $S$, $\Ext^{i<0}(E,E)=0$ and $\Ext^0(E,E) = k$.

In \cite{KLS17gen}, Kiem, Li and Savvas studied a generalized Donaldson-Thomas invariant, which counts 1-dimensional semistable sheaves on a Calabi-Yau threefold with some fixed Chern character $\gamma$. They first considered a moduli stack $\cM$ of Gieseker semistable sheaves on a Calabi-Yau threefold $X$ with a fixed Chern character $\gamma$. Since $\cM$ is a global quotient stack, the authors constructed a morphism $p : \wtil{\cM} \to \cM$ using Kirwan's partial desingularization \cite{Kir84coh}, which is isomorphic over the stable locus $\cM^s \subset \cM$. The authors showed that there is also a semi-perfect obstruction theory on $\wtil{\cM}$ and $\wtil{\cM}$ is a Deligne-Mumfords stack. Using this result, they defined a generalized Donaldson invariant :
\[ 
DT_{\gamma} := \deg[\wtil{\cM}]\virt .
\]

In \cite{Kie16locsemi} Kiem generalized a torus localization formula and cosection localization to the semi-perfect obstruction theory setting. Also, the author showed that dual obstruction cone has a semi-perfect obstruction theory. Dual obstruction cones are defined in the following; When $X$ is a DM stack with a perfect obstruction theory $\phi : E \to L_X$, let $Ob_{\phi} = h^1(E\dual)$. Then the dual obstruction cone $N_X$ of $X$ is defined by :
\[
N_X = \Spec_X(\mathrm{Sym} Ob_{\phi}) \to X.
\]

Using this result, the author proved that five definitions of virtual signed euler characteristic, $e_1(X),e_2(X),\dots,e_5(X)$ in \cite{JT14vseuler} are well-defined without assumptions on $X$ from derived geometry which was necessary in \cite{JT14vseuler} and $e_1(X)=e_2(X)$, $e_3(X)=e_4(X)=e_5(X)$ also hold as in \cite{JT14vseuler}.  

In \cite{KL12cat}, Kiem and Li proved that every critical virtual manifold has a symmetric semi-perfect obstruction theory. 
In \cite{Jia18sym}, Jiang proved that every algebraic d-critical scheme has a symmetric semi-perfect obstruction theory. We note that algebraic d-critical schemes are algebraic versions of critical virtual manifolds.

In \cite{Fen19gerb}, when there is a $\GG_m$-gerb(or gerb banded by a finite cyclic group) $p : \cG \to B$ over a DM stack $B$, and there is a perfect obstruction theory on the Artin stack $\cG$, Qu showed that there is an induced semi-perfect obstruction theory on the base stack $B$.

\medskip

\section{Definition and basic propereties of Chow groups of coherent sheaf stacks}\label{sec:def}
Throughout the paper, we follow the definition of algebraic stacks and Deligne-Mumford stacks in the book of M. Olsson \cite{Ols16stack}. Thus, a Deligne-Mumford stack is a category fibered in groupoids over the \'etale site in the category of $k$-schemes, with an \'etale surjection $\what{X} \to X$ from a scheme $\what{X}$. We only consider representable morphisms in this paper.

When we say that $\cF$ is a coherent sheaf on a Deligne-Mumford stack $X$, it means that $\cF$ is a $\cO_X$-module where $\cO_X$ is a structure sheaf of the topos $X$ on the big \'etale site of $X$. From now on, we usually abbreviate a Deligne-Mumford stack to a DM-stack.

In this section, we define a notion of Chow groups for coherent sheaf stacks, and define a basic morphism of Chow groups, i.e. proper pushforwards, flat-pullbacks, and Gysin maps. Then we prove basic properties of these morphisms.

\subsection{Cycle groups for coherent sheaf stacks}\label{sec:cycle}
Let $\cF$ be a coherent sheaf on DM stack $X$ and a coherent sheaf $\cF$ on $X$. Then, we can consider $\cF$ as a stack by the following.

\begin{defi}
We define a sheaf stack associated to an arbitrary coherent sheaf $\cG$ to be a category fibered in sets over a small \'etale cite. For any \'etale morphism $\vphi : S \to X$ from a scheme $S$, we assign a set $H^0(S,\vphi^*(\cG))$. This make sense since we can consider arbitrary set as a groupoid in a trivial way, and the category of sets is a full subcategory of the category of groupoids.
\end{defi}

We note that we use small \'etale cite instead of big \'etale cite and this is different from \cite{CL11semi}. We adopted small \'etale site because of some technical reasons. 


Since $\cF$ satisfies the sheaf condition on the big \'etale site on $X$ by \cite[4.3.3]{Ols16stack}, it is also satisfies the sheaf condition on the small \'etale cite on $X$. Therefore axioms of stack are all directly satisfied.

\begin{rema}\label{sheaf-stack}

From now on, when we consider a coherent sheaf stack $\cF$, we also write its associated sheaf stack $\cF$ by abuse of notation.

\end{rema}

\begin{prop}
Let $\cV \surra X$ be a vector bundle. Then, the coherent sheaf stack associated to $\cV$ is isomorphic to the total space $Tot(\cV)$ of $\cV$ as an $X$-scheme.
\end{prop}
\begin{proof}
This is clear since for any \'etale morphism morphism $f:S \to X$ from a scheme $S$, we have $\cV(S)=\Gamma(S,f^*\cV)=Hom_X(S,Tot(\cV))$.
\end{proof}

We note that a coherent sheaf stack $\cF$ is a stack but generally not an algebraic stack.





\begin{prop}\label{etale nullstellenstaz} Let $Z = Spec(R)$ be an affine integral schemes and $E \cong Z \times k^n$ be a coherent sheaf stack associated to a locally free sheaf $R^{\oplus n}$ on $Z$. Let $W \subset E$ be a reduced closed subscheme where each irreducible components dominates base $Z$. Let $\cE_W \subset E$ be a substack defined by
\[
\cE_W(U) = \{ s \in E(U) \ | \  s(U) \subset W\times_Z U \subset E\times_Z U \}
\]
for any \'etale morphism $U \to Z$. By abuse of notation, we also denote total space of $E$ by $E$. Let $E = Spec(R[x_1,\dots,x_n])$. We define
\[
I(\cE_W) := \{ f \in R[x_1,\dots,x_n] \ | \ f(s) = 0 \  \textrm{for all \'etale morphisms} \  U \to Z, s\in \cE_W(U) \}.
\]
Then we have $I(W) = I(\cE_W)$. 
\end{prop}
\begin{proof}
Let $W = \cup_i W_i$ where $W_i$ are integral closed subschemes of $|E|$. Then we have $\cE_W = \cup_i \cE_{W_i}$.
Since $I(W) = \cap_i I(W_i)$ and $I(\cup_i \cE_{W_i}) = \cap_i I(\cE_{W_i})$, it is enough to show the Proposition for the case when $W$ is integral.

We assume that $W$ is integral. Let $\pi : E \to Z$ be the projection then let $r$ be the dimension of the generic fiber of the projection $\pi : W \to Z$. Then we choose $r$ generic linear functions $L_1,\dots,L_r$ on $E = Z \times k^n$. Let $W_{a_1\dots,a_r} := \{L_1=a_1 \dots, L_r = a_r \} \cap W$. Then the projection $\pi : W_{a_1,\dots,a_r} \to Z$ is dominant and generically finite. Then we can choose open dense subset $U_{a_1,\dots,a_r} \subset W_{a_1,\dots,a_r}$ such that the restriction $\pi : U_{a_1,\dots,a_r} \to Z$ is \'etale. Then, there is a diagonal section $\Delta : U_{a_1,\dots,a_r} \to U_{a_1,\dots,a_r}\times_Z U_{a_1,\dots,a_r} \subset W \times_Z U_{a_1,\dots,a_r}$.

Since $\cup_{a_1,\dots,a_n \in k^n} U_{a_1\dots,a_n} \subset W$ is a dense subset of $W$, we conclude that $I(\cE_W) \subset I(W)$. Since it is trivial that $I(W) \subset I(\cE_W)$, we have $I(\cE_W) = I(W)$.
\end{proof}

Next, we define notions of a reduced closed substack, an integral substack and a cycle group of a coherent sheaf stack in a similar way as in \cite[Definition 2.2]{CL11semi}. From now on, for a substack $\iota : Z \hra X$ and a coherent sheaf $\cF$ on $X$, $\cF|_Z$ mean a stack correspond to a coherent sheaf $\iota^*\cF$.

\begin{defi}[closed substacks and reduced substacks integral substacks.]\label{closed}
Let $\cA \subset \cF|_Z$ be a substack where $Z\subset X$ is an integral closed substack of $X$. 
We call it closed if for any \'etale morphism $\vphi : U \to X$ from a scheme $U$ and any surjection from a locally free sheave $\cV \surra \cF|_{U}$, $\cA|_{U}\times_{\cF|_U} \cV$ is equal to $\cE_W$ for a closed subscheme $W \subset \cV$ By Proposition \ref{etale nullstellenstaz}, this $W$ is unique if it exists.

We call a closed substack $\cA \subset \cF|_Z$ reduced if for any \'etale morphism $\vphi : U \to X$ from a scheme $U$ and any surjection from a locally free sheave $\cV \surra \cF|_{U}$, $\cA|_{U}\times_{\cF|_U} \cV$ is equal to $\cE_W$ for a reduced closed subscheme $W \subset \cV$ where each irreducible component of $W$ dominates $Z$. 
\end{defi}

We note that above definition of reduced subtacks make sense since reducedness is preserved by smooth pull-backs and descents from flat surjective morphisms \cite[10.158.2, Lemma 10.157.7]{Sta}.


\begin{rema}\label{localchar}
Since closedness is a local property, we can observe that if there exist an \'etale open cover $\{U_i\}_{i\in \cI}$ of $X$ and there exist surjections $\cV_i \surra \cF|_{U_i}$ from locally free sheaves $\cV_i$s so that substacks $\cA|_{Z_i} \times_{\cF|_{Z_i}} \cV_i|_{Z_i} \subset \cV_i|_{Z_i}$s are represented by closed substacks of $\cV_i|_{Z_i}$, then $\cA \subset \cF|_Z$ is closed where $Z_i := Z \times_X U_i$. The proof of this statement is not so hard so we omit it here.
\end{rema}

\begin{defi}[Integral substacks, cycle groups]
\label{cycle}
For an arbitrary coherent sheaf stack $\cE$ on a DM stack $Y$, We note that for substacks $\cA,\cB\subset \cE$, their union $\cA\cup \cB\subset \cF$ is defined to be $(\cA\cup \cB)(U):=\cA(U)\cup \cB(U)\subset \cF(U)$ for any open subset $U \subset Y$.

Let $\cA \subset \cF|_Z$ be a reduced substack where $Z\subset X$ is integral closed substack of $X$. We call it irreducible if there is no nontrivial decomposition $\cA=\cA_1\cup \cA_2$ by closed substacks, and we call a substack $\cA \subset \cF|_Z$ integral if $\cA$ is reduced and irreducible.

Next, we define a cycle group $Z_*(\cF)$ of a coherent sheaf stack $\cF$ on $X$ to be the free abelian group generated by all integral substacks $\cA \subset \cF|_Z$. For an integral substack $\cA\subset \cF|_Z$, we write $[\cA]$ for an element $Z_*(\cF)$ correspond to $\cA$ by $[\cA]$ and we call this elements integral cycles. 
\end{defi}

\begin{defi}[Dimensions of integral cycles] For an integral substack $\cA \subset \cF|_Z$ for an integral closed substack $Z \subset X$, we define a dimension of $\cA$ by the following. By \cite[II,Chapter 5, ex 5.8]{Har13}, We can choose an open dense subset $U \subset Z$ such that the restriction $\cF|_Z$ is locally free. Then we define the dimension of the integral substack $\cA$ to be $\dim \cA := \dim \cA|_U$. 
\end{defi}

Therefore, there is a natural grading in $Z_*(\cF)$ by dimensions, so we can write $Z_*(\cF) = \bigoplus_{k\geq 0} Z_k(\cF)$ where $Z_k(\cF)$ is a free abelian group generated by $k$-dimensional integral substacks of $\cF|_Z$ for integral substacks $Z \subset X$. But we will not focus on dimension of cycles in this paper.

We call a cycle $[\cA]$ correspond to an integral substack $\cA \subset \cF|_Z$ an integral cycle. Note that when we choose an open dense subset $V \subset Z$ such that $\cF|_V$ is locally free, $\cA|_V \subset \cF|_V$ represented by an integral substack of $\cF|_V$, and $\bar{\pi(\cA|_V)} = V$ where $\pi : \cF|_V \to V$ is the projection. This follows from the fact that $\cA \subset \cF|_Z$ where $\cF|_Z$ is induced from the coherent sheaf $\cF|_Z$ on the small \'etale site on $Z$.

\subsection{Proper pushforwards}\label{sec:properpush}
\noindent We define proper pushforwards and flat pull-backs, which are morphisms in cycle groups of coherent sheaf stacks. The following definition of stack-theoretic closure is crucial to define these morphisms, and also crucial to many arguments described later.

\begin{defi}\label{stack-theoretic closure}
Let $Y$ be an integral DM stack. For any open dense subset $V\subset Y$ and any closed substack $\cB \subset \cG|_{V}$ of a coherent sheaf stack $\cG$ on Y, we define a prestack $\cB^{c}$ in the following way:

For any \'etale morphisms $T \stackrel{\phi}{\to} Y$, we know that $T\times_Y V \subset T$ is also open dense. Hence we define;
\[
\cB^c(\phi) :=  \{ s\in \cG(\phi) \ | \ s|_{T\times_Y V}\in \cB(T\times_Y V \stackrel{\phi|_{T\times_Y V}}{\lra} Y) \}. 
\]
Then we define the stack-theoretic closure $\bar{\cB}$ by the stackification of $\cB^c$.
\end{defi}


The following lemma says that taking a stack-theoretic closure preserves reducedness and irreducibility.

\begin{lemm}\label{closurepreserve1}
Let $\cF$ be a coherent sheaf stack on an integral DM-stack $X$ and $U$ be a open dense subset of $X$. Consider a closed substack $\cA \subset \cF|_U$. Then, $\bar{\cA}\subset \cF$ is also a closed substack. \end{lemm}
\begin{proof}
Consider an \'etale morphism $\vphi : T \to X$ and a surjection $\cV \surra \cF|_T$. Let $T' := T\times_X U \subset T$ be an open dense subset of $T$.
Since $\cA$ is closed and $\bar{\cA}|_{T'} = \cA$ by definition, we have $\cV|_{T'} \times_{\cF|_{T'} } \bar{\cA}|_{T'}$ is represented by a closed substack $Z \subset \cV|_{T'}$.

By definition of the stack-theoretic closure, we have the following. For an \'etale morphism $R \stackrel{\phi}{\lra} T$, $(\cV\times_{\cF|_T}\bar{\cA})(\phi)$ is a following collections;
$\{ \phi_{\alpha} : R_{\alpha} \to R, s_{\alpha}\in \Gamma(R_{\alpha}, \cV|_{R_{\alpha}}) \ | \ s_{\alpha}(R_{\alpha}\times_T T')\subset R_{\alpha}\times_T Z \}_{\alpha}$, where each $\phi_{\alpha}$ is an \'etale morphism, $\{\phi_{\alpha} : R_{\alpha} \to R \}_{\alpha}$ is a covering of $R$, and sections $\{s_{\alpha}\}_{\alpha}$ satisfies descend conditions and thus descend to a section $s \in \Gamma(R, \cV|_R)$.

We observe that $s_{\alpha}(R_{\alpha}\times_T T')\subset R_{\alpha}\times_T Z$ if and only if $s_{\alpha}(R_{\alpha}) \subset \bar{R_{\alpha}\times_T Z} = R_{\alpha} \times_T \bar{Z} $ where $\bar{Z}$ is a closure of $Z$ taken in $\cV$ and the equality holds because the morphism $R_{\alpha} \to T$ is \'etale and hence open. 

Then since $\{ R_{\alpha}\times_T \bar{Z} \}_{\alpha}$ descend to $R\times_T \bar{Z}$, we conclude that $s_{\alpha}(R_{\alpha}\times_T T')\subset R_{\alpha}\times_T Z$ for every $\alpha$ if and only if $s(R) \subset \bar{Z}$. Therefore we obtain that $(\cV\times_{\cF|_T}\bar{\cA})(R) = \{s \in \Gamma(R, \cV|_R) \ | \ s(R) \subset R \times_T \bar{Z}\}$. Thus, $\cV\times_{\cF|_T}\bar{\cA}$ is represented by a closed substack $\bar{Z} \subset \cV$, which means that $\bar{\cA} \subset \cF$ is also a closed substack.
\end{proof}

\begin{lemm}\label{closurepreserve2}
Let $\cF$ be a coherent sheaf stack on a DM-stack $X$ and $U$ be a open dense subset of $X$. Consider an irreducible substack $\cA \subset \cF|_U$. Then, $\bar{\cA}\subset \cF$ is also an irreducible substack. Furthermore, if $\cA$ is integral, then $\bar{\cA}$ is also integral.
\end{lemm}
\begin{proof}
Assume that $\bar{\cA}$ is not irreducible, i.e. there is a nontrivial decomposition $\bar{\cA}=\cB_1\cup \cB_2$ by nontrivial closed substacks $\cB_1$, $\cB_2$. If $(\cB_1)|_U \neq \varnothing$ and $(\cB_2)|_U \neq \varnothing$, then we have nontrivial decomposition $\cA = (\cB_1)|_U \cup (\cB_2)|_U$. This contradicts to the assumption that $\cA$ is irreducible. Therefore may assume that $(\cB_2)|_U = \varnothing$. For any \'etale morphism $T \to X$ and a section $s \in \cB_2(T)$, we have $s|_{T\times_X U} \in \cB_2(T\times_X U)$. Since $T \to X$ is \'etale, $T\otimes_X U \to U$ factors through dense open subset of $U$. So, it contradicts to the assumption that $(\cB_2)|_U = \varnothing$. Thus we have $\cB_2 = \varnothing$ and hence $\bar{\cA}$ is irreducible.

Next, we assume that $\cA$ is reduced. Since reducedness is a local property and $\bar{\cA}|_U=\cA$ is already reduced, it is enough to check that $\bar{\cA}|_{U'}$ is reduced when $U' \to X$ is an affine \'etale neighborhood around $X \setminus U$.
Since $U'$ is affine, we can choose a surjection from a vector bundle $\cV \surra \cF|_{U'}$. Let $Z':=\bar{\cA}|_{U'} \times_{\cF|_{U'}} \cV$. Then, $Z'|_U=\cA|_{U'\times_X U} \times_{\cF|_{U'\times_X U}}\cV|_{U'\times_X U}$ is reduced since $\cA$ is reduced. On the other hand, we can easily observe that $Z'=\bar{Z'|_U}$. Therefore, $Z$ is also reduced by \cite[Lemma 28.7.9]{Sta}. Hence $\bar{\cA}$ is reduced. 
\end{proof}

Now we define proper pushforwards. 
Let $\cE$ be a coherent sheaf stack on $Y$. Consider the following fiber diagram in the category of stacks where $X$ and $Y$ are DM stacks :
\[
\xymatrix{\cE|_X \ar[r]^f \ar[d]^{\pi} \ar@{}[rd]|{\square} & \cE \ar[d]^{\pi} \\
X \ar[r]^f & Y.}
\]
We assume that $f$ is a proper morphism. Then, we define a proper pushforward $f_*$ : $Z_*(\cE|_{X}) \to Z_*(\cE)$ by the following.
Consider an integral cycle $[\cA]$ of $Z_*(\cE|_X)$ for an integral substack $\cA \subset \cE|_Z$ for an integral substack $Z \subset X$. Note that $f(Z) \subset Y$ is also an integral closed substack of $Y$ and we can choose an open dense subset $U \subset f(Z)$ such that $\cE|_U$ is locally free. Then, $\cE|_{f^*(U)} \to \cE|_U$ is a proper morphism of algebraic stacks. We note that since $\cA$ is integral, $\cA|_{f^*(U)}$ and its image $f(\cA|_{f^*(U)})$ are also integral stacks.

The flowing definitions are sheaf-theoretic analogues of  a scheme-theoretic image and a degree of a morphism.

\begin{defi}[Stack-theoretic images, Stack-theoretic degrees]
Let $\cA\subset \cE|_Z$ be an integral substack as above. Then, we define a stack-theoretic image of $\cA$ via $f$ by:
\[
f(\cA) := \bar{f(\cA|_{f^*(U)})} \subset \cE|_{f(Z)}
,\]
which is again integral by Lemma \ref{closurepreserve2}.

We define a stack-theoretic degree of the morphism $f|_{\cA} : \cA \to f(\cA)$, denoted by $\deg(f|_{\cA})$ to be :
\[\deg(f|_{\cA}):=\deg(\cA|_{f^*(U)} \to f(\cA|_{f^*(U)})).
\]
\end{defi}

\begin{defi}[Proper pushforwards]\label{pushforward}
We define a proper pushforward $f_* : Z_k(\cE|_{X}) \to Z_k(\cE)$ to be :
\[
f_*[\cA] := \deg(f|_{\cA}) \cdot [f(\cA)]
\]
where $\cA\subset \cE|_Z$ is an integral substack, $Z \subset X$ is a closed integral substack.
\end{defi}

We note that that it is trivial that the definition does not depend on the choice of $U \subset f(X_{\cA})$.

\begin{prop}\label{pushforwardcomp}
Consider proper morphisms $\xymatrix@=15pt{X \ar[r]^f & Y \ar[r]^g & Z}$ between DM stacks and a coherent sheaf stack $\cE$ on Z. Then, we have :
\[
g_* \cdot f_* = (g \cdot f)_* : Z_k(\cE|_X) \to Z_k(\cE). 
\]
\end{prop}
\begin{proof}
Let $\cA \subset \cE|_W$ be an integral substack and $W$ is a closed substack of $X$. Choose an open subset $U \subset (g\cdot f)(W)$ such that $\cE|_U$ is locally free. Then, $f_*[\cA]=\deg(\cA|_{(g\cdot f)^*(U)} \to f(\cA|_{(g\cdot f)^*(U)}))[\bar{f(\cA|_{(g\cdot f)^*(U)})}]$. Thus, we have the following :
\begin{align*}
&g_*(f_*[\cA]) \\
&= \deg(\cA|_{(g\cdot f)^*(U)} \to f(\cA|_{(g\cdot f)^*(U)}))\cdot \deg(f(\cA|_{(g\cdot f)^*(U)}) \to (g\cdot f)(\cA|_{(g\cdot f)^*(U)})) \\ \nonumber
& [\bar{(g\cdot f)(\cA|_{(g\cdot f)^*(U)})} \\ \nonumber
&= \deg(\cA|_{(g\cdot f)^*(U)} \to (g\cdot f)(\cA|_{(g\cdot f)^*(U)}))[\bar{(g\cdot f)(\cA|_{(g\cdot f)^*(U)})}] \\
&= (g\cdot f)_*[\cA]. \nonumber
\end{align*}

\end{proof}

By the proof of the above proposition, we have the following result directly.

\begin{lemm}
On the above setting, we have :
\[ (g\cdot f)(\cA)=g(f(\cA)) \]
and also we obtain :
\[ 
\deg(f|_{\cA})\cdot \deg(g|_{f(\cA)})=\deg((g\cdot f)|_{\cA}).
\]
\end{lemm}

\subsection{Saturated cycles and Projection morphism}
\label{sec:satandproj}
\begin{defi}[Saturated integral subschemes(substacks)]\label{satsubsch}
Let $\cF$ be a coherent sheaf on $X$, and let $f$ be a proper morphism $f:S\to X$ from a scheme(or DM stack) $S$, and let $q : \cV \surra \cF|_S$ be a surjection from a vector bundle $\cV$ on $S$. Consider an open cover $\{ S_{\alpha} \}_{\alpha \in \cI}$ of $S$. Then, we have the following locally free resolution :
\[\xymatrix@=15pt{\cV_{\alpha}' \ar[r]^(0.3){\vphi_{\alpha}} & \cV|_{S_{\alpha}} \ar@{->>}[r]^(0.5){\phi} & \cF|_{S_{\alpha}} \ar[r] & 0
}\]
for each $\alpha \in \cI$. Let $\FF_{\alpha} := [\cV|_{S_{\alpha}} / \cV_{\alpha}']$. 

Let $W \subset \cV|_{S_Z}$ be an integral subscheme where $S_Z$ is a closed subscheme(or closed substack) of $S$ and $W_{\alpha} := W \times_S S_{\alpha}$. We call $W$ is saturated there exist a such open cover $\{S_{\alpha}\}_{\alpha \in \cI}$ and locally free resolutions and corresponding bundle stacks $\FF_{\alpha}$s such that for every $\alpha$, there exists an integral substack $\WW_{\alpha} \subset \FF_{\alpha}$, such that $W_{\alpha} = \cV|_{S_{\alpha}}\times_{\FF_{\alpha}} \WW_{\alpha}$.

We call a reduced subscheme(or substack) $W \subset \cV|_S$ is saturated if it is a union of saturated induced subschemes(or substack).
\end{defi}

\begin{defi}[Saturated cycle groups]\label{satcycle} For the surjection $q : \cV \surra \cF|_S$, we define a saturated cycle group $Z^{sat}_*(\cV)$ to be a free abelian group generated by saturated integral subschemes(or substacks) $W \subset \cV|_{S_Z}$ where $S_Z \subset S$ is an integral closed subscheme(or substack) of $S$. 
\end{defi}

The following lemma says that proper pushforwards and flat pull-backs are well-defined over invariant cycle groups. The proof is straightforward so we omit it here.

\begin{lemm}\label{satschlemm}
Let $\cF$ be a coherent sheaf on $X$ and let $f$ be a proper morphism $f : S \to X$ from an integral scheme(or DM stack) $S$, and let $q : \cV \surra \cF|_S$ be a surjection from a vector bundle $\cV$ on $S$. 

Let $W \subset \cV$ be an integral subscheme(or substack) and $S_W = \bar{\pi(W)}$ where $\pi : \cV \to S$ is the projection.
Then $W$ is saturated if and only if there is open dense subset $U \subset S$ such that there is a locally free resolution;
\[
\cV' \stackrel{\vphi}{\lra} \cV|_U \surra \cF|_U \to 0
\]
and there is an integral substack $\WW \subset \FF := [\cV|_U/ \cV']$ such that $W = \bar{\cV|_U \times_{\FF} \WW}$ where the closure is taken in $\cV|_{S_W}$. 
\end{lemm}
\begin{proof}
Necessary condition is trivial, so we only proof sufficient condition. Consider an affine open cover $\{S_{\alpha}\}_{\alpha\in \cI}$ of $S$. For $S_{\alpha}$ such that $S_{\alpha} \cap S_{W} = \varnothing$, we have $W \cap \cV|_{S_{\alpha}} = \varnothing$ therefore we do not need to consider these $S_{\alpha}$s. Thus we only consider $S_{\alpha}$ such that which has nonzero intersection with $S_W$. Let $S_W \cap S_{\alpha} := S_{W_{\alpha}}$.

By construction we know that $U$ intersect with $S_W$ therefore $S_{\alpha}$s meets $U$. Let $S_{\alpha}' := S_{\alpha}\cap U$.

On the other hand, choose an open dense subset $V \subset S_W$ such that $\cF|_V$ is locally free. Let $V' := V\cap U, V_{\alpha} := V\cap S_{\alpha}, V_{\alpha}' := V\cap S_{\alpha}'$. Let $q' : \cV|_V \surra \cF|_V$ be the restriction of the morphism $q$. Then since $W|_{V'_{\alpha}} = \cV|_{V'_{\alpha}} \times_{\FF|_{V'_{\alpha}}} \WW|_{V'_{\alpha}}$, $W$ is closed under $\cV_{\alpha}'$-action induced by the morphism $\vphi_{\alpha}$. Hence we observe that $W|_{V'_{\alpha}} = \cV|_{V'_{\alpha}} \times_{\cF|_{V'_{\alpha}}} q'(W|_V)|_{V'_{\alpha}}$. Note that since $W|_V$ is irreducible and $\cV|_V \surra \cF|_V$ is a surjection of vector bundle, $q'(W|_V)$ is closed in $\cF|_V$ and hence integral.
Let $\WW_{\alpha}' \subset \FF_{\alpha}|_{V'_{\alpha}}$ be an integral substack which is an image of $W|_{V'_{\alpha}}$. Then we have $W|_{V'_{\alpha}} = \WW_{\alpha}' \times_{\FF_{\alpha}|_{V'_{\alpha}}} \cV|_{V'_{\alpha}}$ by construction. Let $\WW_{\alpha} := \bar{\WW_{\alpha}'} \subset \FF|_{S_{W_{\alpha}}}$. Since $\FF|_{S_{W_{\alpha}}}$ is a closed substack of $\cF|_{S_{\alpha}}$, $\WW_{\alpha}$ is a closed substack of $\FF|_{S_{\alpha}}$. Note that $\WW'_{\alpha}$ is an open dense subset of $\WW_{\alpha}$. We consider the morphism $\cV|_{S_{\alpha}} \to \FF_{\alpha}$ as a vector bundle and therefore $\WW_{\alpha}\times_{\FF_{\alpha}} \cV|_{S_{\alpha}}$ is also integral, whose open dense subset is equal to $W|_{V_{\alpha}'}$. Thus we obtain $\WW_{\alpha}\times_{\FF_{\alpha}} \cV|_{S_{\alpha}} = W_{\alpha}$.  
\end{proof}

\begin{lemm}\label{satschlemm2}
Let $\cF$ be a coherent sheaf on $X$ and let $f$ be a proper morphism $f : S \to X$ from an integral scheme(or DM stack) $S$, and let $q : \cV \surra \cF|_S$ be a surjection from a vector bundle $\cV$ on $S$. 

Let $W \subset \cV$ be an integral subscheme(or substack) of $\cV$ and $S_W = \bar{\pi(W)}$ where $\pi : \cV \to S$ is the projection. Then $W$ is saturated if and only if there is an open dense subset $V \subset S_W$ such that $\cF|_V$ is locally free and there is an integral subscheme(or substack) $B\subset \cF|_V$ such that $W = \cV|_{V} \times_{\cF|_V} B$.
\end{lemm}
\begin{proof}
It is straightforward using Lemma \ref{satschlemm}.
\end{proof}

\begin{lemm}
\begin{itemize}
\item[(i)] Consider the following diagram :
\[\xymatrix{\cV|_{S_2} \stackrel{q}{\surra} \cF|_{S_2} \ar[rr]^v \ar[d]^{\pi} & & \cV \stackrel{q}{\surra} \cF|_{S_1} \ar[d]^{\pi} & \\
S_2 \ar[rr]^v \ar[ddr]^g & & S_1 \ar[ddr]^f & \\
& \cF|_X \ar[d] & & \cF \ar[d] \\
& X \ar[rr]^u & & Y}
\]
where $X,Y$ are DM stacks, $u,v$ are proper morphisms, $S_1$ and $S_2$ are schemes and $\cF$ is a coherent sheaf stack on $Y$ and $q : \cV \surra \cF|_{S_1}$ is a surjection from a vector bundle $\cV$.

Let $W \subset \cV|_{S_2}$ be a saturated integral subschemes of $\cV|_{S_2}$. Using Lemma \ref{satschlemm} we can easily check that $v(W)$ is also an integral saturated subschemes of $\cV$.
Hence we can define pushforward $v_* : Z^{sat}_*(\cV|_{S_2}) \to Z^{sat}_*(\cV)$ as a restriction of the ordinary proper pushforward $v_* : Z_*(\cV|_{S_2}) \to Z_*(\cV)$.

\item[(ii)] Consider the following diagram :
\[ \xymatrix{\cV|_{S\times_Y X} \stackrel{q}{\surra} \cF|_{S\times_Y X} \ar[dd]_{\pi} \ar[rr]^v & & \cV \stackrel{q}{\surra}\cF|_S \ar[dd]^(0.3){\pi} & \\
& \cF|_{X} \ar[dd] & & \cF \ar[dd] \\
S\times_Y X \ar[rr]^(0.7)v \ar[dr]_g & & S \ar[dr]^f & \\
& X \ar[rr]^u & & Y
}
\]
where $X,Y$ are DM stacks, $S$ is an integral scheme, $u,v$ are flat morphisms, and $q : \cV \surra \cF|_S$ is a surjection from a vector bundle $\cV$.

Let $W \subset \cV$ be an integral saturated subscheme of $\cV$. Let $S_W := \bar{\pi(W)}$. Then since $W$ is saturated and by Lemma \ref{satschlemm2}, there is an open subset $V \subset S_W$ such that $\cF|_V$ is locally free and an integral subscheme $B \subset \cF|_V$ such that $W = \bar{\cV|_V \times_{\cF|_V} B }$ where the closure is taken in $\cV|_{S_W}$.

Then $v^*[W]$ is defined by the following. Let $V' := (S\times_Y X) \times_S V$. Then $\cF|_{V'} = v^*(\cF|_V)$ is locally free and let $v^*[B] = \sum_i a_i [B_i]$. Let $S_{W_i} := \bar{\pi(W_i)}$ and $V_i':= V'\cap S_{W_i}$. Let $W_i := \bar{\cV|_{V_i'}\times_{\cF|_{V_i'}} B_i }$ where the closure is taken in $\cV|_{S_{W_i}}$. 

Note that $W_i$ are saturated integral sub-algebraic spaces of $\cV|_{S_{W_i}}$. Then we define $v^*[W] = \sum_i a_i [W_i] \in Z_*^{sat}(\cV|_{S \times_Y X})$

\end{itemize}
\end{lemm}

\begin{defi}[Descent]\label{descent map}
We define a descent map $q_* : Z^{sat}_*(\cV) \to Z_*(\cF|_S)$ to be the following. For a saturated integral subscheme(or substack) $D \subset \cV$, let $S_D := \bar{\pi(D)}$ where $\pi : \cV \to S$ is a projection. Since $D$ is integral, $S_D$ is also integral. Choose an open dense subset $U\subset S_D$ such that $\cF|_U$ is locally free. We define $q(D) := \bar{q(D|_U)} \subset \cF|_{S_{D}}$ where the closure is taken in $\cF|_{S_{D}}$, which is again integral by Lemma \ref{closurepreserve2} and call it an image of $D$ via $q$.
Then, we define $q_*$ to be :
\[ q_*[D] := [q(D)] \subset Z_*(\cF|_S).
\]
\end{defi}

Then next lemma says that there is a 1-1 correspondence between an integral invarinat cycles and integral substacks.

\begin{lemm}\label{pullbacktosat}
Consider the setting of the above lemma.
Let $W \subset \cV$ be an saturated integral subscheme(or substack), $q : \cV \surra \cF$ be the projection. Let $S_W := \bar{\pi(W)}$ and $q(W) \subset \cF|_{S_W}$. Choose an open dense subset $V \subset S_W$ such that $\cF|_V$ is locally free.
Then we have $\bar{ q(W)|_V \times_{\cF|_V} \cV|_V } = W$.
\end{lemm}
\begin{proof}
Since $W$ is saturated integral subscheme(or substack), there exist an open dense subset $U \subset S$ such that there is a locally free resolution $\cV' \to \cV|_{U} \surra \cF|_U \to 0$, and an integral substack $\WW \subset \FF:=[\cV|_U / \cV']$ such that $W = \bar{\cV|_U \times_{\FF} \WW }$. Then $W|_U$ is closed under $\cV'$-action. Let $V' := V\cap U$. Consider a morphism $q' : \cV|_{V'} \surra \cF|_{V'}$, which is a surjection of vector bundles. Then we observe that $q'^*q'(W|_{V'}) = W|_{V'}$ since $W|_{V'}$ is closed under $\cV'$-action and $q'$ is a surjection of vector bundles. Since $W$ is integral, $W|_V' \subset W$ is open dense. Therefore we have
\[
\bar{q(W)|_V \times_{\cF|_V} \cV|_V} = \bar{q(W)|_{V'} \times_{\cF|_{V'}} \cV|_{V'}} = \bar{q'^*q'(W|_{V'})} = \bar{W|_{V'}} = W.
\]
\end{proof}

The next lemma shows that the descent map commutes with proper pushforwards. Consider the following diagram :
\[\xymatrix{\cV|_{S_2} \stackrel{q}{\surra} \cF|_{S_2} \ar[rr]^v \ar[d]^{\pi} & & \cV \stackrel{q}{\surra} \cF|_{S_1} \ar[d]^{\pi} & \\
S_2 \ar@{}[rrrdd]|{\commra} \ar[rr]^v \ar[ddr]^g & & S_1 \ar[ddr]^f & \\
& \cF|_X \ar[d] & & \cF \ar[d] \\
& X \ar[rr]^u & & Y}
\]
where $X,Y$ are DM stacks, $u,v$ are proper morphisms, $S_1$ and $S_2$ are integral schemes and $\cF$ is a coherent sheaf stack on $Y$, $q : \cV \surra \cF|_{S_1}$ is a surjection from a vector bundle $\cV$. Then, we have the following.

\begin{lemm}\label{descentpushforward}
On the above setting, we have :
\[q_*\circ v_* = v_* \circ q_* : Z^{sat}_* (\cV|_{S_2}) \to Z_* (\cF|_{S_1}).
\]
\end{lemm}
\begin{proof}
Let $W \subset \cV|_{S_2}$ be a saturated integral subscheme of $\cV|_{S_2}$. Let $S_{2,W} := \bar{\pi(W)}$, and let $S_{1,W}:=v(S_{2,W})$, which are all integral schemes. Since $W$ is saturated, there is an open dense subset $U \subset S_{2,W}$ such that $\cF|_U$ is locally free and an integral subscheme $B \subset \cF|_U$ such that $W = \bar{\cV|_U\times_{\cF|_U} B}$. 
On the other hand, we can choose an open dense subset $V \subset S_{1,W}$ such that $\cF|_V$ is locally free and let $U' = U \cap v^*(V)$.

Then we have $q_*[W] = \bar{B}$, $v_*(q_*[W]) = \deg(q(W|_{U'}) \to v(q(W|_{U'})))\cdot [\bar{v(q(W|_{U'}))}] $. Consider the following commutative fiber diagram :
\[\xymatrix{W|_{U'} \ar@{}[rd]|{\square} \ar[r]^(0.45)v \ar[d]^q & v(W|_{U'}) \ar[d]^q \\
 q(W|_{U'}) \ar[r]^(0.45)v & v(q(W|_{U'}))
}
\]
Since $q : \cV|_U \surra \cF|_U$ is a surjection of a vector bundle and $W$ is saturated, we can observe that vertical arrows in the above diagram are projection of vector bundles. Therefore, we have $\deg(W|_{U'} \to v(W|_{U'})) = \deg(q(W|_{U'}) \to v(q(W|_{U'})))$. Therefore, we obtain $v_* q_*[W] = \deg(W|_{U'} \to v(W|_{U'}))\cdot $ $[\bar{v(q(W|_{v^*(U)}))}] = \deg(W \to v(W))\cdot [q(v(W))]$.
On the other hand, we have $q_*(v_*[W]) = (\deg(W \to v(W))\cdot q_*[v(W)]) = \deg(W \to v(W))\cdot [q(v(W))].$  
\end{proof}

Now, we can easily re-describe a notion of projections which appeared in \cite[Definition 3.3]{CL11semi}.

\begin{defi}[Projection]\label{projection}
Let $\cF$ be a coherent sheaf stack on $X$. Let $f : S \to X$ be a proper morphism from a scheme(or DM stack) $S$, $q : \cV \surra \cF|_S$ be a surjection from the locally free sheaf $\cV$, and $W \subset \cV$ be a saturated invariant closed subscheme(or substack) of $\cV$. Let $S_W := \bar{\pi(W)}$ where $\pi : \cV \to S$ is the projection. Then we define a substack $\zeta_S(W)(\textrm{or sometimes we write it by }$ $\zeta_{f}(W)\subset \cF|_{f(S_W)})$ to be $\zeta_S(W) := f(q(W))$.
\end{defi}

\noindent Next, we introduce a modification of the notion of projection, which is defined more functorially.

\begin{defi}[Modified projection]
We define a modified projection 
\[(\zeta_S)_*(\textrm{or sometimes we write it by }(\zeta_{f})_* ) := f_* \circ q_* : Z^{sat}(\cV) \to Z_*(\cF) \] 
\end{defi}

\noindent Then, the following lemma is straightforward.

\begin{lemm}
For saturated invariant subscheme(or substack) $W \subset \cV$, we have :
\[ (\zeta_S)_* [W] = d_W \cdot \zeta_S(W).
\]
where $d_W:=\deg(q(W|_{f^*(U)}) \to f(q(W|_{f^*(U)})))$.
\end{lemm}

Next, we can show that the modified projection and proper pushforwards are compatible. Consider the following fiber diagram :
\[\xymatrix{\cV|_{S_2} \stackrel{q}{\surra} \cF|_{S_2} \ar[rr]^v \ar[d]^{\pi} & & \cV \stackrel{q}{\surra} \cF|_{S_1} \ar[d]^{\pi} & \\
S_2 \ar[rr]^v \ar[ddr]^g & & S_1 \ar[ddr]^f & \\
& \cF|_X \ar[d] & & \cF \ar[d] \\
& X \ar[rr]^u & & Y}
\]
where $f,g,u,v$ are proper morphisms, $S_1$,$S_2$ are integral schemes, $\cF$ is a coherent sheaf stack on $Y$, $q : \cV \surra \cF|_{S_1}$ is a surjection from a vector bundle $\cV$. Then, we have the following.

\begin{lemm}\label{mprojectionpushforward} On the above setting, we have :
\[ u_* \circ (\zeta_{S_2})_* = (\zeta_{S_1})_* \circ v_* : Z^{sat}_*(\cV|_{S_2}) \to Z_*(\cF). \]
\end{lemm}
\begin{proof}
We have 
$u_* \circ (\zeta_{S_2})_* = u_*\circ g_* \circ q_* = f_* \circ v_* \circ q_* = f_* \circ q_* \circ v_* = (\zeta_{S_1})_* \circ v_*$ by Lemma \ref{descentpushforward}.
\end{proof}



\subsection{Proper representatives and Modified proper representatives}\label{sec:proper+modifiedrep}
In this section, we review a notion of a proper representative of an integral substack $\cA \subset \cF|_Z$ where $Z$ is an integral substack of $X$ and $\cF$ is a coherent sheaf stack on $X$. This notion comes from \cite{CL11semi} and the authors defined a Gysin map via a coherent sheaf stack $\cF$ using this proper representative. We will deal with Gysin maps later. Moreover, we will introduce a notion of Modified proper representatives for some technical reasons.

\begin{defi}[Proper representatives]\cite[p. 818 (3.6)]{CL11semi}\label{properrep}
Let $\cA \subset \cF|_Z$ be an integral substack where $Z\subset X$ is an integral closed substack. Choose a proper, generically finite surjection $f:S\to Z$ such that $S$ is projective and integral scheme, and choose a surjection from a vector bundle $q : \cV \surra \cF|_S$. A choice of such $f$ exists by \cite{Ols05} and \cite[Lemma 1.10]{Vis89}. Since $f$ is generically finite, we can choose an open dense subset $S' \subset S$ such that $f|_{S'} : S' \to Z$ is \'etale. Then, let $\what{\cA}:=\bar{\cV|_{S'}\times_{\cF|_{S'}} \cA|_{S'}}$. We call the triple $(f:S\to Z, q : \cV \surra \cF|_S, \what{\cA})$ a proper representative of the integral substack $\cA \subset \cF|_{Z}$.
\end{defi}

Note that $\what{\cA}$ is reduced by the definition but may not be integral since an \'etale fiber product does not preserves irreducibility.

\begin{prop}\label{properrepsat}
In Definition \ref{properrep}, the reduced subscheme $\what{\cA} \subset \cV$ is a union of saturated integral subschemes. 
\end{prop}
\begin{proof}
Let $V_0 \subset Z$ be an open dense subset such that $\cF|_{V_0}$ is locally free. Let $V = V_0 \cap f(S')$. Then $\cA|_V \subset \cF|_V$ is represented by a closed substack of $\cF|_V$. Then $f : \cA|_{f^*(V)} \to \cA|_V$ is an \'etale and generically finite morphism. Let $\cA|_{f^*(V)} = \bigcup_i B_i$ where $B_i \subset \cV|_{f^*(V)}$ is an integral subscheme. Note that since $f$ is \'etale and surjective, we have $f(B_i) = \cA|_V$ for every $i$. we define $\what{\cA}^{ref} := \bar{ \cV|_{f^*(V)}\times_{\cF|_{f^*(V)}} \cA|_{f^*(V)} }$. 
By the irreducible decomposition $\cA|_{f^*(V)} = \bigcup_i B_i$, when we let $W_i := \bar{\cV|_V \times_{\cF|_V} B_i}$, we have $\what{\cA}^{ref} = \bigcup_i W_i, [\what{\cA}^{ref}] = \sum_i [W_i]$.  Then $W_i$ are saturated integral subschemes.

By construction, we have $\what{\cA}^{ref} \subset \what{\cA}$. Let $W \subset \what{\cA}$ be an irreducible component which is not contained in $\cA^{ref}$.


But since $\cA \subset \cF|_Z$ is integral substack, we represent $\cV|_{S'} \times_{\cF|_{S'}} \cA|_{S'}$ by a reduced subscheme in $\cV|_{S'}$ whose every irreducible component dominates $S'$ by Definition \ref{cycle}. Hence $W$ dominates $S$, i.e. $\bar{\pi(W)}=S$ where $\pi : \cV \to S$ is the projection.


Then $W|_{f^*(V)} \subset \cV|_{f^*(V)} \times_{\cF|_{f^*(V)}} \cA|_{f^*(V)}$ where $W|_{f^*(V)} \subset W$ is an open dense subset of $W$. Thus $W$ also contained in $\what{\cA}^{ref}$. Thus we have $\what{\cA}^{ref} = \what{\cA}$.
\end{proof}

\begin{rema}\label{properrepcomp}
By the proof of Proposition \ref{properrepsat}, we have $\zeta_S(W_i) = \cA$ for every $i$. 
\end{rema}

\begin{defi}[Modified proper representatives for integral substacks]
For an integral substack $\cA \subset \cF|_Z$ where $Z \subset X$ is an integral closed substack, a modified proper representative is a triple 
\[(f: S \surra Z, q : \cV \surra \cF|_S, W_{\cA}\subset \cV)
\]
such that $S$ is a DM stack, $f$ is a proper surjective morphism, $q$ is a surjection from a vector bundle $\cV$ on $S$ and $W_{\cA}$ is a saturated integral stack of $\cV$ where $q(W_{\cA}) \to \cA$ is generically finite with nonzero degree.

The statement that $q(W_{\cA}) \to \cA$ is generically finite means that if we choose an open dense subset $V \subset Z$ where $\cF|_V$ is locally free, then the morphism $q(W_{\cA})|_{f^*(V)} \to \cA|_V$, which is a morphism of DM stacks, is generically finite. 
\end{defi}

\begin{rema}
We emphasize that $S$ need to not to be a scheme. Also, $f: S \to Z$ need not to be generically finite, which is a mainly different point from a notion of proper representatives. Moreover, since $W_{\cA}$ is integral, it makes many computations simpler.
\end{rema}

\begin{rema}\label{modifiedproperexist}
In Definition \ref{properrep} let $\what{\cA} = \bigcup_i W_i$. Then, by Proposition \ref{properrepsat} and Remark \ref{properrepcomp}, we have that the triple $(f : S \to Z, q : \cV \surra \cF|_S, W_i \subset \cV)$ is a modified proper representative of $\cA \subset \cF|_Z$. Therefore, modified proper representative always exists for an integral substack $\cA \subset \cF|_Z$. 
\end{rema}

\subsection{Flat pull-backs}\label{sec:flatpullback}
Using modified proper representative, we define flat pull-back of the cycle groups for coherent sheaf stacks.

\noindent Next, we define flat pull-backs. Consider the following fiber diagram again :
\[
\xymatrix{\cE|_X \ar[r]^u \ar[d]^{\pi} \ar@{}[rd]|{\square} & \cE \ar[d]^{\pi} \\
X \ar[r]^u & Y.}
\]
Assume that $u$ is a flat morphism of relative dimension $n$. Then, we want to define a flat pull-back $u^* : Z_*(\cE) \to Z_*(\cE|_X)$. In a technical reason, we define a pull-back in $\QQ$-coefficient instead, i.e. $u^* : Z_*(\cE)_{\QQ}:=Z_*(\cE)\otimes_{\ZZ} \QQ \to Z_*(\cE|_X)_{\QQ}$

\begin{lemm}\label{component}
Let $A,B,C$ be integral DM stacks and let $f:B\to A$, $g:C\to A$ be generically finite surjective morphisms. Then, there exists an irreducible component $W$ of $B\times_A C$ such that induced projections $W\to B$ and $Z \to C$ are generically finite and surjective. 
\end{lemm}
\begin{proof}
Since $f$ and $g$ are generically finite, we can choose a dense open subset $U\subset A$ such that $f|_{f^*(U)}$ and $g|_{g^*(U)}$ are \'etale surjective. Then, for any irreducible component $W'$ of $f^*(U)\times_U g^*(U)$, induced projections $W' \to f^*(U)$ and $W' \to g^*(U)$ are finite and surjective. Let $W = \bar{W'}$ be a closure in $B \times _A C$. Then, induced projections $W \to B$ and $W \to C$ are generically finite and surjective.
\end{proof}

\begin{defi}[Flat pull-backs]\label{pullback}
For an integral substack $\cA \subset \cE|_Z$, choose a modified proper representative $f : S \to Z, q : \cV \surra  \cE|_Z, W_{\cA} \subset \cV$. Let $u^*[W_{\cA}] = \sum_i a_i[W_i] \in Z_*[\cV|_{S\times_Y X}]$. Let $S_i = \bar{\pi(W_i)}$ where $\pi : \cV|_{S\times_Y X} \to S\times_Y X$ is the projection. Let $v : S\times_Y X \to S$, $g : S\times_Y X \to X$ be induced morphisms. We define $Z_i : = g(S_i)$. Then we define $\cA_i := \zeta_{S_i}(W_i) \subset \cF|_{Z_i}$. We define the flat pull-back by
\[
u^*[\cA]:= \frac{1}{d_{W_{\cA}}}( \sum_i a_i d_{W_i}[\cA_i] ) = \frac{1}{d_{W_{\cA}}}(\zeta_{S\times_Y X})_*(u^*[W_{\cA}]).
\]

\end{defi}

\begin{rema}\label{pullbackcompat1}
By the definition of the pull-back, we directly obtain;
\[
u^*(\zeta_{S})_*[W_{\cA}] = (\zeta_{S\times_Y X})_*(u^*[W_{\cA}]).
\]

Thus we can observe that
\[
u^* \circ (\zeta_S)_* = (\zeta_{S\times_Y X})_* \circ u^*
\]
here both sides are morphisms from $Z_*^{sat}(\cV)$ to $Z_*(\cF|_X)$. Since descent morphism $q_*$ defined in Definition \ref{descent map} can be considered as a special case of the modified projection and since we can easily observe that flat pull-back of the saturated integral substack is a linear sum of saturated integral substacks from the definition of saturated integral substacks, flat pull-back also commutes with descent morphisms.
\end{rema}

\begin{lemm}\label{pullbackwelldef}
The above definition of flat pull-backs does not depend on the choice of modified proper representative of $\cA \subset \cE|_Z$.
\end{lemm}
\begin{proof}
Consider two different modified proper representatives of $\cA$;
\[
(f_1 : S_1 \to Z, q_1 : \cV_1 \surra \cF|_{S_1}, W_{\cA,1}), \ (f_2 : S_2 \to Z, q_2 : \cV_2 \surra \cF|_{S_2}, W_{\cA,2}).
\]
Consider an open subset $U\subset Z$ such that $\cF|_U$ is locally free. Since $q_1(W_{\cA,1})|_{f_1^*(U)} \to \cA|_U$, $q_2(W_{\cA,2})|_{f_2^*(U)} \to \cA|_U$ are generically finite and surjective, by Lemma \ref{component}, there exist an irreducible component $B$ of $q_1(W_{\cA,1})|_{f^*(U)} \times_{\cA|_U} q_2(W_{\cA,2})|_{f_2^*(U)}$ such that induced projections $W \to q_1(W_{\cA,1})|_{f_1^*(U)}$ and $Z \to q_2(W_{\cA,2})|_{f_2^*(U)}$ are both surjective and generically finite.
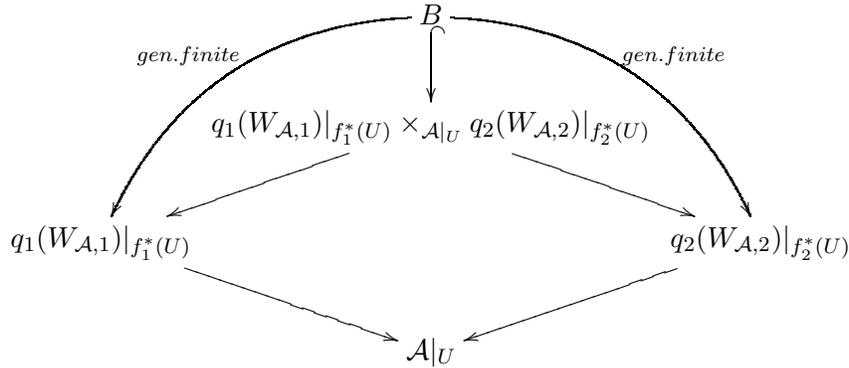
\[\xymatrix@C=1pt{& B \ar@{^(->}[d] \ar@/^25pt/[ddr]^{gen.finite} \ar@/_25pt/[ddl]_{gen.finite} & \\
& q_1(W_{\cA,1})|_{f_1^*(U)}\times_{\cA|_U} q_2(W_{\cA,2})|_{f_2^*(U)} \ar[dr] \ar[dl] \\
q_1(W_{\cA,1})|_{f_1^*(U)} \ar[dr] & & q_2(W_{\cA,2})|_{f_2^*(U)} \ar[dl] \\
&\cA|_U & }\]

Let $g_1 : S_1\times_Y S_2 \to S_1$, $g_2 : S_1\times_Y S_2 \to S_2$, $f_3 : S_1\times_Y S_2 \to Y$ be induced morphisms.
Let $S_3 := \bar{\pi(B)}$ where $\pi : \cF|_{g_1^*f_1^*(U)} \to g_1^*f_1^*(U)$ is the projection.
Since $S_3$ is projective, we can choose a surjection $q'_3 : \cV_3 \surra \cV_1|_{S_3} \times_{\cF|_{S_3}} \cV_2|_{S_3}$ from a vector bundle $\cV_3$ on $S_3$. Let $q_3$ be the composition $q_3 : \cV_3 \stackrel{q_3'}{\surra} \cV_1|_{S_3} \times_{\cF|_{S_3}} \cV_2|_{S_3} \surra \cF|_{S_3}$.

Let $W_{\cA,3} := \bar{ \cV_3 \times_{\cF|_{S_3}} B'}$. Note that $W_{\cA,3} \subset \cV_3$ is saturated integral subscheme by Lemma \ref{satschlemm2} and we observe that $f_3(S_3)=Z$.
Then, $(f_3 : S_3 \to Z, q_3 : \cV_3 \surra \cF|_{S_3}, W_{\cA,3})$ is a modified proper representative of $\cA$. Let $r_1 : \cV_3 \surra \cV_1|_{S_3}$, $r_2 : \cV_3 \surra \cV_2|_{S_3}$ be the projections. Since $W_{\cA,3}$ and $W_{\cA,2}$ are integral, and $r_1$, $r_2$ are projections of vector bundles and from the constructions of $W_{\cA,1}$ and $W_{\cA,2}$, we can observe that $W_{\cA,i} = r_i^*(\bar{r_i(W_{\cA,3})})$ for $i=1,2$. Thus we have $r_i(W_{\cA,3})$ are closed for $i=1,2$ and we have $W_{\cA,3} = r_i^*(r_i(W_{\cA,3}))$ for $i=1,2$.

We can observe that the natural projections, $g_1 : r_1(W_{\cA,3}) \to W_{\cA,1}$ and $g_2 : r_2(W_{\cA,3}) \to W_{\cA,2}$, are generically finite and surjective. Then we have;

\begin{align*}
& \frac{1}{d_{W_{\cA,3}}} (\zeta_{S_3 \times_Y X})_*(u^*[W_{\cA_3}]) = \frac{1}{d_{W_{\cA,3}}} (f_3)_* (q_3)_* (u^*[W_{\cA,3}]) \\
& = \frac{1}{d_{W_{\cA,3}}}(f_3)_*(q_1)_*(r_1)_*(r_1^*u^*[r_1(W_{\cA,3})]) = (f_3)_*(q_1)_*(u^*[r_1(W_{\cA,3})])  \\
& = \frac{1}{d_{W_{\cA,3}}}(q_1)_*(f_3)_*(u^*[r_1(W_{\cA,3})]) = \frac{1}{d_{W_{\cA,3}}}(q_1)_*(f_1)_*(g_1)_*(u^*[r_1(W_{\cA,3})]) \\
& = \frac{1}{d_{W_{\cA,3}}}(q_1)_*(f_3)_*(u^*[r_1(W_{\cA,3})]) = \frac{1}{d_{W_{\cA,3}}}(q_1)_*(f_1)_*u^*((g_1)_*[r_1(W_{\cA,3})]) \\
& = \frac{1}{d_{W_{\cA,3}}}\cdot \deg(r_1(W_{\cA_3}) \to W_{\cA,1})(q_1)_*(f_1)_*u^*[W_{\cA,1}] \\
& = \frac{\deg(r_1(W_{\cA_3}) \to W_{\cA,1})}{d_{W_{\cA,3}}} \cdot (\zeta_{S_1\times_Y X})_*(u^*[W_{\cA,1}]).
\end{align*}

From the following commutative diagram;
\[
\xymatrix{
W_{\cA,1}\subset \cV_1 \ar[d]^{q_1} & & r_1(W_{\cA,3}) \subset \cV_1|_{S_3} \ar[ll]^{g_1} \ar[d]^{q_1} \\
q_1(W_{\cA,1}) \subset \cF|_{S_1} \ar[rd]_{f_1} & & q_1(r_1(W_{\cA,3})) \subset \cF|_{S_3} \ar[ll]^{g_1} \ar[ld]^{f_3} \\
& \cA &
}
\]
We observe that $\deg(r_1(W_{\cA,3}) \to W_{\cA,1}) = \deg(g_1|_{q_1(r_1(W_{\cA,3}))})$ by same manner as in Lemma \ref{descentpushforward}. From the diagram, we have $\deg(g_1|_{q_1(r_1(W_{\cA,3}))}) \cdot \cdot d_{W_{\cA,1}} = d_{W_{\cA,3}}$. Thus we have $\deg(r_1(W_{\cA,3}) \to W_{\cA,1}) = \frac{d_{W_{\cA,3}}}{d_{W_{\cA,1}}}$. Hence we obtain $\frac{1}{d_{W_{\cA,3}}} (\zeta_{S_3 \times_Y X})_*(u^*[W_{\cA_3}]) = \frac{1}{d_{W_{\cA,1}}} \cdot (\zeta_{S_1\times_Y X})_*(u^*[W_{\cA,1}])$. Thus the definition of the flat pull-backs of cycle groups of coherent sheaf stacks does not depend on a of modified proper representatives of $\cA$.

\end{proof}

Next, we prove the compatibility of flat pull-backs and proper pushforwards. Consider the following fiber diagram :
\[\xymatrix{\cE|_{Y_1\times_X Y_2} \ar[rr]^f \ar[dd]^u \ar[dr] & & \cE|_{Y_2} \ar[dr] \ar[dd]^(0.7)u  \\
& Y_1\times_X Y_2 \ar[rr]^(0.35)f \ar[dd]^(0.3)u & & Y_2 \ar[dd]^u \\
\cE|_{Y_1} \ar[rr]^(0.3)f \ar[dr] & & \cE \ar[rd] & \\& Y_1 \ar[rr]^f & & X
}
\]
where $f$ are proper and $u$ are flat morphisms. Then, we have the following.

\begin{prop}\label{pushforwardpullback}
We have $u^* \circ f_* = f_* \circ u^* : Z_*(\cE|_{Y_1}) \to Z_*(\cE|_{Y_2})$.
\end{prop}
\begin{proof}
Let $\cA \subset \cE|_Z$ be an integral subset where $Z \subset Y_1$ be an integral substack. 
Choose a modified proper representative $(g : S \to f(Z), q : \cV \surra \cE|_S, W_{f(\cA)} )$ of $f(\cA) \subset \cE|_{f(Z)}$.

Choose a dense open subset $U \subset f(Z)$ such that $\cE|_U$ is locally free. If $\deg(f|_{\cA}) =0$, we have $f_*[\cA]=0$.
Let $p_1 : Y_1\times_X S \to Y_1$, $p_2 : Y_1\times_X S \to S$ be the natural morphisms.
By \cite[Lemma 1.10]{Vis89}, we can choose an integral closed subscheme $\cA' \subset \cA|_{f^*(U)}\times_{f(\cA)|_U} q(W_{f(\cA)})|_{g^*(U)}  \subset \cE|_{p_2^*(g^*(U))} $ such that the first projection $\cA' \to \cA|_{g^*(U)}$ is generically finite and surjective. Note that since $\cA' \to f(\cA)$ is surjective, $\cA' \to q(W_{f(\cA)})|_{g^*(U)}$ is also surjective by dimension reason and the fact that $q(W_{f(\cA)})|_{g^*(U)}$ is integral.

Let $S_1 :=\bar{\pi(\cA')}$ where $\pi : \cE|_{p_2^*(g^*(U))} \to p_2^*(g^*(U))$ is the projection. and let $W_{\cA}:= \bar{ \cV|_{p_2*(g^*(U))} \times_{\cE|_{p_2*(g^*(U))}} \cA' } \subset \cV|_{S_1}$. We note that $p_1(S_1) = Z$ and $(p_1 : S_1 \to Z, q : \cV|_{S_1}\surra \cE|_{S_1}, W_{\cA})$ is a modified proper representative of $\cA$.

We can observe that via the morphism $p_2 : \cV|_{S_1} \to \cV$, we have $p_2(W_{\cA}) = W_{f(\cA)}$. 
\[
\xymatrix{
\cV|_{S_1} \supset W_{\cA} \ar[d] \ar[r]^{p_2} & W_{u(\cA)} \subset \cV \ar[d] \\
S_1 \ar[r]^{p_2} \ar[d]^{p_1} & S \ar[d]^{g} \\
Z \ar[r]^{f} & f(Z)
}
\]
Since we have $\deg(f|_{\cA})=0$, generic fiber of the morphism $( \cA|_{f^*(U)} \to f(\cA)|_U )$ has positive dimension. Furthermore, since the morphism $( \cA' \to \cA|_{f^*(U)} )$ and $( q(W_{f(\cA)})|_{g^*(U)} \to u(\cA)|_U )$ are generically finite and surjective, the generic fiber of the morphism $( \cA' \to q(W_{f(\cA)})|_{g^*(U)} )$ has positive dimension, the generic fiber of the morphism $W_{\cA} \to W_{f(\cA)}$ also has positive dimension. 
Therefore, we obtain $(p_2)_*[W_{\cA}]=0$. 
Then we have 
\begin{align*}
& f_* \circ u^* [\cA] = \frac{1}{d_{W_{\cA}}} f_* (\zeta_{(Y_1\times_X Y_2)\times_{Y_2} S})_* (u^*[W_{\cA}]) = \frac{1}{d_{W_{\cA}}}(\zeta_{Y_1\times_X S})_*((p_2)_* u^*[W_{\cA}]) \\
& = \frac{1}{d_{W_{\cA}}}(\zeta_{Y_1\times_X S})_*(u^* (p_2)_* [W_{\cA}])=0
\end{align*}
where the last equality follows from the fact that the generic fiber of the morphism $p_2 : W_{\cA} \to W_{\cA} $ has positive dimension. On the other hand, we have $ u^* \circ f_* [\cA] = 0$ since $\deg(f|_{\cA}) = 0$ Therefore we proved $u^* \circ f_* = f_* \circ u^*$ in this case.

Next, consider the case $\deg(f|_{\cA})>0$. Let $(f : S \to Z, q : \cV \surra \cE|_{S}, W_{f(\cA)})$ be a modified proper representative of $f(\cA)$. Then, by Lemma \ref{component}, there exist an irreducible component $\cA'$ of $\cA|_{f^*(U)} \times_{u(\cA)|_U} q(W_{u(\cA)})|_{g^*(U)} \subset \cE|_{p_2^*(g^*(U))}$ such that the natural morphisms $\cA' \to \cA|_{f^*(U)}, \cA' \to q(W_{f(\cA)})|_{g^*(U)}$ are generically finite and surjective.

Let $S_1 :=\bar{\pi(\cA')}$ where $\pi : \cE|_{p_2^*(g^*(U))} \to p_2^*(g^*(U))$ is the projection. and let $W_{\cA}:= \bar{ \cV|_{p_2*(g^*(U))} \times_{\cE|_{p_2*(g^*(U))}} \cA' } \subset \cV|_{S_1}$. Note that $g(S_1) = Z$, $p_1(S_1) = Z$, $p_2(W_{\cA}) = W_{f(\cA)}$.

Then $(g : S_1 \to Z, q : \cV|_{S_1} \surra \cE|_{S_1}, W_{\cA}\subset \cV)$ is a modified proper representative of $\cA$. Then, we have
\begin{align*}
& f_* \circ u^* [\cA] = \frac{1}{d_{W_{\cA}}} f_* (\zeta_{(Y_1\times_X Y_2)\times_{Y_2} S})_* (u^*[W_{\cA}]) = \frac{1}{d_{W_{\cA}}}(\zeta_{Y_1\times_X S})_*((p_2)_* u^*[W_{\cA}]) \\
& = \frac{1}{d_{W_{\cA}}}(\zeta_{Y_1\times_X S})_*(u^* (p_2)_* [W_{\cA}]) = \frac{\deg(W_{\cA} \to W_{f(\cA)})}{d_{W_{\cA}}}\cdot(\zeta_{Y_1\times_X S})_*(u^*[W_{f(\cA)}]).
\end{align*}
On the other hand, $u^* \circ f_*[\cA] = \frac{\deg{f|_{\cA}}}{d_{W_{\cA}}} \cdot (\zeta_{Y_1\times_X S})_*(u^*[W_{f(\cA)}])$.

From the following commutative diagram;
\[
\xymatrix{
q(W_{\cA}) \ar[r] \ar[d] & q(W_{f(\cA)}) \ar[d] \\
\cA \ar[r] & f(\cA)
}
\]
we obtain $\deg(W_{\cA} \to W_{f(\cA)})\cdot d_{W_{f(\cA)}} = \deg(f|_{\cA}) \cdot d_{W_{\cA}}$. Thus we have $f_* \circ u^* [\cA] = u^* \circ f_* [\cA]$.
\end{proof}

\subsection{Gysin maps via coherent sheaf stacks}\label{sec:gysin}
In this section, we will review a definition of Gysin map of coherent sheaf stack in \cite{CL11semi}. Moreover we will prove compatibility of Gysin map with proper pushforward and flat pull-backs.

\begin{defi}[Gysin map]\cite[Proposition 3.1]{CL11semi}\label{Gysin}
For a coherent sheaf stack $\cF$ on a DM stack $X$, We define a Gysin map to be :
\[ \Gysin{\cF} : Z_*(\cF)_{\QQ} \to A_*(X)_{\QQ}
\]
by the following. We first define for integral cycles and extend linearly. Let $(f: S\to X, \cV \surra \cF|_S, \what{\cA}\subset \cV)$ be a proper representative of an integral substack $\cA \subset \cF|_Z$. Then, we define :
\[\Gysin{\cF}[\cA]:=\frac{1}{\deg(f)}f_* \Gysin{\cV}[\what{\cA}] \in A_*(X)_{\QQ}.
\]
\end{defi}

Since a proper representatives of an integral substack are not unique, we need to check that the Gysin map does not depend on a choice of a proper representative.

First, we can introduce another way to define a Gysin map using modified proper representatives, which will turn out to be equal to an original Gysin map in Proposition \ref{Gysin=mGysin} later.

\begin{defiprop}[Modified Gysin map]\label{mGysin}
We define a modified Gysin map $Z_*(\cF) \to Z_*(X)_{\QQ}$ to be the following. For an integral substack $\cA \subset \cF|_Z$ where $Z \subset X$ is a closed integral substack, and a modified proper representative $(f : S\to Z, q : \cV \surra \cF|_S, W_{\cA})$ of $\cA$, we define a modified Gysin map to be :
\[
0_{\cF}^{m,!}[\cA] := \frac{1}{d_{W_{\cA}}}f_*\Gysin{\cV}[W_{A}]
\]
where $d_{W_{\cA}}:=\deg(q(W_{\cA}) \to \cA)$. Then, the modified Gysin map does not depend on a choice of a modified proper representative.
\end{defiprop}
\begin{proof}
For an integral substack $\cA\subset \cF|_Z$, consider two modified proper representatives $(f_1 : S_1 \to Z, q_1 : \cV_1 \surra \cF|_{S_1}, W_{\cA,1})$, $(f_2 : S_2 \to Z, q_2 : \cV_2 \surra F|_{S_2}, W_{\cA,2})$.
In a same manner as in Lemma \ref{pullbackwelldef}, we can choose a third proper representative $(f_3 : S_3 \to Z, q_3 : \cV_3 \surra \cF|_{S_3}, W_{\cA,3} \subset \cV_3)$ of $\cA$, such that there is a proper morphisms $g_1 : S_3 \to S_1$, $g_2 : S_3 \to S_2$ and surjections of vector bundles $r_1 : \cV_3 \surra \cV_1|_{S_3}$, $r_2 : \cV_3 \surra \cV_2|_{S_3}$ which satisfies $f_1\circ g_1 = f_2 \circ g_2 = f_3 $, $W_{\cA,3} = r_i^*(r_*(W_{\cA,3}))$, $r_*(W_{\cA,3})$ is closed for $i=1,2$. Moreover, $q_3(W_{\cA,3}) \to q_1(W_{\cA,1})$, $q_3(W_{\cA,3}) \to q_2(W_{\cA,2})$ are generically finite. Then, we have

\begin{align*}
& \frac{1}{d_{W_{\cA,3}}}(f_3)_* \Gysin{\cV_3}[W_{\cA,3}]=\frac{1}{d_{W_{\cA,3}}}(f_3)_* \Gysin{\cV_3}[(r_1)^*r_1(W_{\cA,3})]=\frac{1}{d_{W_{\cA,3}}}(f_3)_*\Gysin{\cV_1|_{S_3}}[r_1(W_{\cA,3})] \\
& = \frac{1}{d_{W_{\cA,3}}}(f_1)_*(p_1)_* \Gysin{\cV_1|_{S_3}}[r_1(W_{\cA,3})] = \frac{1}{d_{W_{\cA,3}}}(f_1)_*\Gysin{\cV_1}(p_1)_*[r_1(W_{\cA,3})] \\
& = \frac{\deg(r_1(W_{\cA,3}) \to W_{\cA,1} )}{d_{W_{\cA,3}}}\cdot (f_1)_*\Gysin{\cV_1}[W_{\cA,1}] \\
& = \frac{\deg(r_1(W_{\cA,3}) \to W_{\cA,1} )\cdot  d_{W_{\cA,1}}}{d_{W_{\cA,3}}} \cdot \frac {1}{d_{W_{\cA,3}}} (f_1)_*\Gysin{\cV_1} [W_{\cA,1}]. 
\end{align*}

Similarly, we can show that $\frac{1}{d_{W_{\cA,3}}}(f_3)_*\Gysin{\cV_3}[W_{\cA,3}]=\frac{1}{d_{W_{\cA,2}}}(f_2)_*\Gysin{\cV_2}[W_{\cA,2}]$.
Therefore, the modified Gysin map $0_{\cF}^{m,!}$ does not depend on the choice of modified proper representatives. 

From the following commutative diagram;
\[
\xymatrix{
q_1(r_1(W_{\cA,3})) \ar[r]^{g_1} \ar[d]^{f_3} & q_1(W_{\cA,1}) \ar[dl]^{f_1} \\
\cA &
} 
\]
we obtain $\deg(r_1(W_{\cA,3}) \to W_{\cA,1} )\cdot  d_{W_{\cA,1}} = d_{W_{\cA,3}}$. Thus we have the conclusion.
\end{proof}

\begin{prop}\label{Gysin=mGysin}
The Gysin map $\Gysin{\cF}$ defined in Definition \ref{Gysin} equals to the modified Gysin map $0_{\cF}^{m,!}$. In particular, Gysin map $\Gysin{\cF}$ is well-defined.
\end{prop}
\begin{proof}
Let $\cA \subset \cF$ be an integral substack of a coherent sheaf stack $\cF$ on $X$. Consider a proper representative $(f: S\to X, \cV \surra \cF|_S, \what{\cA})$ of $\cA$. We recall that $\what{\cA}$ is reduced because \'etale pull-backs preserve reducedness. Let $\what{\cA}=\cup_i W_i$. Then, by Remark \ref{modifiedproperexist}, $(f: S\to X, \cV \surra \cF|_S, W_i)$ is a modified proper representation of $\cA$ for every $i$. Via the above proper representative, we have : 
\[
\Gysin{\cF}[\cA] = \frac{1}{deg(f)} f_*\Gysin{\cV}[\what{\cA}]=\frac{1}{deg(f)}\sum\limits_i f_i \Gysin{\cV}[W_i] = \frac{1}{deg(f)}(\sum\limits_i d_{W_i})\cdot 0_{\cF}^{m,!}[\cA].
\]
But by the definition of proper representatives, we can easily observe that the degree $\deg(f)$ is equal to $\sum_i d_{W_i}$. Therefore, we have $\Gysin{\cF}[\cA]=0_{\cF}^{m,!}[\cA]$. 
\end{proof}

Next, we check the compatibility of Gysin maps with proper pushforwards and flat pull-backs. Consider the following fiber diagram :
\[\xymatrix{ \cF|_X \ar[d] \ar[r]^u & \cF \ar[d] \\
X \ar[r]^u & Y}
\]
when $\cF$ is a coherent sheaf stack on $Y$.

\begin{prop}\label{Gysinpushforwardpullback}

\begin{itemize}
\item[(i)] If $u$ is proper, we obtain :
\[\Gysin{\cF} \circ u_* = u_* \circ \Gysin{\cF|_X} : Z_*(\cF|_X)_{\QQ} \to A_*(Y)_{\QQ}.
\]
\item[(ii)] If $u$ is flat, we obtain :
\[ \Gysin{\cF|_X} \circ u^* = u^* \circ \Gysin{\cF} : Z_*(\cF)_{\QQ} \to A_*(X)_{\QQ}.
\]
\end{itemize}
\end{prop}
\begin{proof}(i). Let $\cA \subset \cF|_Z$ be an integral substack where $Z \subset X$ is an integral closed substack. Let $(f : S_2 \to u(Z), q : \cV \surra \cF|_{S_2}, W_{u(\cA)})$ be a modified proper representative of $u(\cA) \subset \cF|_{u(Z)}$. Choose a open dense subset $U \subset u(Z)$ such that $\cF|_U$ is locally free

First consider the case $\deg(u|_{\cA})=0$. Let $p_1 : X\times_Y S_2 \to X$, $p_2 : X \times_Y S_2 \to S_2$ be the projections. 
In the same manner as in the proof of Proposition \ref{pushforwardpullback}, we can choose an integral closed subscheme $\cA' \subset \cA|_{u^*(U)}\times_{u(\cA)|_U} q(W_{u(\cA)})|_{f^*(U)}  \subset \cE|_{p_2^*(f^*(U))} $ such that the first projection $\cA' \to \cA|_{u^*(U)}$ is generically finite and surjective. Note that since $\cA' \to f(\cA)$ is surjective, $\cA' \to q(W_{u(\cA)})|_{f^*(U)}$ is also surjective by dimension reason and the fact that $q(W_{u(\cA)})|_{f^*(U)}$ is integral.

Let $S_1 :=\bar{\pi(\cA')}$ where $\pi : \cF|_{p_2^*(f^*(U))} \to p_2^*(f^*(U))$ is the projection. and let $W_{\cA}:= \bar{ \cV|_{p_2*(f^*(U))} \times_{\cF|_{p_2*(f^*(U))}} \cA' } \subset \cV|_{S_1}$. We note that $p_1(S_1) = Z$ and $(p_1 : S_1 \to Z, q : \cV|_{S_1}\surra \cF|_{S_1}, W_{\cA})$ is a modified proper representative of $\cA$.

We can observe that via the morphism $p_2 : \cV|_{S_1} \to \cV$, we have $p_2(W_{\cA}) = W_{u(\cA)}$. Since we have $\deg(u|_{\cA})=0$, generic fiber of the morphism $( \cA|_{u^*(U)} \to f(\cA)|_U )$ has positive dimension. Furthermore, since the morphism $( \cA' \to \cA|_{u^*(U)} )$ and $( q(W_{u(\cA)})|_{f^*(U)} \to u(\cA)|_U )$ are generically finite and surjective, the generic fiber of the morphism $( \cA' \to q(W_{u(\cA)})|_{f^*(U)} )$ has positive dimension, the generic fiber of the morphism $W_{\cA} \to W_{u(\cA)}$ also has positive dimension. 
Therefore, we obtain $(p_2)_*[W_{\cA}]=0$.

Thus we have 
\[u_*\circ\Gysin{\cF}[\cA]=u_*(p_1)_*\Gysin{\cV|_{S_1}}[W_{\cA}]=f_*(p_2)_*\Gysin{\cV|_{S_1}}[W_{\cA}]=f_*\Gysin{\cV}(p_2)_*[W_{\cA}]=0.\]
On the other hand, $\Gysin{\cF}\circ u_*[\cA] = 0$ since $\deg(u|_{\cA}) = 0$.

Next, assume that $\deg(u|_{\cA})>0$. Then the morphism $\cA|_{u^*(U)} \to u(\cA)|_U$ is generically finite. Consider a modified proper representative $(f : S_2 \to Z, q : \cV \surra \cF|_{S_2}, W_{u(\cA)})$ of $u(\cA)$. Then, by Lemma \ref{component}, there exist an irreducible component $\cA'$ of $\cA|_{u^*(U)} \times_{u(\cA)|_U} q(W_{u(\cA)})|_{f^*(U)} \subset \cF|_{p_2^*(f^*(U))}$ such that the projections $\cA' \to \cA|_{u^*(U)}, \cA' \to q(W_{u(\cA)})|_{f^*(U)}$ are generically finite and surjective.

Let $S_1 :=\bar{\pi(\cA')}$ where $\pi : \cF|_{p_2^*(f^*(U))} \to p_2^*(f^*(U))$ is the projection. and let $W_{\cA}:= \bar{ \cV|_{p_2*(f^*(U))} \times_{\cF|_{p_2*(f^*(U))}} \cA' } \subset \cV|_{S_1}$. We note that $p_1(S_1) = Z$ and $(p_1 : S_1 \to Z, q : \cV|_{S_1}\surra \cF|_{S_1}, W_{\cA})$ is a modified proper representative of $\cA$. Then we have :

\begin{align*}
& \Gysin{\cF}(u_*[\cA]) = (\deg(u|_{\cA})\cdot \Gysin{\cF}\cdot[u(\cA)]) = \frac{\deg(u|_{\cA})}{d_{W_{u(\cA)}}}f_*\Gysin{\cV}[W_{u(\cA)}]
\end{align*}
and
\begin{align*} 
& u_*\Gysin{\cF|_X}[\cA] = \frac{1}{d_{W_{\cA}}} u_* (p_1)_* \Gysin{\cV|_{S_1}}[W_{\cA}] = \frac{1}{d_{W_{\cA}}} f_* (p_2)_* \Gysin{\cV|_{S_1}}[W_{\cA}] \\ 
& = \frac{1}{d_{W_{\cA}}}f_*\Gysin{\cV|_{S_2}}(p_2)_*[W_{\cA}]=\frac{\deg(W_{\cA} \to W_{u(\cA)})}{d_{W_{\cA}}}f_*\Gysin{\cV}[W_{u(\cA)}].
\end{align*}

Then, from the commutative diagram 
\[\xymatrix{q(W_{\cA})|_{p_2^*(f^*(U))}=\cA' \ar[r] \ar[d] \ar@{}[rd]|{\commra} & q(W_{u(\cA)})|_{f^*(U)} \ar[d] \\
\cA|_{u^*(U)} \ar[r] & u(\cA)|_U
}
\]
we have $d_{W_{\cA}} \cdot \deg(u|_{\cA}) = d_{W_{u(\cA)}} \cdot \deg(W_{\cA} \to W_{u(\cA)})$. Since $\cA' \to q(W_{u(\cA)})|_{f^*(U)}$ is generically finite with positive degree, $\deg(W_{\cA} \to W_{u(\cA)}) = \deg(\cA' \to q(W_{u(\cA)})|_{f^*(U)})$ is positive. Thus we have $\Gysin{\cF}(u_*[\cA]) = u_*\Gysin{\cF|_X}[\cA]$.

(ii). Let $\cA \subset \cF|_Z$ be an integral substack where $Z \subset Y$ is an integral closed substack. Choose a modified proper representative $(f : S \to Z, q : \cV \surra \cF|_{S}, W_{\cA})$. 

Let $p_1 : X \times_Y S \to X $, $p_2 : X\times_Y S \to S$ be the projections. Then we have $u^*\Gysin{\cF}[\cA] = \frac{1}{d_{W_{\cA}}}u^*\left( f_* \Gysin{\cV}[W_{\cA}] \right) = \frac{1}{d_{W_{\cA}}}(p_1)_* (p_2)^* \Gysin{\cV}[W_{\cA}] = \frac{1}{d_{W_{\cA}}}(p_1)_* \Gysin{\cV|_{X\times_Y S}}(p_2^*[W_{\cA}])$. 

On the other hand,
\begin{align*}
& \Gysin{\cF}(u^*[\cA]) = \frac{1}{d_{W_{\cA}}} \Gysin{\cF} \circ (\zeta_{X\times_Y S})_*((p_2)^*[W_{\cA}]).
\end{align*}

Let $(p_2)^*[W_{\cA}] = \sum_i a_i [W_i]$ where $W_i$ are integral. Let $S_i := \bar{\pi(W_i)}$ where $\pi : \cV|_{X\times_Y S} \to X\times_Y S$ is the projection. Let $Z_i = p_1(S_i)$, $\cA_i := \zeta_{S_i}(W_i)$. Then we can observe that $(p_1 : S_i \to Z_i, q : \cV|_{S_i} \surra \cF|_{S_i}, W_i)$ is a modified proper representative of $\cA_i \subset \cF|_{Z_i}$.

Then we have $\zeta_{X\times_Y S})_*((p_2)^*[W_{\cA}] = \sum_i a_i d_{W_i} [\cA_i]$. Then we have
\begin{align*}
& \frac{1}{d_{W_{\cA}}}\Gysin{\cF} \circ (\zeta_{X\times_Y S})_*(p_2^*[W_{\cA}]) = \frac{1}{d_{W_{\cA}}}(p_1)_*\left( \sum_i a_i \cdot \Gysin{\cV|_{X\times_Y S}}[W_i] \right) \\
& = \frac{1}{d_{W_{\cA}}}(p_1)_* \Gysin{\cV|_{X\times_Y S}}(p_2^*[W_{\cA}]) = u^*\Gysin{\cF}[\cA].
\end{align*}
\end{proof}

\subsection{Rational equivalences and Chow groups of coherent sheaf stacks}\label{sec:rationaleq} 
In this section, we define notions of rational equivalences and boundary maps. Using these notions, we define a Chow group $A_*(\cF)$ for a coherent sheaf stack $\cF$ on a DM stack $X$. First we define a notion of a group of rational functions. Let $\cA \subset \cF|_Z$ be an integral substack of a coherent sheaf stack $\cF$ on $X$.

\begin{defi}[Group of rational functions]
Let $\cA \subset \cF|_Z$ be an integral substack where $Z \subset X$ is an integral closed substack. Let $U \subset Z$ be an open substack such that $U$ is locally free. Then, we define a group of rational functions $K(\cA)$ on $\cA$ by $K(\cA):=K(\cA|_U)$. For a rational function $h\in K(\cA)$, we use a notation $h : \cA \dra k$.
\end{defi}

Next, we define a notion of rational equivalences.

\begin{defi}[Rational equivalences]
For a coherent sheaf stack $\cF$ on $X$, we define a group of rational equivalences $W_*(\cF)$ to be the free abelian group generated by integral rational equivalences $(\cA,h)$ where $\cA \subset \cF|_Z$ is an integral substack and $h$ is a rational function on $\cA$. We write $[(\cA,h)]$ for an element in $W_*(\cF)$ which corresponds to an integral rational equivalence $(\cA,h)$.
\end{defi}

\begin{defi}[Saturated rational morphism]
Let $\cF$ be a coherent sheaf on $X$. Let $f : S \to X$ be a proper morphism from a DM stack $S$, and let $\cV \surra \cF|_S$ be a surjection from a vector bundle $\cV$ on $S$. Let $W \subset \cV|_S$ be a saturated integral substack. Consider a rational morphism $h : W \dra k$.

Let $S_{W} := \bar{\pi(W)}$ where $\pi : \cV|_S \to S$ is the projection. Then we call $h$ is saturated if there exist an open dense subset $U \subset S_W$ and a locally free resolution $ \cV' \to \cV|_U \surra \cF|_U \to 0$ and an integral substack $\WW \subset \FF := [\cV|_U/ \cV']$ such that $W = \bar{\cV|_U \times_{\FF} \WW} \subset \cV|_{S_W}$(Note that this triple of open dense subset, locally free resolution, and integral substack exists since $W$ is saturated), and there exist a rational function $\bar{h} : \WW \dra k$ such that $p^* \bar{h} = h$ where $p : W|_U \surra \WW$ is the projection.
\end{defi}

We note that the group of rational equivalence does not behave well under proper pushforwards and flat pull-backs. So we introduce a notion of modified rational equivalences.

\begin{defi}[Extended rational equivalences]\label{extendedrateq}
Let $\cF$ be a coherent sheaf stack $\cF$ on $X$. We define a group of extended rational equivalences $W_{\cF}$ by the following. First, consider the free abelian group $W^{E,0}_*(\cF)$ generated by the following generators;
$(f : S \to Z, q : \cV \surra \cF|_S, W_{\cA} \subset \cV, h : W_{\cA} \dra k)$ where $(f : S \to Z, q : \cV \surra \cF|_S, W_{\cA} \subset \cV)$ is a modified proper representative of an integral substack $\cA \subset \cF|_Z$ and $h : W_{\cA} \dra k$ is a saturated rational morphism.

Then we give equivalence relation between generators. Let $H_1 = [(f_1 : S_1 \to Z, q_1 : \cV_2 \surra \cF|_{S_1}, W_{\cA,1} \subset \cV, h_1 : W_{\cA,1} \dra k)]$, $H_2 = [(f_2 : S_2 \to Z, q_2 : \cV_2 \surra \cF|_{S_2}, W_{\cA,2} \subset \cV_2, h : W_{\cA,2} \dra k)]$ be two generating element in $W^{E,0}_*(\cF)$ such that $(f_1 : S_1 \to Z, q_1 : \cV_2 \surra \cF|_{S_1}, W_{\cA,1} \subset \cV)$, $(f_2 : S_2 \to Z, q_2 : \cV_2 \surra \cF|_{S_2}, W_{\cA,2} \subset \cV_2)$ are both modified proper representatives of $\cA \subset \cF|_Z$. 

Assume that there is a third modified proper equivalence $(f_3 : S_3 \to Z, q_3 : \cV_3 \surra \cF|_{S_3}, W_{\cA,3} \subset \cV)$ such that there is a proper morphisms $p_1 : S_3 \to S_1, p_2 : S_3 \to S_2$, and a surjections $r_1 : \cV_3 \surra \cV_1|_{S_3}$, $r_2 : \cV_3 \surra \cV_2|_{S_2}$ such that $r_i(W_3)\subset \cV_i|_{S_i}$ are closed and $W_{\cA,3} = r_i^* r_i(W_{\cA,3})$ for $i=1,2$, and moreover, $p_i(r_i(W_{\cA,3})) = W_{\cA,i}$ and the morphism $p_i : r_i(W_{\cA,3}) \to W_{\cA,i}$ is generically finite for $i=1,2$.

Next, further assume that the following. Consider the descent of the saturated rational morphism $ r_2^*(p_2\circ h_2) : W_{\cA,3} \dra k$,  $\bar{r_2^*(p_2\circ h_2)} : r_2(W_{\cA,2}) \dra {k}$. Then we consider its norm $N(\bar{r_2^*(p_2\circ h_2)}) : W_{\cA,1} \dra k$ where the norm is taken via the proper generically finite morphism $p_1 : r_1(W_{\cA,3}) \to W_{\cA,1}$. 

Assume that $N(\bar{r_2^*(p_2\circ h_2)}) = h_1^{d}$. Let $c = gcd(\deg(r_1(W_{\cA,3}) \to W_{\cA,1}), d)$ and let $d = c \cdot d_1$, $\deg(r_1(W_{\cA,3}) \to W_{\cA,1}) = c \cdot d_2$.
Then we say $d_1 \cdot H_1$ and $d_2 \cdot H_2$ are equivalent.

We define the group of rational equivalence $W_*^{E}(\cF)$ to be the quotient of the free abelian group $W_*^{E,0}(\cF)$ by the above equivalence relations.

\end{defi}

\begin{defi}[Boundary map]\label{boundary}
We define boundary map for rational equivalences and extended rational equivalences.

\begin{itemize}
\item[(i)] We define a morphism $i : W_*(\cF) \to W_*^{E}(\cF)$ Let $(\cA,h)$ be an integral rational equivalence where $\cA \subset \cF|_Z$. Let $(f : S\to Z, q : \cV \surra \cF|_S, W_{\cA} )$ be a modified proper representative of $\cA \subset \cF$. Let $\wtil{h} : W_{\cA} \to q(Z_{\cA}) \to \cA \dra k$ be a rational function obtained by the composition. Then we define $i([(\cA,h)]) := [(f : S\to Z, q : \cV \surra \cF|_S, W_{\cA}, \wtil{h} )]$.

\item[(ii)]
$[(f : S\to Z, q : \cV \surra \cF|_S, W_{\cA}, h )]$ be a generating element of $W^E_*(\cF)$.
Let $\rou h= \sum_i c_i [W_i]$, and let $d_{W_i} = \deg(q(Z_i) \to f(q(W_i)))$.

Then we define the boundary map $\rou : W^{E}_*(\cF) \to Z_*(\cF)$ to be
\begin{align*}
&\rou [(f : S\to Z, q : \cV \surra \cF|_S, W_{\cA}, h )] := \frac{1}{d_{W_{\cA}}}\sum\limits_i d_{W_i} \cdot c_i[\zeta_S(W_i)] \\
& = \frac{1}{d_{W_{\cA}}}(\zeta_S)_*(\rou h).
\end{align*}

\item[(iii)]
We define the boundary map $rou : W_*(\cF) \to Z_*(\cF)$ by the composition $\rou \circ i$.

\end{itemize}
\end{defi}

We check the well-definedness of the boundary map in the following three lemmas.

\begin{lemm}
In Definition \ref{boundary}, $W_i$ are saturated integral subschemes of $\cV$.
\end{lemm}
\begin{proof}
Since $W_i$ is saturated, there exist an open cover $\{S_{\alpha}\}_{\alpha \in\ cI}$, locally free resolutions $\cV_{\alpha}' \to \cV|_{S_{\alpha}} \surra \cF|_{S_{\alpha}} \to 0$ for every $\alpha$, and integral substacks $\WW_{\alpha} \subset \FF_{\alpha} := [\cV|_{S_{\alpha}}/ \cV_{\alpha}']$ such that $W_i|_{S_{\alpha}} = \WW_{\alpha}\times_{\FF_{\alpha}} \cV|_{\alpha}$. Since $h \in K(\cA|_U)$ where $U \subset Z$ is an open dense subset such that $\cF|_U$ is locally free, we observe that the rational function $h$ is $\cV_{\alpha}'$-invariant for each $\alpha$. Therefore, $h$ descend to a rational function $\bar{h}_{\alpha} : \FF_{\alpha} \dra k$ for each $\alpha$. Since $p : \cV|_{S_{\alpha}} \to \FF_{\alpha}$ is smooth, we have $\rou(\wtil{h}) |_{S_{\alpha}} = p^* \rou(\bar{h}_{\alpha})$. Moreover, since $p$ is a projection of vector bundle, there exist an integral substack $\WW_{\alpha} \subset \FF_{\alpha}$ such that $p^*[\WW_{\alpha}] = [W_i|_{S_{\alpha}}]$, i.e. $\WW_{\alpha} \times_{\FF_{\alpha}} \cV|_{S_{\alpha}} = W_i|_{S_{\alpha}}$. Therefore $W_i$ is a saturated integral subscheme of $\cV$.
\end{proof}

\begin{lemm}\label{boundarywelldef1}
The morphism $i : W_*(\cF) \to W_*^{E}(\cF)$ defined in Definition \ref{boundary}, $(i)$ does not depend on the choice of modified proper representatives.
\end{lemm}
\begin{proof}
Let $(\cA, h)$ be an integral rational equivalences where $\cA \subset \cF|_Z$ is an integral substack.
Consider two modified proper representatives
$(f_1 : S_1 \to Z, q_1 : \cV_1 \surra \cF|_{S_1}, W_{\cA,1})$ and $(f_2 : S_2 \to Z, q_2 : \cV_2 \surra \cF|_{S_2}, W_{\cA,2})$ of $(\cA)$. In a similar manner as in the proof of Lemma \ref{pullbackwelldef}, we choose a third modified proper representative $(f_3 : S_3 \to Z, q_3 : \cV_3 \surra \cF|_{S_3}, W_{\cA,3})$ such that there is a proper morphisms $p_1 : S_3 \to S_1$, $p_2 : S_3 \to S_2$ and surjections of vector bundles $r_1 : \cV_3 \surra \cV_1|_{S_3}$, $r_2 : \cV_3 \surra \cV_2|_{S_3}$ which satisfies $f_1\circ p_1 = f_2 \circ p_2 = f_3$, $W_{\cA,3} = r_i^*(r_*(W_{\cA,3}))$, $r_*(W_{\cA,3})$ is closed for $i=1,2$. Moreover, $p_i : r_i(W_{\cA,3}) = W_{\cA,i}$, $p_i : r_i(W_{\cA,3}) \to W_{\cA,i}$ are generically finite for $i=1,2$.

Let $\wtil{h}_1 : W_{\cA,1} \dra k$, $\wtil{h}_2 : W_{\cA,2} \dra k$, $\wtil{h}_3 : W_{\cA,3} \dra k$ be rational morphisms induced from $h$. Then clearly $r_2^*(p_2\circ \wtil{h}_2) = \wtil{h}_3$ is a saturated rational morphism and let $\bar{h}_3 : r_1(W_{\cA,3}) \dra k$ be a descent of $\wtil{h}_3$. Let $N(\bar{h}_3) : W_{\cA} \dra k$ be the norm of $\bar{h}_3$. Since $r_1^* (p_1 \circ \wtil{h}_1) = \wtil{h}_3$, we have $p_1 \circ \wtil{h}_1 = \bar{h}_3$. Hence we have $N(\bar{h}_3) = \wtil{h}_1^{\deg(r_1(W_{\cA,3}) \to W_{\cA,1})}$. Therefore $[(f_1 : S_1 \to Z, q_1 : \cV_1 \surra \cF|_{S_1}, W_{\cA,1}, \wtil{h}_1)]$ and $[(f_2 : S_2 \to Z, q_2 : \cV_2 \surra \cF|_{S_2}, W_{\cA,2}, \wtil{h}_2 )]$ are equivalent.

\end{proof}

\begin{lemm}\label{boundarywelldef2}
The boundary map $\rou : W^{E}_*(\cF) \to Z_*(\cF)$ is well-defined, i.e. equivalent element gives same boundary.
\end{lemm}
\begin{proof}
Let $H_1 = [(f_1 : S_1 \to Z, q_1 : \cV_2 \surra \cF|_{S_1}, W_{\cA,1} \subset \cV, h_1 : W_{\cA,1} \dra k)]$, $H_2 = [(f_2 : S_2 \to Z, q_2 : \cV_2 \surra \cF|_{S_2}, W_{\cA,2} \subset \cV_2, h : W_{\cA,2} \dra k)]$ be two generating element in $W^{E}_*(\cF)$ such that $aH_1$ and $bH_2$ are equivalent and $(f_1 : S_1 \to Z, q_1 : \cV_2 \surra \cF|_{S_1}, W_{\cA,1} \subset \cV)$, $(f_2 : S_2 \to Z, q_2 : \cV_2 \surra \cF|_{S_2}, W_{\cA,2} \subset \cV_2)$ are both modified proper representatives of $\cA \subset \cF|_Z$. As in the Definition \ref{extendedrateq}, there is a third modified proper representative $(f_3 : S_3 \to Z, q_3 : \cV_3 \surra \cF|_{S_3}, W_{\cA,3} \subset \cV)$. 

Let $\wtil{h}_3 := r_2^*(p_2\circ \wtil{h}_2)$. Since $r_2 : \cV_3 \surra \cV_2|_{S_3}$ is a surjection of vector bundle, $W_{\cA,3} = r_2^*(r_2(W_{\cA,3}))$, we conclude that $(r_2)_* \rou \wtil{h}_3 = \rou( p_2 \circ \wtil{h}_2 )$

We have $(p_2)_* \rou( p_2 \circ \wtil{h}_2 ) = \deg(r_2(W_{\cA,3}) \to W_{\cA,2}) \cdot \rou \wtil{h}_2$. Then we obtain 
\begin{align*}
& \frac{1}{d_{W_{\cA,3}}}(\zeta_{S_3})_*(\rou \wtil{h}_3) = \frac{1}{d_{W_{\cA,3}}} (f_3)_* (q_3)_* (\rou \wtil{h}_3) = \frac{1}{d_{W_{\cA,3}}} (f_2)_* \circ (p_2)_* \circ (q_2)_* \circ (r_2)_* (\rou \wtil{h}_3) \\
& = \frac{1}{d_{W_{\cA,3}}} (f_2)_* \circ (p_2)_* \circ (q_2)_* \rou(p_2 \circ \wtil{h}_2) = \frac{\deg(r_2(W_{\cA,3}) \to W_{\cA,2})}{d_{W_{\cA,3}}} (f_2)_* (q_2)_* \rou \wtil{h}_2 \\
& = \frac{\deg( r_2(W_{\cA,3}) \to W_{\cA,2} )}{d_{W_{\cA,3}}} (\zeta_{S_2})_*(\rou \wtil{h}_2).
\end{align*}

In a similar manner as in Definition-Proposition \ref{mGysin}, we can show that $\frac{\deg( r_2(W_{\cA,3}) \to W_{\cA,2} )}{d_{W_{\cA,3}}} = \frac{1}{d_{W_{\cA,2}}}$. Thus we have $\frac{1}{d_{W_{\cA,3}}}(\zeta_{S_3})_*(\rou \wtil{h}_3) = \rou H_2$.

Next, let $\bar{h}_3 : r_1(W_{\cA,3}) \dra k$ be a descent of $\wtil{h}_3$ on $r_1(W_{\cA,3})$. By assumption, we have $N(r_1(\bar{h}_3)) : W_{\cA,1} \dra k$ is equal to $\wtil{h}_1^{d}$ such that $a \cdot \deg(r_1(W_{\cA,3}) \to W_{\cA,1}) = d \cdot b$.

Since $r_1 : \cV_3 \surra \cV_1|_{S_3}$ is a surjection of vector bundle, $W_{\cA,3} = r_1^*(r_1(W_{\cA,3}))$, we conclude that $(r_1)_* \rou \wtil{h}_3 = \rou( \bar{h}_3 )$. Then we obtain

\begin{align*}
& \frac{1}{d_{W_{\cA,3}}}(\zeta_{S_3})_*(\rou \wtil{h}_3) = \frac{1}{d_{W_{\cA,3}}} (f_3)_* (q_3)_* (\rou \wtil{h}_3) = \frac{1}{d_{W_{\cA,3}}} (f_1)_* \circ (p_1)_* \circ (q_1)_* \circ (r_1)_* (\rou \wtil{h}_3) \\
& = \frac{1}{d_{W_{\cA,3}}} (f_1)_* \circ (p_1)_* \circ (q_1)_* \rou \bar{h}_3 = \frac{1}{d_{W_{\cA,3}}} (f_1)_* (q_1)_* \rou N(\bar{h}_3) \\
& = \frac{d}{d_{W_{\cA,3}}} (\zeta_{S_1})_* \rou \wtil{h}_1 = \frac{a \cdot \deg(r_1(W_{\cA,3}) \to W_{\cA,1})}{b \cdot d_{W_{\cA,3}}} (\zeta_{S_1})_* \rou \wtil{h}_1 = \frac{a}{b \cdot d_{W_{\cA,1}}} (\zeta_{S_1})_* \rou \wtil{h}_1 \\
& = \frac{a}{b} \rou H_1.
\end{align*}

Thus we have $a \cdot \rou H_1 = b \cdot \rou H_2$, so the boundary map gives same value for equivalent elements.
\end{proof}

\begin{rema}
We note that our definition of the boundary map is equivalent to the definition of the boundary map in \cite[Proposition 3.3]{CL11semi} when restricted to the group of rational equivalence $W_*(\cF)$. The difference is that we used modified proper representatives, so our definition is a little simpler than the definition in \cite[Proposition 3.3]{CL11semi}. Since the proof of the equivalence of these two definitions is similar to the proof of Proposition \ref{properrepsat}, we omit it here.
\end{rema}

Finally, we are define a Chow group $\cA_*(\cF)$ for a coherent sheaf stack $\cF$ on $X$. We define the Chow group $A_*(\cF)$ to be :
\[A_*(\cF) := Z_*(\cF)/ \rou(W^E_*(\cF)).
\]
Since we used extended rational equivalence, the definition of Chow group is slightly different from the definition in \cite{CL11semi}.

%

We show that the rational equivalence factors though Gysin map of coherent sheaf stacks.

\begin{prop} For a coherent sheaf stack $\cF$ on $X$, the Gysin map $\Gysin{\cF} : Z_*(\cF) \to A_*(\cF)$.
Therefore, it also factors through the Chow group $A_*(\cF)$.
\end{prop}
\begin{proof}
Let $(\cA, h)$ be an integral rational equivalence where $\cA \subset \cF|_Z$ is an integral substack. Choose a modified rational equivalence $(f : S\to Z, q : \cV \surra \cF|_S, W_{\cA} \subset \cV)$ of $\cA$. Then we have
\begin{align*}
& \Gysin{\cF}(\rou(\cA,h)) = \frac{1}{d_{W_{\cA}}}\Gysin{\cF}\left( (\zeta_S)_*\rou \wtil{h} \right)
\end{align*}
where $\wtil{h} : W_{\cA} \dra k$ is a rational morphism induced from $h$. Let $\rou \wtil{h} = \sum_i a_i [W_i]$ where $W_i$ are integral. Let $S_i := \bar{\pi(W_i)}$ where $\pi : \cV \to S$ is the projection, and let $Z_i = f(S_i)$.

Let $\cA_i = \zeta_{S}(W_i)$, then $(f : S_i \to Z_i, q : \cV|_{S_i} \surra \cF|_{S_i}, W_i )$ is a modified proper representative of $\cA_i \subset \cF|_{Z_i}$. Then we obtain
\begin{align*}
& \frac{1}{d_{W_{\cA}}}\Gysin{\cF}\left( (\zeta_S)_*\rou \wtil{h} \right) = \frac{1}{d_{W_{\cA}}}\Gysin{\cF}\left( a_i d_{W_i} [\cA_i] \right) = \frac{1}{d_{W_{\cA}}} f_* \left( \sum_i a_i \Gysin{\cV}[W_i] \right) \\
& = \frac{1}{d_{W_{\cA}}} f_* \Gysin{\cV}(\rou \wtil{h}).
\end{align*}
Since Gysin homomorphism of vector bundles factors through rational equivalences, we have the conclusion.

\end{proof}


The notions of modified proper representatives and modified Gysin make it easy to prove well-definedness of the Gysin map. But there is another critical reason why we should consider the notion of modified proper representatives. Look at the following example :

\begin{exam}\label{ratwelldefex}
Consider the following diagram :
\[\xymatrix@C=1pt{
Z:=\{\alpha b - \beta a = 0\} \ar[r]^(0.5)f  \ar@{}[d]|{\bigcap} & \cA:= \{vb-wa = 0\} \ar@{}[d]|{\bigcap} \\
\cF|_S \ar[d]^{\pi} \ar[r]^(0.3){f} & \cF = \PP^2 \times k^2 = \{[u:v:w],(a,b)\} \ar[d]^{\pi} \\
S:=Bl_{[1:0:0]}\PP^2 \ar[r]^(0.4)f \ar@{}[d]|{\bigcap} & X:=\PP^2 = \{[u:v:w]\} \\
\PP^2 \times \PP^1 = \{[u:v:w],[\alpha : \beta]\}. &
}
\]
The proper representative of $\cA \subset \cF$ via the generically finite, surjective morphism $f$ is
\[(f : S \to X, q : \cV|_S = \cF|_S=S \times k^2 \stackrel{=}{\lra} \cF|_S).
\]
Let $h : \cA \dra k$ be a rational function defined by $h([u:v:w],(\alpha,\beta))=\frac{u}{v}$.
Let $\wtil{h}$ be a rational function $\cA|_S \dra k$ induced from $h$. Then, we have :
\[\rou \wtil{h} = [\{u=0\}]-[\textrm{\wwtil{\{v=0\}\subset \PP^2 \times k^2}}] - [\{([1:0:0],[\alpha : \beta],(a,b))|\alpha b-\beta a = 0\}].
\]
Therefore, we have : 
\begin{align*}
&\rou h = [\zeta_S(\{u=0\})]-[\zeta_S(\textrm{\wwtil{\{v=0\}\subset \PP^2 \times k^2}})] \\
& - [\zeta_S(\{([1:0:0],[\alpha : \beta],(a,b))|\alpha b-\beta a = 0\})].
\end{align*}
But since $\{([1:0:0],[\alpha : \beta],(a,b))|\alpha b-\beta a = 0\}$ lies over the exceptional divisor $E=\{([1:0:0],[\alpha:\beta])\}\subset Bl_{[1:0:0]}\PP^2$ and $E \to \{[1:0:0]\}\subset \PP^2$ is not generically finite, $(E \to \{[1:0:0]\}, E \times k^2 \stackrel{=}{\lra}\cF|_E, \{([1:0:0],[\alpha : \beta],(a,b))|\alpha b-\beta a = 0\})$ is not a proper representative of $\zeta_S(\{([1:0:0],[\alpha : \beta],(a,b))|\alpha b-\beta a = 0\})$ but it is a modified proper representative.

Therefore, in this case, boundary of the proper representative is not a proper representative of the boundary, but it is a modified proper representative of the boundary. Therefore, it is necessary to extend the notion of proper representatives to the notion of modified proper representatives. 
\end{exam}

It is natural to consider proper pushforwards and flat pull-backs in Chow groups. But unfortunately, proper pushforwards does not behaves well in Chow groups. So we consider $A_*(\cF)_{\QQ} := A_*(\cF) \otimes_{\ZZ} \QQ$.

We need to check the compatibility of proper pushforwards and flat pullbacks with boundary maps. Consider the following diagram : 
\[\xymatrix{ \cF|_X \ar[r]^u \ar[d] & \cF \ar[d] \\
X \ar[r]^u & Y}
\]
where $\cF$ is a coherent sheaf stack on $Y$ and $u$ is a proper morphism. For a generating element $H_1 = [(f_1 : S_1 \to Z, q_1 : \cV_1 \surra \cF|_S, W_{\cA,1} \subset \cV_1, h_1 : W_{\cA,1} \dra k )] \in W^E_*(\cF|_X)$ where $\cA \subset \cF|_Z$ is an integral substack, $Z \subset X$ is an integral closed substack, we have the following. 

\begin{prop} \label{rationalpushforward}
In the above setting, we have :
\[ u_* \rou H = 0 \in A_*(\cF)_{\QQ}.
\]
Therefore, proper pushforward $u_* : Z_*(\cF|_X) \to Z_*(\cF)$ induces a morphism $u_* : A_*(\cF|_X)_{\QQ} \to A_*(\cF)_{\QQ}$.
\end{prop}
\begin{proof}
In the same manner as described in Proposition \ref{pushforwardpullback}, we can fill in the diagram
\[\xymatrix{W_{\cA,2} \subset \cV|_{S_2}\surra \cF|_{S_2} \ar[d] & & W_{u(\cA)} \subset \cV \surra \cF|_{S} \ar[d] & \\
S_2 \ar[rr]^{p_2} \ar[dr]_{p_1} & & S \ar[dr]^f & \\
& X \ar[rr]^u & & Y
}
\]
such that $(f:S \to u(Z), q : \cV \surra \cF|_{S}, W_{u(\cA)})$ is a modified proper representative of $u(\cA)$ and $(p_1 : S_2 \to Z, q_2 : \cV|_{S_2} \surra \cF|_{S_2}, W_{\cA,2})$ is a modified proper representative of $\cA$ where $S_2 \subset X\times_Y S$, $p_1 : X\times_Y S \to X$, $p_2 : X\times_Y S \to S$ are natural projections,
and $p_2(W_{\cA,2}) = W_{u(\cA)}$, $p_2 : W_{\cA,1} \to W_{u(\cA)}$ is generically finite.

Next, in a similar manner as in the proof of Lemma \ref{pullbackwelldef}, we choose a third modified proper representative $(f_3 : S_3 \to Z, q_2 : \cV_3 \surra \cF|_{S_3}, W_{\cA,3})$ such that there is a proper morphisms $g_1 : S_3 \to S_1$, $g_2 : S_3 \to S_2$ and surjections of vector bundles $r_1 : \cV_3 \surra \cV_1|_{S_3}$, $r_2 : \cV_3 \surra \cV_2|_{S_3}$ which satisfies $f_1\circ p_1 = f_2 \circ p_2 = f_3$, $W_{\cA,3} = r_i^*(r_*(W_{\cA,3}))$, $r_*(W_{\cA,3})$ is closed for $i=1,2$. Moreover, $p_i : r_i(W_{\cA,3}) = W_{\cA,i}$, $p_i : r_i(W_{\cA,3}) \to W_{\cA,i}$ are generically finite for $i=1,2$.

Let $h_3 := r_1^*(p_1\circ h_1)$ be a rational function on $W_{\cA,3}$. Then we can observe that it is saturated. Then it descend to a rational function $\bar{h}_3 : r_2(W_{\cA,3}) \dra k$. Consider its norm $h_2:=N(\bar{h}_3) : W_{\cA,2} \dra k$. Let $H_2 = [(p_1 : S_2 \to Z, q : \cV|_X \surra \cF|_{S_2}, W_{\cA,2}, h_2 )] \in W^E_*(\cF)$. Then we can easily check that $\deg(r_2(W_{\cA,3}) \to W_{\cA,2}) \cdot H_1 = H_2$.

Next, let $h := N(h_2)$ be a norm of $h_2$, which is a rational function on $W_{u(\cA)}$. Then $H = [(f:S \to u(Z), q : \cV \surra \cF|_{S}, W_{u(\cA)}, h)]$ is an element of $W^E_*(\cF)$. Then we have

\begin{align*}
& u_* \rou H_2 = u_* \left( \frac{1}{d_{W_{\cA,2}}}  (\zeta_{S_2})_*(\rou h_2) \right) = \frac{1}{d_{W_{\cA,2}}} u_*(p_1)_* (q_2)_* \rou h_2 \\
& = \frac{1}{d_{W_{\cA,2}}}f_* (p_2)_* (q_2)_* \rou h_2 = \frac{1}{d_{W_{\cA,2}}} f_* q_* \rou h_3 = \frac{d_{W_{u(\cA)}}}{d_{W_{\cA,2}}}(\zeta_{S})_* \rou h_3 \\
& = \frac{d_{W_{u(\cA)}}}{d_{W_{\cA,2}}} \rou H.
\end{align*}

Hence we conclude that 
\[
u_* \rou H_1 = \frac{1}{\deg(r_2(W_{\cA,3}) \to W_{\cA,2})} \cdot u_* \rou H_2 = \frac{ d_{W_{u(\cA)}}}{\deg(r_2(W_{\cA,3}) \to W_{\cA,2}) \cdot d_{W_{\cA,2}}} \rou H
\]
in $Z_*(\cF)_{\QQ}$. Thus we have $u_* \rou H_1 = 0$ in $A_*(\cF)_{\QQ}$.

\end{proof}

\begin{rema}\label{extratpushforward}
In the proof of the Proposition \ref{rationalpushforward}, by defining $u_* H_1 := \frac{ d_{W_{u(\cA)}}}{\deg(r_2(W_{\cA,3}) \to W_{\cA,2}) \cdot d_{W_{\cA,2}}} \rou H$, we obtain a homomorphism $u_* : W^E_*(\cF|_X)_{\QQ} \to W^E_*(\cF)_{\QQ}$. Then we can rewrite Proposition \ref{rationalpushforward} by $u_* \circ \rou = \rou \circ u_*$.
\end{rema}

Next, we consider flat pull-backs. We prove that flat pull-backs also factors through rational equivalences.
Consider the following diagram : 
\[\xymatrix{ \cF|_X \ar[r]^u \ar[d] & \cF \ar[d] \\
X \ar[r]^u & Y}
\]
where $\cF$ is a coherent sheaf stack on $Y$ and $u$ is a flat morphism.

For a generating element $H = [(f : S \to Z, q : \cV \surra \cF|_S, W_{\cA} \subset \cV, h : W_{\cA} \dra k )] \in W^E_*(\cF)$ where $\cA \subset \cF$ is an integral substack, $Z \subset Y$ is an integral closed substack, we have the following.

\begin{prop} \label{rationalpullback}
We have :
\[ u^* \rou H = 0 \in A_*(\cF|_X).
\]
\end{prop}
\begin{proof}

Choose a modified rational equivalence $(f : S \to Z, q : \cV \surra \cF|_S, W_{\cA} \subset \cV)$ of $\cA$. 
Let $u^*[W_{\cA}] = \sum_i a_i [W_i]$ where $W_i$ are integral. 

There are natural morphisms $W_i \to W_{\cA}$ and therefore there are induced rational morphisms $h_i : W_i \dra k$. We note that $h_i$ are saturated. Let $p_1 : X\times_Y S \to X$, $p_2 : X\times_Y S \to S$ be the projections. Let $S_i := \bar{\pi(W_i)}$ where $\pi : \cV|_{X\times_Y S} \to X\times_Y S$ is the projection, and let $Z_i = p_1(S_i) \subset X$. Let $\cA_i = \zeta_{X\times_Y S}(W_i) \subset \cF|_{Z_i}$, then $(p_1 : S_i \to Z_i, q : \cV|_{S_i} \surra \cF|_{S_i}, W_i )$ is a modified proper representative of $\cA_i \subset \cF|_{Z_i}$ and $H_i = [(p_1 : S_i \to Z_i, q : \cV|_{S_i} \surra \cF|_{S_i}, W_i, h_i)] \in W^E_*(\cF|_X)$.

Then we have
\begin{align*}
& u^* \rou H = \frac{1}{d_{W_{\cA}}} u^*\left( (\zeta_S)_* \rou h \right) = \frac{1}{d_{W_{\cA}}}u^*f_*q_* \rou h = \frac{1}{d_{W_{\cA}}}(p_1)_* q_* u^* \rou h \\
& = \frac{1}{d_{W_{\cA}}}(p_1)_*q_*\left( \sum_i a_i \rou h_i \right) = \frac{1}{d_{W_{\cA}}}\sum_i a_i \cdot d_{W_i} \rou H_i \in Z_*(\cF|_X)_{\QQ}.
\end{align*}

Hence $u^* \rou H = 0$ in $A_*(\cF|_X)$.

\end{proof}

\begin{rema}\label{extratpullback}
In the proof of the Proposition \ref{rationalpullback}, by defining $u^* H := \sum_i \frac{a_i d_{W_i}}{d_{W_{\cA}}} H_i$, we obtain a homomorphism $u^* : W^E_*(\cF)_{\QQ} \to W^E_*(\cF|_X)$. Then we can rewrite Proposition \ref{rationalpullback} by $u^* \circ \rou = \rou \circ u^*$.
\end{rema}

\section{Semi-virtual pull-backs via semi-perfect obstruction theories}\label{sec:semivirtpullback}

Consider a DM-type morphism $X \to Y$, where $X$ is a DM stack and $Y$ is a DM stack or a smooth Artin stack. The definition of DM-type morphism appear in \cite{Man08}, i.e. for any morphism $U \to Y$ where $U$ is a scheme, $X\times_Y U$ is a DM stack. In this section, we define semi-virtual pull-backs $A_*(Y) \to A_*(X)$ when there is a semi-perfect obstruction theory over $X \to Y$, which is defined in Chang-Li \cite{CL11semi}. Here, Chow group $A_*(Y)$ means the na\"ive chow group of $Y$ defined in \cite{Kre99}. Then we prove basic properties of semi-virtual pull-backs, including a functoriality property, which lead to a deformation invariances property of virtual fundamental cycles. 
When there there is a perfect obstruction theory over $X \to Y$, which is also an semi-perfect obstruction theory, we show that virtual pull-backs and semi-virtual pull-backs coincide. It leads to a conclusion that the virtual fundamental class defined via the perfect obstruction theory and the virtual fundamental class defined via the semi-perfect obstruction theory are equal.

\subsection{Semi-perfect obstruction theories}\label{sec:semiperf}

We briefly review the definition of semi-perfect obstruction theories which was introduced in the paper of H.L. Chang and J. Li, \cite{CL11semi}.

Consider a DM type morphism $X \to Y$ where $X$ is a DM stack and $Y$ is a DM stack or a smooth Artin stack. Assume that there is an \'etale covering $\{ \vphi_{\alpha} : X_{\alpha} \to X \}_{\alpha \in \cI}$ and there are perfect obstruction theories $\phi_{\alpha} : (E_{\alpha})\hseq \to L_{X_{\alpha}/Y}$ for each $\alpha \in \cI$.
Note that we have isomorphism $L_{X_{\alpha}/Y} \cong \vphi_{\alpha}^* L_{X/Y}$ since there is a distinguished triangle $\vphi_{\alpha}^*L_{X/Y} \to L_{X_{\alpha}/Y} \to L_{X_{\alpha}/X} \to \vphi_{\alpha}^*L_{X/Y}[1]$ and $L_{X_{\alpha}/X}=0$ because $X_{\alpha} \to X$ is \'etale.

Let $(E_{\alpha})\cseq := ((E_{\alpha})\hseq)\dual$. For $\alpha, \beta \in \cI$, assume that there are transition isomorphisms $\psi_{\alpha \beta} : h^1(E_{\alpha})\cseq|_{X_{\alpha} \times_{X} X_{\beta}} \to h^1(E_{\beta})\cseq|_{X_{\alpha }\times_X X_{\beta}}$ of first cohomology sheaves which satisfy cocycle conditions. Therefore, first cohomology sheaves $h^1(E_{\alpha})\cseq$'s glue to a coherent sheaf $\cE$ on $X$.

Furthermore, we assume that $\phi_{\alpha}|_{X_{\alpha}\times_X X_{\beta}} : (E_{\alpha})\hseq|_{X_{\alpha} \times_X X_{\beta}} \to L_{X_{\alpha}/Y}|_{X_{\alpha}\times_X X_{\beta}} \cong L_{X_{\alpha}\times_X X_{\beta}/Y}$ and $\phi_{\beta}|_{X_{\alpha}\times_X X_{\beta}} : (E_{\beta})\hseq|_{X_{\alpha} \times_X X_{\beta}} \to L_{X_{\beta}/Y}|_{X_{\alpha}\times_X X_{\beta}}\cong L_{X_{\alpha}\times_X X_{\beta}/Y}$ are $\nu$-equivalent via $\psi_{\alpha \beta}$. Then, we call a triple of data $\phi = (\{X_{\alpha}\}_{\alpha \in \cI}, \phi_{\alpha}, \psi_{\alpha \beta})$ a semi-perfect obstruction theory on $X/Y$, and call $\cE$ an obstruction sheaf of $\phi$ and denote it $Ob_{\phi}$. We define the notion of $\nu$-equivalence in the followings.

\begin{defi}[Infinitesimal lifting problem]\label{inflifting}
Consider a morphism $u : A \to B$ between Artin stacks. Let $(R,m)$ be a local Artinian ring and $I\subset R$ be an ideal such that $I\cdot m = 0$. Let $T = Spec(R)$ and $T' = Spec(R/I) \subset T$. We usually call $T' \hra T$ a small extension. Consider the following diagram :
\[\xymatrix{x \in T' \ar[r]^(0.4)f \ar@{^(->}[d]_{\iota} & A \ni p=f(x) \ar[d]^u \\
T \ar[r]^{g} \ar@{-->}[ur]^{h} & B
}
\]
where $x\in T'$ is a closed point of $T$, and $f(x)=p$. Then, to find a morphism $h : T \to A$ which makes the diagram commutes is called an infinitesimal lifting problem on $A/B$ at $p$.

By \cite[Chapter 3, Theorem 2.1.7]{Ill06}, we can observe that there is an element $w(f,T',T)\in \Ext^1(f^*L_{A/B},I)\cong T^1_{p,A/B}\otimes_k I$, whose vanishing is equivalent to the existence of the morphism $h$ for above infinitesimal lifting problems. Consider a perfect obstruction theory $\phi : E\hseq \to L_{A/B}$. Then, there is an induced morphism $H^1(\phi\dual) : \Ext^1(f^* L_{A/B},I) \to \Ext^1(f^* E\hseq , I) \cong Ob(\phi,p)\otimes_k I$. We write $H^1(\phi\dual)(w(f,T',T))$ as $Ob(\phi,f,T',T)$.
\end{defi}

\begin{defi}[$\nu$-equivalence]\cite[Definition 2.6]{CL11semi}\label{v-equiv}
Let $A \to B$ be a morphism of Artin stacks. Let $\phi : E\hseq \to L_{A/B}$ and $\phi' : F\hseq \to L_{A/B}$ be perfect obstruction theories on $A/B$ and assume that there is an isomorphism $\psi : h^1(E\cseq) \to h^1(F\cseq)$. Then, we say that two perfect obstruction theories $\phi$ and $\phi'$ are $\nu$-equivalent via $\psi$ if $\psi|_p\otimes_k I : Ob(\phi,p)\otimes_k I \to Ob(\phi',p)\otimes_k I$ sends $Ob(\phi,f,T',T)$ to $Ob(\phi',f,T',T)$ for every closed point $p \in A$ and every infinitesimal lifting problem on $A/B$ at $p$.
\end{defi}

We introduce the following lemma from \cite{CL11semi}. Here we need some proof since we changed some definitions

\begin{lemm}\cite[Lemma 2.1]{CL11semi}
Let $E\cseq = [E^0 \to E^1]$ be a sequence of vector bundles a scheme $X$ and let $\EE = [E^1/E^0]$ be a vector bundle stack on $X$. Let $p : \EE \to h^1(E\cseq) = \coker(E^0 \to E^1)$ be a projection. Rewrite $h^1(E\cseq) = \cE$. Let $\WW \subset \EE$ be an integral substack of $\EE$. Then $p(\WW)$ is an integral substack of $\cE$. Therefore $p$ induces a homomorphism $p_* : Z_*(\EE) \to Z_*(\cE)$
\end{lemm}
\begin{proof}
Let $Z = \bar{\pi(\WW)}$ where $\pi : \EE \to X$ is the projection. Then $\WW \subset \EE|_Z$ and let $W = \WW \times_{\cE|_Z} E^1|_Z \subset E^1|_Z$. Then $Z$ is integral and $W$ is a saturated integral substack. For the projection $q : E^1 \surra h^1(E\cseq)$ We have the descent $q(W) \subset \cF|_Z$, which is integral substack. We observe that $p(\WW) = q(W)$. Therefore $p$ induces a homomorphism $p_* : Z_*(\EE) \to Z_*(\cE)$.
\end{proof}



\begin{defi}
Let $E\cseq = [E^0 \to E^1]$ be a sequence of vector bundles a scheme $X$ and let $\EE = [E^1/E^0]$ be a vector bundle stack on $X$. Let $\pi : E^1 \to \EE$ be a projection. Then there is an induced morphism $\pi_* : Z^{sat}_*(E^1) \to Z_*(\EE)$ which maps a saturated integral substack $W \subset E^1|_Z$ where $Z$ is an integral closed substack of $X$, to an integral substack $[W/E^0] \subset \EE|_Z$.
\end{defi}

The following lemma is clear by definition

\begin{lemm}\label{projectioncomp}
Let $E\cseq = [E^0 \to E^1]$ be a sequence of vector bundles a scheme $X$ and let $\EE = [E^1/E^0]$ be a vector bundle stack on $X$. Let $\pi : E^1 \to \EE$, $p : \EE \to h^1(E\cseq)$, $q : \EE^1 \to h^1(E\cseq)$ be projections. Then we have $q_* = p_* \circ \pi_*$.
\end{lemm}

Next we recall the following property of $\nu$-equivalent perfect obstruction theories.

\begin{prop}\cite[Proposition 2.1]{CL11semi}\label{semidescent1} 
Let $\phi : E\hseq \to L_{A/B}$ and $\phi' : F\hseq \to L_{A/B}$ be $\nu$-equivalent perfect obstruction theories, i.e. there is an isomorphism $\psi : h^1(E\cseq) \to h^1(F\cseq)$ which satisfies $\nu$-equivalence conditions.

Consider morphisms $\eta_{\phi} : \fN_{A/B}\hra\bdst{E\cseq} \to h^1(E\cseq)$, $\eta_{\phi'} : \fN_{A/B} \hra \bdst{F\cseq} \to h^1(F\cseq)$ where $\bdst{E\cseq} \to h^1(E\cseq)$ and $\bdst{F\cseq} \to h^1(F\cseq)$ are natural morphism from bundle stack to its coarse moduli spaces and $\fN_{A/B}$ is the intrinsic normal sheaf defined in \cite{BF97}. Let $[\WW] \subset \fN_{A/B}$ be an integral cycle. Then we have $\psi(\eta_{\phi}(\WW)) = \eta_{\phi'}(\WW)$.
\end{prop}


By the above proposition and remark, we have the following result.
\begin{defiprop}\cite[Lemma 3.1]{CL11semi}\label{semidescent2}
Let $\phi = (\{X_{\alpha}\}_{\alpha \in \cI}, \phi_{\alpha}, \psi_{\alpha \beta})$ be a semi-perfect obstruction theory over a DM type morphism $X \to Y$ where $X$ is a DM stack. Let $A \subset \fN_{X/Y}|_Z$ be an integral substack. Then, by the above proposition the collection $[A_{\alpha}] := (\eta_{\phi_{\alpha}})_*[A\times_X X_{\alpha}] \in Z_* Ob_{\phi_{\alpha}}$ satisfies the descent condition. Moreover we observe that it descend to an integral cycle $(\eta_{\phi})_*[A] = [\eta_{\phi}(A)] \in Z_*Ob_{\phi}$ where $\eta_{\phi}(A) \subset Ob_{\phi}|_Z$ is an integral substack.

Hence we define a morphism
\[
(\eta_{\phi})_* : Z_*\fN_{X/Y} \to Z_* Ob_{\phi}.
\]
\end{defiprop}

\begin{prop}\label{rationaldescent}
The morphism $\eta_{\phi}$ factors through rational equivalences. Hence it induces a morphism
\[
(\eta_{\phi})_* : A_*\fN_{X/Y} \to A_*Ob_{\phi}.
\]
\end{prop}
\begin{proof}
Let $[\WW]$ be an integral cycle of $\fN_{X/Y}$ and a rational function $h : \WW \dra k$ on $\WW$.
Let $\WW_{\alpha} := \WW \times_X X_{\alpha} \subset \fN_{X_{\alpha}/Y} \cong \vphi_{\alpha}^* \fN_{X/Y}$. Then there is a induced rational function $h_{\alpha} : \WW_{\alpha} \dra k$.
Since $k$ has no automorphism group, $h_{\alpha}$ descends to a morphism $\bar{h}_{\alpha} : \eta_{\phi_{\alpha}}(\WW_{\alpha}) \dra k$. We can easily check that the collection $\{ \bar{h}_{\alpha} \}_{\alpha\in \cI}$ satisfies descent conditions, hence they descend to a rational morphism $\bar{h} : \eta_{\phi}(\WW) \dra k$.

We will check $(\eta_{\phi})_* \rou h = \rou \bar{h}$ locally first. For each \'etale chart $\vphi_{\alpha} : X_{\alpha} \to X$, we may assume that there is a perfect obstruction theory $(E_{\alpha})\hseq$ quasi-isomorphic to a 2-term complex $[(E_{\alpha})_{-1} \to (E_{\alpha})_{0}]$. Then we have $(E_{\alpha})\cseq \stackrel{qis}{\cong} [(E_{\alpha})^{0} \to (E_{\alpha})^{1}]$, where $(E_{\alpha})^{0} = (E_{\alpha})_{0} \dual$ and $(E_{\alpha})^{1} = (E_{\alpha})_{-1}\dual$.
 
Let $[\WW_{\alpha}] = \sum_i a_i [\WW_{\alpha i}]$ where $\WW_{\alpha i}$ are integral substacks. 
Let $W_{\alpha} := \WW_{\alpha} \times_{\vphi_{\alpha}^*\fN_{X/Y}} (E_{\alpha})^1$, then $[W_{\alpha}] = \sum_i a_i[W_{\alpha i}]$ where $W_{\alpha i} = \WW_{\alpha i}\times_{\vphi_{\alpha}^* \fN_{X/Y}}(E_{\alpha})^1$ is an integral substack.
 and $\wtil{h}_{\alpha i} : W_{\alpha i} \dra k$ be the induced rational function. Let $\rou \wtil{h}_{\alpha i} = \sum_j b_{ij} [W_{\alpha ij}]$.
 
Then we have $\rou h_{\alpha i} = \sum_j b_{ij}[W_{\alpha ij}/(E_{\alpha})^0]$. We define $\rou h_{\alpha} := \sum_i a_i \rou h_{\alpha i}$. Let $q : (E_{\alpha})^1 \to Ob_{\phi_{\alpha}}  = \coker{(E_{\alpha})^0 \to (E_{\alpha})^1}$ be the surjection. Then we have $\eta_{\phi_{\alpha}}(\rou h_{\alpha}) = \eta_{\phi_{\alpha}}(\sum_i a_i \rou h_{\alpha i}) = \eta_{\phi_{\alpha}}(\sum_i \sum_j [W_{\alpha ij}/(E_{\alpha})^0]) = \sum_i \sum_j a_i b_{ij} [q(W_{\alpha ij})]$. Let $\bar{h}_{\alpha i} : q(W_{\alpha i}) \to k$ be an induced rational function. Note that $\eta_{\phi_{\alpha}}(\WW_{\alpha i}) = q(W_{\alpha i})$. Then we have $\rou \bar{h}_{\alpha i} = \sum_j b_{ij}[q(W_{\alpha ij})]$ by definition. Let $\rou \bar{h}_{\alpha} := \sum_i a_i \rou\bar{h}_{\alpha i}$. Then we have $(\eta_{\phi_{\alpha}})_* \rou h_{\alpha} = \rou \bar{h}_{\alpha}$.
It is clear that $\vphi_{\alpha}^* \rou h = \rou h_{\alpha}$. Therefore, it is enough to show that $\rou \bar{h}_{\alpha}$ descend to $\rou \bar{h}$. 

Let $Z = \bar{\pi(\WW)}$ where $\pi : \fN_{X/Y} \to X$ is the projection. Choose a modified proper representative $(f : S \to Z, \cV \surra Ob_{\phi}|_S, W_{\cA} \subset \cV)$ of an integral cycle $\cA:=\eta_{\phi}(W) \subset \Ob_{\phi}$. Let $S_{\alpha} := S \times_X X_{\alpha}$ and let $(W_{\cA})_{\alpha} := W_{\cA}\times_S S_{\alpha}$. Let $[(W_{\cA})_{\alpha}] = \sum_k c_k (W_{\cA})_{\alpha k}$. By abuse of notation, we call the surjections $\cV \surra Ob_{\phi}|_{S}, \vphi_{\alpha}^* \cV \surra Ob_{\phi_{\alpha}}|_{S_{\alpha}}$ by $q$. By Remark \ref{pullbackcompat1}, we have $q_*\vphi_{\alpha}^*[W_{\cA}] = \sum_k c_k [q((W_{\cA})_{\alpha k})] = \vphi_{\alpha}^* [q(W_{\cA})] = \vphi_{\alpha}^* q_* [W_{\cA}]$.

By construction, $d = \deg(q(W_{\cA}) \to \cA)$ is finite. Then we have $(\zeta_S)_* [W_{\cA}] = f_* q_*[W_{\cA}] = d[\cA]$. Let $[\cA_{\alpha}] := \vphi_{\alpha}^* [\cA]$.
Hence by Remark \ref{pullbackcompat1}, we have 
\begin{align*}
& [\cA_{\alpha}] = \vphi_{\alpha}^* [\cA] = \frac{1}{d}\vphi_{\alpha}^* (\zeta_S)_* [W_{\cA}] = \frac{1}{d} (\zeta_{S_{\alpha}})_* \vphi_{\alpha}^* [W_{\cA}] = \frac{1}{d} f_* q_* [(W_{\cA})_{\alpha}] \\ 
& = \frac{1}{d} f_* ( \sum_k c_k [q((W_{\cA})_{\alpha k})] ).
\end{align*}
We have $[\cA_{\alpha}] = \eta_{\phi_{\alpha}}(\sum_i a_i[\WW_{\alpha i}])$. So we can write $[\cA_{\alpha}] = \sum_i a_i [\cA_{\alpha i}]$ where $[\cA_{\alpha i}] = (\eta_{\phi_{\alpha}})_*[\WW_{\alpha i}]$, which is an integral cycle. Therefore, for each $k$, there exist some $i(k)$ such that $f(q((W_{\cA})_{\alpha k}))) = \cA_{\alpha i(k)}$. Let $d_{\alpha k} := \deg ( f(q((W_{\cA})_{\alpha k}))) \to \cA_{\alpha i(k)} )$, which is finite. Let $\wtil{h}' : W_{\cA} \dra k$, $\wtil{h}'_{\alpha k} : (W_{\cA})_{\alpha k} \dra k$ be induced rational functions.

Let $S_{W_{\alpha k}} = \bar{\pi( (W_{\cA})_{\alpha k} )}$ where $\pi : \cV|_{S_{\alpha}} \to S_{\alpha}$ is the projection and let $Z_{\alpha k} = f(S_{W_{\alpha k}}) \subset X_{\alpha}$.

Then $(f : S_{W_{\alpha k}} \to Z_{\alpha k}, \cV|_{S_{W_{\alpha k}}} \surra Ob|_{S_{W_{\alpha k}}}, (W_{\cA})_{\alpha k} )$ is a modified proper  representative of $\cA_{\alpha i(k)}$. Then we have $d_{\alpha k} \cdot \rou \bar{h}_{\alpha i(k)} = (\zeta_{S_{\alpha}})_* [\rou \wtil{h}'_{\alpha k}]$.
Let $\rou \wtil{h}'_{\alpha} := \sum_k c_k \rou \wtil{h}_{\alpha k}$. Then we have $(\zeta_{S_{\alpha}})_* (\rou \wtil{h}_{\alpha k})$. 
Let $\rou \wtil{h}'_{\alpha} := \sum_k c_k \rou \wtil{h}'_{\alpha k}$. Then we have $(\zeta_{S_{\alpha}})_*(\rou \wtil{h}'_{\alpha}) = \sum_k c_k (\zeta_{S_{\alpha}})_* [\rou \wtil{h}'_{\alpha k}] = \sum_k c_k d_{\alpha k} \rou \bar{h}_{\alpha i(k)}$. Since we have $\sum_i a_i [\cA_{\alpha i}] = \frac{1}{d} f_* ( \sum_k c_k [q((W_{\cA})_{\alpha k})] ) = \sum_k c_k d_{\alpha k} [\cA_{\alpha i(k)}]$. Therefore we have $(\zeta_{S_{\alpha}})_* (\rou \wtil{h}'_{\alpha} ) = \sum_k c_k d_{\alpha k} \rou \bar{h}_{\alpha i(k)} = d \cdot \rou \bar{h}_{\alpha}$.

Since $\vphi_{\alpha}^* \rou \wtil{h}' = \rou\wtil{h}'_{\alpha}$  we conclude that the collection $\{\rou\wtil{h}'_{\alpha}\}_{\alpha}$ descend to $\rou \wtil{h}'$. Note that projections compatible with flat pull-backs. Therefore we conclude that the collection $\{ \frac{1}{d}(\zeta_{S_{\alpha}})_* (\rou \wtil{h}'_{\alpha} ) \}_{\alpha} = \{ \rou \bar{h}_{\alpha} \}_{\alpha}$ descend to an element $ \frac{1}{d} (\zeta_{S})_* (\rou \wtil{h}' )$, which is equal to $\rou \bar{h}$ by definition.
\end{proof}


\begin{rema} For a semi-perfect obstruction theory $\phi$ over $X/Y$, the obstruction sheaf $Ob(\phi)$ is naturally defined as a coherent $\cO_{X_{\textrm{\'et}}}$-module where $X_{\textrm{\'et}}$ is the topos over the small \'etale site on $X$. However, by arguments in \cite[Section 2]{BF97}, we can naturally see $Ob(\phi)$ as a coherent $\cO_{X}$-module where $X$ is the topos over the big \'etale site on $X$.
\end{rema}



\subsection{Semi-virtual pull-backs}\label{sec:semivirtpullback-subsection}
In this subsection, we define a notion of semi-virtual pull-backs which is a generalization of virtual pull-backs introduced in \cite{Man08} when we have a relative semi-perfect obstruction theory. We prove various functorial properties properties for Semi-virtual pull-backs and as a result, we show that Semi-virtual pull-backs defined as a bivariant class.

Let $u : X \to Y$ be a DM-type morphism where $X$ is a DM stack and $Y$ is a DM stack or smooth Artin stack and $L_{X/Y}$ be the cotangent complex over $X \to Y$. Then, we know that $C_{X/Y}$ is embedded in $\bdst{L_{X/Y}\dual}=\fN_{X/Y}$ as a closed substack by \cite{BF97}. 
Assume further that there is a semi-perfect obstruction theory $\phi = (\{X_{\alpha}\}_{\alpha \in \cI}, \phi_{\alpha}, \psi_{\alpha \beta})$ over $X/Y$ with an obstruction sheaf $Ob_{\phi} = \cE$. 
We define a semi-virtual pull-back in the following.

\begin{defi}[Semi-virtual pull-backs]\label{def:semivert}
Consider a fiber diagram of Artin stacks of finite type over $k$ :
\[\xymatrix{Z \ar[r]^v \ar[d]_f \ar@{}[rd]|{\square} & W \ar[d] \\
X \ar[r]^u & Y.}
\]
where $Z$ is a DM stack. Then we define a semi-virtual pull-back $u^!_{\cE} : A_*(W)_{\QQ} \to A_*(Z)_{\QQ}$ to be the composition of the following morphisms :
\begin{multline*}
A_*(W)_{\QQ} \stackrel{\sigma}{\lra} A_*(C_{Z/W})_{\QQ} \stackrel{i_*}{\lra} A_*(f^*(C_{X/Y}))_{\QQ} \\ \stackrel{\iota_*}{\lra} A_*(f^*\fN_{X/Y})_{\QQ} \stackrel{(\eta_{\phi})_*}{\lra} A_*(f^*\cE)_{\QQ} \stackrel{\Gysin{f^*\cE}}{\lra} A_*(Z)_{\QQ}
\end{multline*}
where the morphism $\sigma$ is induced from deformation to a normal cone, which is defined in \cite[Construction 3.6]{Man08}, the morphism $i_*$ is induced from the closed embedding of cones $i : C_{Z/W} \hra f^*(C_{X/Y})$ constructed in \cite[Proposition 2.26]{Man08} and $\iota : f^*C_{X/Y} \hra f^*\fN_{X/Y}$ is the natural closed embedding.
\end{defi}

\noindent More explicitly, for an integral substack $A \subset W$ and $B := A\times_W Z$, we have $u^!_{\cE}[A] = \Gysin{f^*\cE}[(\eta_{\phi})_*(\iota_*i_*[C_{B/A}])]$.

Next, we introduce some basic formulas on intersection theory of bundle stacks and coherent sheaf stacks.

\begin{lemm}\label{bdstpushforward}
Let $E\cseq = [E^0 \to E^1]$ be a sequence of vector bundles on DM stack $Y$ and let $\EE = [E^1/E^0]$ be a vector bundle stack. Let $p : \EE \to h^1(E\cseq) = \coker(E^0 \to E^1)$ be a projection. Rewrite $h^1(E\cseq) = \cE$. Then there is an induced homomorphism $p : Z_*(\EE) \to Z_*(\cE)$. Let $f : X \to Y$ be a proper morphism of DM stacks. We also have induced morphisms $p : Z_*(f^*\EE) \to Z_*(f^*\cE)$ and $f_* : Z_*(f^*\EE) \to Z_*(\EE)$, $f_* : Z_*(f^* \cE) \to Z_*(\cE)$. Then we have $f_* \circ p_* = p_* \circ f_*$. 
\end{lemm}
\begin{proof}
Let $\WW \subset f^* \EE$ be an integral substack. Let $W := f^* E^1 \times_{f^* \EE} \WW$. Then we have $f(W) = E^1 \times_{\EE} f(\WW)$. Consider the following diagram:
\[
\xymatrix{
W \ar[r]^f \ar[d]^p \ar@{}[dr]|{\square} & f(W) \ar[d]^p \\
\WW \ar[r]^f & f(\WW)
}
\]
where the vertical arrows are vector bundles since $E^1 \to \EE$ and $f^*E^1 \to f^*\EE$ are vector bundles on $\EE$ and $f^*\EE$ each. Therefore, $\deg(W \to f(W)) = \deg(\WW \to f(\WW))$.

Let $Y_W := \bar{\pi(W)}$ where $\pi : E^1 \to Y$ is the projection. Choose an open dense subset $U \subset Y_W$ such that $\cE|_U$ is locally free. Then $f^* E^1|_{f^{-1}(U)} \to f^* \cE|_{f^{-1}(U)}$ and $E^1|_{U} \to \cE|_{U}$ are vector bundles. Then we can consider the following fiber diagram:
\[
\xymatrix{
W|_{f^{-1}(U)} \ar[r]^f \ar[d]^p \ar@{}[dr]|{\square} & f(W)|_U \ar[d]^p \\
p(\WW)|_{f^{-1}(U)} \ar[r]^f & p(f(\WW))|_{U}
}.
\]
Let $q : E^1 \surra \cE$ be the surjection. Then we have $p(\WW) = q(W)$. Therefore horizontal arrows are vector bundle projections, so we have $\deg(W \to f(W)) = \deg(p(\WW) \to p(f(\WW)))$. Therefore we have $\deg(\WW \to f(\WW)) = \deg(p(\WW) \to p(f(\WW)))$. Then we have $f_* p_*[Z] = \deg(p(\WW) \to p(f(\WW))) \cdot [p(f(Z)))] = \deg(\WW \to f(Z))[p(f(Z)))] = p_* f_* [Z]$.
\end{proof}

\begin{lemm}
Let $E\cseq = [E^0 \to E^1]$ be a sequence of vector bundles on DM stack $Y$ and let $\EE = [E^1/E^0]$ be a vector bundle stack. Let $p : \EE \to h^1(E\cseq) = \coker(E^0 \to E^1)$ be a projection. Rewrite $h^1(E\cseq) = \cE$. Then there is an induced homomorphism $p : Z_*(\EE) \to Z_*(\cE)$. Let $f : X \to Y$ be a flat morphism of DM stacks. We also have induced morphisms $p : Z_*(\EE|_X) \to Z_*(\cE|_X)$ and $f^* : Z_*(\EE) \to Z_*(\EE|_X)$, $f^* : Z_*(\cE) \to Z_*(\cE|_X)$. Then we have $f^* \circ p_* = p_* \circ f^*$.
\end{lemm}\label{bdstpullback}
\begin{proof}
Let $\WW \subset \EE$ be an integral substack. Let $W = E^1 \times_{\EE} \WW$. Let $q : E^1 \surra \cE$ be the surjection. Then we have $p(\WW) = q(W)$.
Then we have $f^*\circ p_*[\WW] = f^* q_*[W] = q_*f^*[W]$.

Let $r : E^1 \to \EE$ be the projection. Then we have $W = r^*[\WW]$. Thus we obtain 
\[
q_* f^*[W] = q_* r^*(f^*[\WW]) = p_* \circ r_* \circ r^*( f^*[\WW]) = p_* (f^*[\WW])
\]
where the last equality comes from the fact that $r : E^1 \to \EE$ is a projection of a vector bundle.





\end{proof}

The followings statements from \cite{Man08} used for cycle computation of cones. The next lemma comes from the proof of \cite[Theorem 4.1]{Man08}

\begin{lemm}\label{conepushforward1}
Consider a fiber diagram of Artin stacks, finite type over $k$ :
\[\xymatrix@=30pt{A \ar[r] \ar[d] \ar@{}[dr]|{\square} & C \ar@{->>}[d]^u \\
B \ar[r]^f & D
}
\]
where $u$ is a surjective and either a proper morphism of DM stacks or a projective morphism of Artin stacks, $C$ and $D$ are irreducible, and $f$ is a DM-type morphism. Then, $u_*[C_{A/C}] = \deg(u)[C_{B/D}] \in Z_*(C_{B/D})$.
\end{lemm}

\begin{lemm}\cite[Proposition 2.26]{Man08}\label{conepullback1}
Consider a fiber diagram of Artin stacks, finite type over $k$ : 
\[ \xymatrix@=30pt{ A \ar[r] \ar[d]_v \ar@{}[dr]|{\square} & C \ar[d]^u \\
B \ar[r]^f & D
}
\]
where $f$ is a morphism of DM-types and $u$ is a flat morphism. Then, there is a closed embedding $\iota : C_{A/C} \hra v^* C_{B/D}$. Furthermore, if $u$ is flat, then $\iota$ is an isomorphism.
\end{lemm}

Next we consider a semi-perfection theory over $B/D$ and generalize above lemmas \ref{conepushforward1}, \ref{conepullback1}.

\begin{lemm}\label{conepushforward2}
Consider a fiber diagram of Artin stacks, finite type over $k$ : 
\[\xymatrix@=30pt{ A\ar[r] \ar[d]_v & C \ar@{->>}[d]^u \\
B \ar[r]^f \ar@{}[ur]|{\square} & D
}
\]
where $u$ is surjective, $u$ is either a proper morphism of DM stacks or a projective morphism of Artin stacks, $C$ and $D$ are irreducible, and $f$ is a DM-type morphism. We further assume that there is a semi-perfect obstruction theory $\phi = (\{\vphi_{\beta} : B_{\beta} \to B \}_{\beta \in \cI}, \phi_{\beta}, \psi_{\beta \gamma})$ over $B/D$ with an obstruction sheaf $Ob_{\phi} = \cE$. Then, the intrinsic normal sheaf $\bs{c}_{B/D}$ is a closed substack of $\cE$ by \cite{Beh09}.

For $v_* : Z_*(\cE|_A) \to Z_*(\cE)$, we have $v_*(\eta_{\phi})_*[C_{A/C}] = \deg(u)(\eta_{\phi})_*[C_{B/D}] \in Z_*(\cE)$.

\end{lemm}
\begin{proof}
Let $A_{\beta} := A\times_B B_{\beta}$ and let $\vphi_{\beta} : A_{\beta} \to A$ be induced  \'etale morphisms.
By Lemma \ref{bdstpushforward}, We have $v_* (\eta_{\phi_{\beta}})_* \vphi_{\beta}^*[C_{A/B}] = (\eta_{\phi_{\beta}})_* \vphi_{\beta}^* v_*[C_{A/B}] = \deg(u) \cdot (\eta_{\phi_{\beta}})_* \vphi_{\beta}^*  [C_{B/D}]$. Therefore we conclude that $v_* (\eta_{\phi})_* [C_{A/C}]$ is the descent of the collection $\{ \deg(u) \cdot (\eta_{\phi_{\beta}})_* \vphi_{\beta}^*  [C_{B/D}] \}_{\beta}$, which is equal to $\deg(u) \cdot (\eta_{\phi})_*[C_{B/D}]$.
\end{proof}

\begin{lemm}\label{conepullback2}
Consider a fiber diagram of Artin stacks, finite type over $k$ :
\[\xymatrix@=30pt{ A \ar[r] \ar[d]_v \ar@{}[rd]|{\square} & C \ar[d]^u \\
B \ar[r]^f & D
}
\]
where $f$ is a morphism of DM-types and $u$ is a flat morphism. We further assume that there is a semi-perfect obstruction theory $\phi = (\{\vphi_{\beta} : B_{\beta} \to B \}_{\beta \in \cI}, \phi_{\beta}, \psi_{\beta \gamma})$ over $B/D$ with an obstruction sheaf $Ob_{\phi} = \cE$. 

Then, for $v^* : Z_*(\cE) \to Z_*(\cE|_A)$, we have $v^*(\eta_{\phi})_*[C_{B/D}] = (\eta_{\phi})_*[C_{A/C}]$.
\end{lemm}

\begin{proof}
Let $A_{\beta} := A\times_B B_{\beta}$ and let $\vphi_{\beta} : A_{\beta} \to A$ be induced  \'etale morphisms. By Lemma \ref{conepullback1} and \ref{bdstpullback}, we have $v^*  (\eta_{\phi_{\beta}})_* \vphi_{\beta}^* [C_{B/D}] =(\eta_{\phi_{\beta}})_* \vphi_{\beta}^* v^* [C_{B/D}] = (\eta_{\phi_{\beta}})_* \vphi_{\beta}^* [C_{A/C}]$. Thus we observe that $v^*(\eta_{\phi})_*[C_{B/D}]$ is the descent of the collection $\{ (\eta_{\phi_{\beta}})_* \vphi_{\beta}^* [C_{A/C}] \}_{\beta}$, which is equal to $(\eta_{\phi})_*[C_{A/C}]$.
\end{proof}

At last, we are ready to prove various functorial properties about semi-virtual pull-backs. First we prove that semi-virtual pull-backs commute with proper pushforwards and flat pull-backs. Consider the following diagram of Artin stacks :
\[\xymatrix{Z' \ar[r]^w \ar[d]_g \ar@{}[rd]|{\square} & W' \ar[d]^f \\
Z \ar[r]^v \ar[d] \ar@{}[rd]|{\square} & W \ar[d] \\
X \ar[r]^u & Y.
}
\]
Then, we have the following.

\begin{prop}\label{vpullbackcompatibility}
Let $\phi$ be a semi-perfect obstruction theory on $X/Y$ and $\cE = Ob_{\phi}$. Assume that $Z,Z'$ are DM-stacks and $u$ is a DM-type morphism. Then, we have the following.
\begin{itemize}
\item[(i)] If $f$ is flat, then we have :
\[ u^!_{\cE} \circ f^* = g^* \circ u^!_{\cE} : A_*(W)_{\QQ} \to A_*(Z')_{\QQ}.
\]
\item[(ii)] If $f$ is a proper morphism between DM stacks, then we have :
\[ u^!_{\cE} \circ f_* = g_* \circ u^!_{\cE} : A_*(W')_{\QQ} \to A_*(Z)_{\QQ}. 
\]
\end{itemize}
\end{prop}

\begin{proof}
(i) Let $A \subset W$ be an integral substack of $W$. Let $B:= Z\times_W A, B':=g^{-1}(B)$ and $A':=f^{-1}(A)$.
Then by Lemma \ref{conepullback1}, and Lemma \ref{Gysinpushforwardpullback}, we have
$g^*[C_{B/A}] = [C_{B'/A'}]$.
 
By Lemma \ref{conepullback2}, we obtain $g^* u^!_{\cE} [A] = g^* \Gysin{\cE}([ (\eta_{\phi})_*[C_{B/A}] ]) = \Gysin{\cE} \circ (\eta_{\phi})_* g^*[C_{B/A}] = \Gysin{\cE} \circ (\eta_{\phi})_* [C_{B'/A'}] = u^!_{\cE}[A'] = u^!_{\cE} \circ f^* [A]$.

(ii) Let $A' \subset W'$ be an integral substack of $W'$. Let $A := f(A')$ and let $d = \deg(f : A' \to A)$. Let $B:= Z\times_{W} A$, $B' := Z'\times_{W'} A' = B\times_{A} A'$. Then by Lemma \ref{conepushforward1}, we have $g_*[C_{B'/A'}] = d[C_{B/A}]$.

By Lemma \ref{conepushforward2} and Lemma \ref{Gysinpushforwardpullback}, we have $g_* u^!_{\cE} [A'] = g_* \Gysin{\cE}((\eta_{\phi})_*[C_{B'/A'}]) = d \cdot \Gysin{\cE}(  (\eta_{\phi})_* [C_{B/A}] ) = d \cdot u^!_{\cE}[A] = u^!_{\cE} \cdot f_* [A']$.
\end{proof}

We note that a perfect obstruction theory is a semi-perfect obstruction theory. Then, one may ask whether a virtual pull-back via perfect a obstruction theory which is defined in \cite{Man08}, and the semi-virtual pull-back via the semi-perfect obstruction theory are equal or not. The next proposition gives an affirmative answer to the above question.

Recall the following diagram of Artin stacks :
\[\xymatrix{Z \ar[r] \ar[d] \ar@{}[rd]|{\square} & W \ar[d] \\
X \ar[r]^u & Y
}
\]
where $X,Z$ are DM stacks and $u$ is a DM-type morphism. 
Assume that there is a perfect obstruction theory
$\phi : (E_{X/Y})\hseq \to L_{X/Y}$ with $Ob_{\phi} = \cE$ and $\bdst{E_{X/Y}\dual}=\EE$. Then, we have the following.

\begin{prop}\label{comparison} In the above setting, we have :
\[u^!_{\cE} = u^!_{\EE} : A_*(W)_{\QQ} \to A_*(Z)_{\QQ} 
\]
where $u^!_{\EE}$ is a virtual pull-back defined in \cite{Man08}
\end{prop}
\begin{proof}
Let $A \subset W$ be an integral substack and $B = A\times_Y X\subset Z$. Then, $u^!_{\EE}[A] = \Gysin{\EE}[C_{B/A}]$.

Let $[C_{B/A}] = \sum_i a_i[C_i]$ be the fundamental cycle, i.e. $C_i$ are integral substacks. Let $\bs{c}_i$ be a coarse moduli sheaf of $C_i$. Then, we have $(\eta_{\phi})_*[C_{B/A}] = \sum_i a_i(\eta_{\phi})_*[C_i]$. Therefore, it is enough to show that $\Gysin{\EE}[C_i] = \Gysin{\cE}(\eta_{\phi})_*[C_i]$.

Let $Z_i := \bar{\pi(C_i)}$ where $\pi : \EE \to X$ is the projection. By Remark \ref{modifiedproperexist}, we can choose a modified proper representative $(f : S_i \to Z_i, q : \cV_i \surra \cE|_{S_i}, W_i \subset \cV_i)$ of $\eta_{\phi}(C_i)$ such that $S_i$ are projective schemes. Since $S_i$ are projective, $f_i^*E_{X/Y} \dual$ has a two-term locally free resolution $[E_i^0 \to E_i^1]$ such that $E_i^1 = \cV_i$.

Let $r_i : E_i^1 \to f_i^*\EE$ be the projection.
Since $W_i \subset E_i^1$ is a saturated integral subscheme, $r_i(W_i) \subset f_i^*\EE$ is an integral closed substack. In a similar manner in the proof of Lemma \ref{bdstpushforward}, we get $\deg(r_i(W_i) \to C_i) = d_{W_i}$, which means that by a proper morphism $f_i : f_i^*\EE \to \EE$ we obtain $(f_i)_*[r_i(W_i)] = d_{W_i}[C_i]$. Thus, we have :
\[ \Gysin{\EE}[C_i] = \frac{1}{d_{W_i}}(f_i)_*\Gysin{f_i^*\EE}[r_i(W_i)] = \frac{1}{d_{W_i}}(f_i)_*\Gysin{E_i^1}[W_i] = \Gysin{\cE}(\eta_{\phi})_*[C_i].
\]
\end{proof}

\begin{lemm}\label{virtgysincomm}
Consider a following fiber diagram
\begin{align*}
\xymatrix{
D \ar[r] \ar[d] \ar@{}[rd]|{\square} & C \ar[r] \ar[d] \ar@{}[rd]|{\square} & Z \ar[d]^{0} \\
B \ar[r] \ar[d] \ar@{}[rd]|{\square} & A \ar[r] \ar[d] & E \\
X \ar[r]_{u} & Y &
} 
\end{align*}
such that $X,Y,E,Z,A,B,C,D$ are DM stacks, $u$ is a DM-type morphism, $E$ is a vector bundle over $Z$ and $0$ is the zero section. Assume that there is a semi-perfect obstruction theory $\phi$ on $X/Z$. Let $\phi = \{ \{X_{\alpha}\}_{\alpha \in \cI}, \phi_{\alpha}, Ob_{\phi} =: \cE$. Then we have
\[
u^!_{\cE} \circ \Gysin{E} = \Gysin{E} \circ u^!_{\cE}.
\]
\end{lemm}
\begin{proof}
Let $A' \subset A$ be an integral substack. We can write $\Gysin{E}[A'] = \sum_i a_i[C_i]$ where $C_i$ are integral substacks of $C$. Let $D_i := C_i\times_{C} D$. Then we have
\[
u^!_{\cE} \circ \Gysin{E} [A'] = \sum_ j a_{j} u^!_{\cE} [C_i] = \sum_i a_i \Gysin{\cE|_D} (\eta_{\phi})_*[C_{ D_i / C_i }]
\]
On the other hand, let $B' := A'\times_A B$. Then we have 
\[
\Gysin{E|_D} \circ u^!_{\cE}[A']= \Gysin{E|_D} \circ \Gysin{\cE|_B}(\eta_{\phi})_*[C_{B'/A'}].
\]

Therefore, it is enough to show that $\sum_i a_i \cdot \pi^* \circ \Gysin{\cE|_D} (\eta_{\phi})_* [C_{D_i/C_i}] = \Gysin{\cE|_B}(\eta_{\phi})_*[C_{B'/A'}]$ where $\pi : E|_D \to D$ is the projection. By Lemma \ref{conepullback2}, we have $\pi^* \Gysin{\cE|_D} (\eta_{\phi})_* [C_{D_i/C_i}] = \Gysin{\cE|_B} (\eta_{\phi})_* [C_{\pi^*(D_i)/\pi^*(C_i)}]. $

Since $\Gysin{E}[A'] = \sum_i a_i[C_i]$, we have $[A'] = \sum_i a_i [\pi^*(C_i)]$ in $A_*(A)$. Hence we have 
\begin{align*}
& \sum_i a_i \cdot \pi^* \circ \Gysin{\cE|_D} (\eta_{\phi})_* [C_{D_i/C_i}] = \sum_i a_i \Gysin{\cE|_B}(\eta_{\phi})_*[C_{\pi^*(D_i)/\pi^*(C_i)}] \\
& = \Gysin{\cE|_B} (\eta_{\phi})_*[C_{B'/A'}].
\end{align*}
\end{proof}

Next, we prove a functoriality of semi-virtual pull-back, which is an analogue of \cite[Theorem 4.8]{Man08}. Consider the morphism of Artin stacks :
\[ \xymatrix{X \ar[r]_u & Y \ar[r]_v & Z}.
\]
such that $X,Y$ are DM stacks and $u,v$ are DM-type morphisms.
Assume that there are semi-perfect obstruction theories $\phi$ on $X/Z$, $\phi'$ on $Y/Z$, and $\phi''$ on $X/Y$. We say that $\phi, \phi', \phi''$ are compatible in the following situations.

Let $\phi' = \{ \{Y_{\alpha}\}_{\alpha \in \cI}, \phi'_{\alpha}, \psi'_{\alpha \beta}  \}$. Then, $\phi = \{ \{X_{\alpha}\}_{\alpha \in \cI}, \phi_{\alpha}, \psi_{\alpha \beta} \}$ and $\phi''=\{\{X_{\alpha}\}_{\alpha \in I}, \phi_{\alpha}'', \psi''_{\alpha \beta}\}$ where $X_{\alpha} := X \times _Y Y_{\alpha}$.

Let $\phi_{\alpha} : ((E_{X/Z})_{\alpha})\hseq \to L_{X_{\alpha}/Z}$, $\phi'_{\alpha} : ((E_{Y/Z})_{\alpha})\hseq \to L_{Y_{\alpha}/Z}$, $\phi''_{\alpha} : (E_{X/Y})\hseq \to L_{X/Y}$. Then, there is a morphism of distinguished triangles :

\[\xymatrix{u^*((E_{Y/Z})_{\alpha})\hseq \ar[d] \ar[r]^a & ((E_{X/Z})_{\alpha})\hseq \ar[r]^b \ar[d] & ((E_{X/Y})_{\alpha})\hseq \ar[r]^c \ar[d] & u^*((E_{Y/Z})_{\alpha})\hseq[1] \ar[d] \\
u^*L_{Y_{\alpha}/Z} \ar[r]^{can} & L_{X_{\alpha}/Z} \ar[r] & L_{X_{\alpha}/Y}\cong L_{X_{\alpha}/Y_{\alpha}} \ar[r] & u^*L_{Y_{\alpha}/Z}[1]
}
\]
where $(E_{X/Y}):=E_{X/Y}|_{X_{\alpha}}$. The isomorphism $L_{X_{\alpha}/Y}\cong L_{X_{\alpha}/Y_{\alpha}}$ follows from the fact that $L_{Y_{\alpha}/Y} \cong \Omega^1_{X_{\alpha}/X}=0$ since $Y_{\alpha} \to Y$ is an \'etale morphism.
Furthermore, the morphisms in above diagrams are compatible with transition isomorphisms $\psi_{\alpha \beta}, \psi'_{\alpha \beta}$. Then, we say that $\phi, \phi', \phi''$ are compatible.

Let $\cE = \Ob_{\phi},\cF=Ob_{\phi'},\cG=Ob_{\phi''}$. Then, there is an exact sequence $\cG \to \cE \to u^*\cF \to 0$. Then we have the following.

\begin{prop}\label{functoriality}
Consider the following fiber diagram of Artin stacks :
\[\xymatrix{U \ar[r] \ar@{}[rd]|{\square} \ar[d] & V \ar[r] \ar@{}[rd]|{\square} \ar[d] & W \ar[d] \\
X \ar[r]^u & Y \ar[r]^v & Z
}
\]
such that $X,Y,U,V$ are DM stacks and $u,v$ are DM-type morphisms, $W,Z$ are smooth Artin stacks or DM stacks.
Assume that there are semi-perfect obstruction theories $\phi$ on $X/Z$, $\phi'$ on $Y/Z$ and $\phi''$ on $X/Y$ where $\phi$, $\phi'$, $\phi''$ are compatible. We set $Ob_{\phi}=:\cE,Ob_{\phi'}=:\cF,Ob_{\phi''}=:\cG$. Then, we have :
\[ u^!_{\cG} \circ v^!_{\cF} = (v\circ u)^!_{\cE} : A_*(Z)_{\QQ} \to A_*(X)_{\QQ}.
\]
\end{prop}
\begin{proof}
We first consider the double deformation space $M^{\circ}_{X\times\PP^1/M^{\circ}_{Y/Z}} \to \PP^1 \times \PP^1$ constructed in \cite{KKP03}. Its fibers over $\{0\} \times \PP^1$, $\{0\} \times \{0\}$ and $\{0\} \times \{\infty\}$ are $C_{X\times\PP^1/M^{\circ}_{Y/Z}}$, $C_{X/C_{Y/Z}}$ and $C_{X/Z}$ respectively. Again by \cite{KKP03}, we obtain $\fN_{X\times\PP^1/M^{\circ}_{Y/Z}} = \bdst{cone(g)\dual}$ where $g$ is defined by :
\[ g = (id\boxtimes U, can\boxtimes T) : u^*L_{Y/Z} \boxtimes \cO_{\PP^1}(-1) \to u^*L_{Y/Z} \boxtimes \cO_{\PP^1} \oplus L_{X/Z}\boxtimes \cO_{\PP^1}
\]
where $U$, $T$ are global sections of $\cO_{\PP^1}(1)$ corresponding to $\{0\}\subset \PP^1$ and $\{\infty \} \subset \PP^1$ and $can : u^*L_{Y/Z} \to L_{X/Z}$ is a morphism induced from the distinguished triangle. Consider a morphism :
\[ h_{\alpha} : u^*((E_{Y/Z})_{\alpha})\hseq \boxtimes \cO_{\PP^1}(-1) \to u^*((E_{Y/Z})_{\alpha})\hseq \boxtimes \cO_{\PP^1} \oplus ((E_{X/Z})_{\alpha})\hseq \boxtimes \cO_{\PP^1}.
\]
We note that $cone(g_{\alpha}) = L_{X\times \PP^1/M^{\circ}_{Y/Z}}$, where $g_{\alpha} = g|_{Y_{\alpha}}$ by \cite{KKP03}. We also claim that an induced morphism of distinguished triangles $\phi_{\alpha} : cone(h_{\alpha}) \to cone(g_{\alpha}) = L_{X\times \PP^1/M^{\circ}_{Y/Z}}$, and transition morphisms $\psi'''_{\alpha\beta}$ form a semi-perfect obstruction theory called $\phi'''$. We will prove this later.

Let $Ob_{\phi'''} = \cH$. Then, $\cH_0 := \cH|_{X \times \{0\}} = \cE$, $\cH_{\infty} := \cH|_{X\times \{\infty\}} = u^*\cF \oplus \cG$. Let $D \subset W$ be an integral substack and let $B = D \times _Z Y$, $A = D\times_Z X$. Then, it is enough to show that $(v\circ u)^!_{\cE}[D] = \Gysin{\cE}(\eta_{\phi})_*[C_{A/D}] = u^!_{\cG}\circ v^!_{\cF}[D] = u^!_{\cG}\Gysin{\cF}(\eta_{\phi'})_*[C_{B/D}].$

Consider natural projection $h : C_{A\times \PP^1/M^{\circ}_{B/D}} \to \PP^1$. Let $[C_{A\times \PP^1/M^{\circ}_{B/D}}] = \sum_i a_i [C_i]$ be a fundamental cycle where $C_i$ are integral. Then there are induced rational morphisms $h_i : C_i \dra k$. Then, by the proof of Proposition \ref{rationaldescent}, $h_i$ descend to a rational function $\bar{h}_i : \eta_{\phi'''}(C_i) \dra k$. Let $K_i := \bar{\pi(C_i)}$ where $\pi : \fN_{X\times\PP^1/M^{\circ}_{Y/Z}} \to X\times\PP^1/M^{\circ}_{Y/Z}$ is the projection. Then $\eta_{\phi'''}(C_i) \subset Ob_{\phi'''}|_{K_i}$ are integral integral substacks.

Since $C_{A\times\PP^1/M^{\circ}_{B/D}}$ is flat over $\PP^1$, all integral components of $C_{A\times\PP^1/M^{\circ}_{B/D}}$ are closures of irreducible components of $C_{A\times(\PP^1 \setminus \{\infty\})/D\times(\PP^1 \setminus \{\infty\})} \cong C_{A/D} \times (\PP^1 \setminus \{\infty\})$. Therefore, each $K_i$ is of the form $U_i \times \PP^1 \subset U\times \PP^1$.
Thus, we can choose modified proper representatives of the form $(f_i : S_i \times \PP^1 \to U_i \times \PP^1, \cV_i \surra \cH|_{S_i\times \PP^1}, W_i)$ for each $\eta_{\phi'''}(C_i)$ and therefore we have $i( [(\eta_{\phi'''}(C_i), h_i)] ) = [(f_i : S_i \times \PP^1 \to U_i \times \PP^1, \cV_i \surra \cH|_{S_i\times \PP^1}, W_i, \wtil{h}_i)] =: H_i \in W^E_*(Ob_{\phi'''})$ where $\wtil{h}_i : W_i \dra k$ are induced rational morphisms.

Let $(\cV_i)_0 := \cV_i|_{S_i \times \{0\}}$, $(\cV_i)_{\infty}:=\cV_i|_{S_i \times \{\infty\}}, (Z_i)_0 := Z_i|_{S_i \times \{0\}}, (Z_i)_{\infty} := Z_i|_{S_i \times \{\infty\}}$. Then, we obtain :
\[ (f_i)_* \Gysin{(\cV)_0}[\rou_0(W_i,\wtil{h}_i)] = (f_i)_*\Gysin{(\cV_i)_{\infty}}[\rou_{\infty}(W_i,\wtil{h}_i)].
\]


We note that $(\cH_i)_0 = \cE$, $(\cH_i)_{\infty} = u^*\cF \oplus \cG$. Let $\rou_0(\eta_{\phi'''}(C_i),\bar{h}_i) := \frac{1}{d_{W_i}} (\zeta_{S_i \times \{0\} })_*[\rou_0(\eta_{\phi'''}(C_i), \bar{h}_i)]$ and $\rou_{\infty}(\eta_{\phi'''}(C_i), \bar{h}_i) := \frac{1}{d_{W_i}} (\zeta_{S_i \times \{ \infty \} })_*[\rou_{\infty}(\eta_{\phi'''}(C_i), \bar{h}_i)]$.
By simple calculation, we can check that
\[
\Gysin{\cE|_U}[\rou_0(\eta_{\phi'''}(C_i),\bar{h}_i)]=\Gysin{\cF|_U\oplus \cG|_U}[\rou_{\infty}(\eta_{\phi'''}(C_i),\bar{h}_i)].
\]


We have $\rou_0(C_{A\times \PP^1/M^{\circ}_{B/D}}, h) = [C_{A/D}]$ and $\rou_{\infty}(C_{A\times \PP^1/M^{\circ}_{B/D}}, h) = [C_{A/C_{B/D}}]$.
Then, in a similar manner as in Proposition \ref{rationaldescent}, we can show that 
\begin{align*}
& \rou_0 \left( \sum_i a_i ( (\eta_{\phi'''})(C_i),\bar{h}_i ) \right) = (\eta_{\phi})_*( \rou_0(C_{A\times \PP^1/M^{\circ}_{B/D}}, h) ).
\end{align*}
Hence we have $\rou_0 \left( \sum_i a_i ( \eta_{\phi'''}(C_i),\bar{h}_i ) \right) = (\eta_{\phi'''})_*[C_{A/D}]$.
In a same manner, we can show that $\rou_{\infty} \left( \sum_i a_i ( \eta_{\phi'''}(C_i),\bar{h}_i ) \right) = (\eta_{u^*\phi'} \oplus \eta_{\phi''})_*[C_{A/C_{B/D}}]$.


Therefore, we get $u^!_{\cE}[D] = \Gysin{\cE|_U}(\eta_{\phi})_*[C_{A/D}] = \Gysin{\cF|_U \oplus \cG|_U}(\eta_{u^*\phi'} \oplus \eta_{\phi''})_*[C_{A/C_{B/D}}].$ Thus, it is enough to show that $\Gysin{\cF|_U\oplus \cG|_U}(\eta_{u^*\phi'} \oplus \eta_{\phi''})_*[C_{A/C_{B/D}}] = u^!_{\cG} \circ v^!_{\cF} [D].$ Therefore, we obtain the proof by the following lemmas, Lemma \ref{keylemma1} and Lemma \ref{keylemma2}.

\end{proof}

\begin{lemm}\label{keylemma1}
$\phi'''$ defined above is indeed a semi-perfect obstruction theory.
\end{lemm}
\begin{proof} Everything is clear except for a $\nu$-equivalence condition. For $\alpha,\beta \in \cI$, $p \in (X_{\alpha}\times_X X_{\beta})\times \PP^1$, it is enough to show that $\psi'''_{\alpha \beta}$ sends $Ob(\phi'''_{\alpha}|_{X_{\alpha} \times_X X_{\beta}},T',T)$ to $Ob(\phi'''_{\beta}|_{X_{\alpha}\times_X X_{\beta}}, T', T)$.

If $p$ lies on $\PP^1\setminus \{\infty\}$, we can write $p = (p',t)$ where $p' \in X_{\alpha}\times_X X_{\beta},t\in \PP^1\setminus \{\infty\}$. Then, we have $cone(h_{\alpha})|_p = (E_{X/Z})_{\alpha|_{p'}}$. In this case, we obtain $\psi'''_{\alpha \beta}|_p = \psi_{\alpha \beta}|_{p'}$. So that $\nu$-equivalence property holds.

If $p$ lies on $\{\infty\} \in \PP^1$, let $p = \{p',\{\infty\}\}$ where $p' \in X_{\alpha} \times_X X_{\beta}$, $cone(h_{\alpha})|_p=u^*(E_{Y/Z})_{\alpha}|_{p'} \oplus E_{X/Y}|_{p'}$. In this case, we have $\psi_{\alpha \beta}''' = u^*(\psi'_{\alpha \beta})|_{p'} \oplus id$. So that $\nu$-equivalence property holds.
\end{proof}

\begin{lemm}\label{keylemma2}
\[ \Gysin{\cF|_A\oplus \cG|_A}(\eta_{u^*\phi'} \oplus \eta_{\phi''})_*[C_{A/C_{B/D}}] = u^!_{\cG} \circ v^!_{\cF} [D].
\]
\end{lemm}
\begin{proof}
Let $[C_{B/D}] = \sum_i a_i [C_i]$ be a fundamental cycle where $C_i$ are integral. Then, we have $(\eta_{\phi'})_*[C_{B/D}] = \sum_i a_i [\eta_{\phi'}(C_i)]$. Then, $v^!_{\cF}[D] = \sum_i a_i \Gysin{\cF|_V}[\eta_{\phi'}(C_i)]$. Let $K_i := \bar{\pi(C_i)}$ where $\pi : \fN_{Y/Z}|_V \to V$ be the projection. Choose modified proper representatives $(f_i : S_i \to K_i, q_i : \cV_i \surra \cF|_{S_i}, W_i)$ of $\eta_{\phi'}(C_i)$s. Then, $v^!_{\cF}[D] = \sum_i\frac{a_i}{d_{W_i}} \cdot (f_i)_* \Gysin{\cV_i}[W_i]$. Thus, we have :
\begin{align*}
u^!_{\cG} \circ v^!_{\cF}[D] = \sum\limits_i \frac{a_i}{d_{W_i}}\cdot u^!_{\cG}(f_i)_* \Gysin{\cV_i}[W_i].
\end{align*}

Let $R_i := S_i \times_V U$, Consider a zero section $0 : S_i \to \cV_i$ then there is a composition map $R_i \to S_i \stackrel{0}{\to} \cV_i$. Let $T_i := R_i\times_{\cV_i} W_i$. Note that $T_i \subset R_i$.
By Lemma \ref{virtgysincomm}, we have
\begin{align*}
& \sum\limits_i \frac{a_i}{d_{W_i}}\cdot u^!_{\cG}(f_i)_* \Gysin{\cV_i}[W_i] = \sum\limits_i \frac{a_i}{d_{W_{i}}}(g_i)_* u^!_{\cG} \Gysin{\cV_i}[W_i] = \sum\limits_i \frac{a_i}{d_{W_{i}}}(g_i)_* \Gysin{\cV_i|_{R_i}}u^!_{\cG}[W_i] \\
& = \sum\limits_i \frac{a_i}{d_{W_{i}}}(g_i)_* \Gysin{\cV_i|_{R_i}}\Gysin{ \cG|_{R_i} }(\eta_{\phi''})_*[C_{T_i/W_i}] = \sum\limits_i \frac{a_i}{d_{W_{i}}}(g_i)_* \Gysin{\cV_i|_{R_i} \oplus \cG|_{R_i}}(\eta_{\phi''})_*[C_{T_i/W_i}] 
\end{align*}
in $A_*(U)$.

On the other hand, we have
\[ \Gysin{\cF|_U \oplus \cG|_U}(\eta_{u^*\phi'} \oplus \eta_{\phi''})_*[C_{A/C_{B/D}}]= \sum\limits_i a_i\cdot \Gysin{\cF|_U \oplus \cG|_U}(\eta_{u^*\phi'} \oplus \eta_{\phi''})_*[C_{A_i/C_i}] 
\]
where $A_i := A\times_B C_i$.


Next we compute the cycle $(\eta_{\phi''})_*[C_{T_i/W_i}]$. Define $(\eta_{\phi''_{\alpha}}\oplus q_i)_*[(C_{T_i/W_i})_{\alpha}]$ to be a cycle obtained from the projection $(\pi''_{\alpha}\oplus q_i)_*[(\wtil{C}_{T_i/W_i})_{\alpha}]$, where $\pi''_{\alpha}$ is the projection $(E''_{\alpha})^1|_{(R_i)_{\alpha}} \to \cF|_{(R_i)_{\alpha}}$ and $(\wtil{C}_{T_i/W_i})_{\alpha} := (C_{T_i/W_i})_{\alpha}\times_{\bdst{(E''_{\alpha})\cseq}|_{ (R_i)_{\alpha} } \oplus \cV_i|_{(R_i)_{\alpha}} } (E''_{\alpha})^1|_{ (R_i)_{\alpha} } \oplus \cV_i|_{ (R_i)_{\alpha} }$, $ (C_{T_i/W_i})_{\alpha} := C_{T_i/W_i}\times_X {X_{\alpha}} $ and $(R_i)_{\alpha} := (R_i)\times_{X} X_{\alpha}$. This cycle is well-defined since we can show that integral components of $(\wtil{C}_{T_i/W_i})_{\alpha}$ are saturated invariant substacks of $(E''_{\alpha})^1|_{(R_i)_{\alpha}} \oplus \cV_i|_{(R_i)_{\alpha}}$ because $W_i$ are saturated invariant substacks of $\cV_i$.

Then we can easily check that the collections $\{ (\eta_{\phi''_{\alpha}}\oplus q_i)_*[(C_{T_i/W_i})_{\alpha}] \}_{\alpha \in \cI}$ satisfies the descent conditions therefore they descend to a cycle, which is denoted by $(\eta_{\phi''}\oplus q_i)_*[C_{T_i/W_i}] \in Z_*(\cG|_{R_i} \oplus \cF|_{R_i})$.
 
We can easily check that 
\[
\Gysin{\cG|_{R_i} \oplus \cV_i|_{R_i}}(\eta_{\phi''})_*[C_{T_i/W_i}] = \Gysin{\cG|_{R_i} \oplus \cF|_{R_i}}(\eta_{\phi''}\oplus q_i)_*[C_{T_i/W_i}]
\]
because we can compute both side by using exactly the same linear sum of modified proper representatives. Therefore, it is enough to check that
\[
(g_i)_*(\eta_{\phi''}\oplus q_i)_*[C_{T_i/W_i}] = d_{W_i}\cdot(\eta_{u^*\phi'} \oplus \eta_{\phi''})_*[C_{A_i/C_i}].
\]

We claim that
\[
(g_i)_*(\eta_{\phi''_{\alpha}}\oplus q_i)_*[(C_{T_i/W_i})_{\alpha}] = d_{W_i}\cdot(\eta_{u^*\phi'_{\alpha}}\oplus \eta_{\phi''_{\alpha}})_*[(C_{A_i/C_i})_{\alpha}].
\]

Let $(W_i')_{\alpha} := q_i((W_i)_{\alpha}) \times_{\cF|_{(S_i)_{\alpha}}} (E'_{\alpha})^1|_{(S_i)_{\alpha}}$ where $(W_i)_{\alpha} := (W_i)\times_{Y} Y_{\alpha}$ and $(S_i)_{\alpha} := S_i\times_Y Y_{\alpha}$. 

For the projections $\eta_{\phi''_{\alpha}}\oplus \pi'_{\alpha}|_{R_i} : (\bdst{(E''_{\alpha})\cseq})|_{R_i}\oplus (E'_{\alpha})^1|_{R_i} \to \cG|_{R_i}\oplus \cF|_{R_i}$ and $\eta_{\phi''_{\alpha}}\oplus \pi'_{\alpha}|_{U} : \bdst{(E''_{\alpha})\cseq}_{U}\oplus (E'_{\alpha})^1|_U  \to \cG|_U \oplus \cF|_U$, we have : 
\begin{align*}
&(\eta_{\phi''_{\alpha}}\oplus \pi'_{\alpha}|_{R_i})_*[C_{(T_i)_{\alpha}/(W_i')_{\alpha}}] = (\eta_{\phi''_{\alpha}}\oplus q_i)_*[(C_{T_i/W_i})_{\alpha}] \textrm{ and}\\
&(\eta_{\phi''_{\alpha}}\oplus \pi'_{\alpha}|_{U})_*[C_{(A_i)_{\alpha}/(\wtil{C}_i)_{\alpha}}] =
(\eta_{u^*\phi'_{\alpha}} \oplus \eta_{\phi''_{\alpha}} )_*[(C_{A_i/C_i})_{\alpha}]. 
\end{align*}

where $(\wtil{C}_i)_{\alpha} := (C_i)_{\alpha}\times_{\bdst{(E'_{\alpha})\cseq}|_U} (E'_{\alpha})^1|_U$, $(T_i)_{\alpha} := T_i \times_{X} X_{\alpha}$.



From the fact that $(g_i)_*[q_i(W_i)] = d_{W_i}[\eta_{\phi'}(C_i)]$, we can easily check that $(g_i)_*[W_i']_{\alpha} = d_{W_i}[(\wtil{C}_i)_{\alpha}]$. Therefore, in a similar manner as in the proof of Lemma \ref{conepushforward2}, we can show that 
\begin{align*}
& (g_i)_*(\eta_{\phi''_{\alpha}}\oplus q_i)_*[(C_{T_i/W_i})_{\alpha}] = (g_i)_*(\eta_{\phi''_{\alpha}}\oplus q_i)_*[C_{(T_i)_{\alpha}/(W_i)_{\alpha}}] \\
& = d_{W_i}\cdot(\eta_{u^*\phi'_{\alpha}} \oplus \eta_{\phi''_{\alpha}})_*[C_{(A_i)_{\alpha}/(C_i)_{\alpha}}] = d_{W_i}\cdot(\eta_{u^*\phi'_{\alpha}} \oplus \eta_{\phi''_{\alpha}})_*[(C_{A_i/C_i})_{\alpha}]
\end{align*}

This implies that the collection $\{(g_i)_*(\eta_{\phi''_{\alpha}}\oplus q_i)_*[(C_{T_i/W_i})_{\alpha}] \}_{\alpha\in \cI}$ descend to a cycle $d_{W_i}\cdot(\eta_{u^*\phi'} \oplus \eta_{\phi''})_*[C_{A_i/C_i}]$. Therefore we proved $
(g_i)_*(\eta_{\phi''}\oplus q_i)_*[C_{T_i/W_i}] = d_{W_i}\cdot(\eta_{u^*\phi'} \oplus \eta_{\phi''})_*[C_{A_i/C_i}].$

\end{proof}

\begin{prop}
semi-virtual pull-back $u^!_{\cE}$ defined in Definition \ref{def:semivert} defines a bivariant class in $A^*(X \stackrel{u}{\lra} Y)$.
\end{prop}
\begin{proof}
It is clear by Proposition \ref{vpullbackcompatibility} and Proposition \ref{functoriality}.
\end{proof}

\begin{prop}[Commutativity]\label{virtgysincomm2}
Consider the following fiber diagram
\begin{align*}
\xymatrix{
D \ar[r]^{s} \ar[d]_{r} \ar@{}[rd]|{\square} & C \ar[r] \ar[d] \ar@{}[rd]|{\square} & Z \ar[d]^{v} \\
B \ar[r] \ar[d] \ar@{}[rd]|{\square} & A \ar[r] \ar[d] & W \\
X \ar[r]_{u} & Y &
} 
\end{align*}
such that $X,Y,Z,W,A,B,C,D$ are DM stacks, $u,v$ are DM-type morphisms. Assume that there are semi-perfect obstruction theory $\phi$ on $X/Z$. Let $\phi = \{ \{X_{\alpha}\}_{\alpha \in \cI}, \phi_{\alpha}, \psi_{\alpha \alpha'} Ob_{\phi} =: \cE \}$ and $\phi$ on $X/Z$. Let $\phi' = \{ \{Z_{\beta}\}_{\beta \in \cJ}, \phi_{\beta}, \psi_{\beta \beta'}, Ob_{\phi'} =: \cG \}$. Let $\cF$ be a coherent sheaf stack on $\cA$.

Then we have
\[
v^!_{\cG}\circ u^!_{\cE} = u^!_{\cE} \circ v^!_{\cG} : A_*(\cF) \to A_*(\cF|_D).
\]
\end{prop}
\begin{proof}
Let $A' \subset A$ be an integral substack. Consider the Vistoli's rational equivalence(\cite[Lemma 3.16]{Vis89}, \cite[Proposition 4]{Kre99can})
\[
[ C_{ C_{B'/A'}\times_{B'} D' / C_{B'/A'} } ] = 
[ C_{ C_{C'/A'}\times_{C'} D' / C_{C'/A'} } ]
\]
in $A_*(C_{B'/A'}\times_{A'} C_{C'/A'})$. Then, it is enough to show that 
\[
\Gysin{\cE|_D \oplus \cG|_D} (\eta_{\phi} \oplus \eta_{\phi'})_* [C_{ C_{B'/A'}\times_{B'} D' / C_{B'/A'} }] = v^!_{\cG} \circ u^!_{\cE} [A'].
\]
where $(\eta_{\phi} \oplus \eta_{\phi'})_* [C_{ C_{B'/A'}\times_{B'} D' / C_{B'/A'} }]$ is a cycle which is a descent of the collection 
\begin{align*}
& \{ (\eta_{\phi_{\alpha}} \oplus \eta_{\phi'_{\beta}})_* [C_{ C_{B'/A'}\times_{B'} D' / C_{B'/A'} } \times_{D'} D'_{\alpha \beta}] \}_{\alpha\in \cI,\beta \in \cJ} \\
& = \{ (\eta_{\phi_{\alpha}} \oplus \eta_{\phi'_{\beta}})_* [C_{ ((C_{B'/A'})_{\alpha} \times_{B'} D')\times_{D'}D'_{\beta} / (C_{B'/A'})_{\alpha} }] \}_{\alpha\in \cI,\beta \in \cJ} \\
& = \{ (\eta_{\phi_{\alpha}} \oplus \eta_{\phi'_{\beta}})_* [C_{ ((C_{B'/A'})_{\alpha} \times_{B'} D'_{\beta}) / (C_{B'/A'})_{\alpha} }] \}_{\alpha\in \cI,\beta \in \cJ}
\end{align*}

where $D'_{\alpha \beta} = (D'\times_{X} X_{\alpha})\times_Z Z_{\beta}$, $(C_{B'/A'})_{\alpha} = C_{B'/A'}\times_{X} X_{\alpha}$ and $D'_{\beta} = D' \times_Z Z_{\beta}$.


We have 
\[
v^!_{\cG} \circ u^!_{\cE} [A'] = v^!_{\cG} \circ \Gysin{\cE|_B} (\eta_{\phi'})_* [C_{B'/A'}]
\]
Let $[C_{B'/A'}] = \sum_i a_i [K_i]$ be a fundamental cycle where $K_i$ are integral. Then we have $(\eta_{\phi'})_* [C_{B'/A'}] = \sum_i a_i [\eta_{\phi'}(K_i)]$. We may assume that $\eta_{\phi'}(K_i)$ is an integral substack of $\cE|_{Z_i}$ where $Z_i \subset B$ is an integral closed substack of $B$ for each $i$.
Choose a proper representative $(f_i : S_i \to Z_i, q_i : \cV_i \surra \cE|_{S_i}, W_{K_i})$ for each $\eta_{\phi'}(K_i)$. Then we have

\begin{align*}
& v^!_{\cG} \circ \Gysin{\cE|_B} (\eta_{\phi'})_* [C_{B'/A'}] = \sum_i \frac{a_i}{d_{W_{K_i}}} v^!_{\cG} (f_i)_* \Gysin{\cV_i}[W_{K_i}] = \sum_i \frac{a_i}{d_{W_{K_i}}} (f_i)_* \Gysin{\cV_i} v^!_{\cG}[W_{K_i}] \\
& = \sum_i \frac{a_i}{d_{W_{K_i}}} (f_i)_* \Gysin{\cV_i} \Gysin{\cG|_{ (\cV_i|_{T_i}) }} (\eta_{\phi'})_*[C_{ W_{K_i}\times_{S_i} T_i / W_{K_i} }]
\end{align*}
by Proposition \ref{virtgysincomm} where $T_i := S_i\times_B D$.

We can consider $C_{ W_{K_i}\times_{S_i} T_i / W_{K_i} }$ as cycle in $Z_*(\cV|_{T_i} \times_{T_i} C_{T_i/S_i} )$. Then we have
\[
\Gysin{\cV_i} \Gysin{\cG|_{ (\cV_i|_{T_i}) }} (\eta_{\phi'})_*[C_{ W_{K_i}\times_{S_i} T_i / W_{K_i} }] = \Gysin{\cG|_{T_i} \oplus \cV_i|_{T_i}} (\eta_{\phi'})_*[C_{ W_{K_i}\times_{S_i} T_i / W_{K_i} }].
\]

In a similar manner as in the proof of Proposition \ref{functoriality}, we define a cycle $(\eta_{\phi}\oplus q_i)_*[C_{ W_{K_i}\times_{S_i} T_i / W_{K_i} }]$ and we have
\[
\Gysin{\cG|_{T_i} \oplus \cV_i|_{T_i}} (\eta_{\phi'})_*[C_{ W_{K_i}\times_{S_i} T_i / W_{K_i} }] = \Gysin{\cE|_{T_i} \oplus \cG|_{T_i}}(\eta_{\phi'}\oplus q_i)_*[C_{ W_{K_i}\times_{S_i} T_i / W_{K_i} }].
\]
In a similar manner as in the proof of Proposition \ref{functoriality}, we can check that
\[
(f_i)_*\Gysin{\cE|_{T_i} \oplus \cG|_{T_i}}(\eta_{\phi}\oplus q_i)_*[C_{ W_{K_i}\times_{S_i} T_i / W_{K_i} }] = d_{W_{K_i}} \cdot \Gysin{\cE|_D \oplus \cG|_D} (\eta_{\phi} \oplus \eta_{\phi'} )_* [C_{K_i\times_{S_i} T_i / K_i}].
\]

Hence we have
\begin{align*}
& \sum_i \frac{a_i}{d_{W_{K_i}}} (f_i)_* \Gysin{\cV_i} \Gysin{\cG|_{ (\cV_i|_{T_i}) }} (\eta_{\phi'})_*[C_{ W_{K_i}\times_{S_i} T_i / W_{K_i} }] = \sum_i a_i \cdot \Gysin{\cE|_D \oplus \cG|_D} (\eta_{\phi} \oplus \eta_{\phi'} )_* [C_{K_i\times_{S_i} T_i / K_i}] \\
& = \Gysin{\cE|_D \oplus \cG|_D} (\eta_{\phi} \oplus \eta_{\phi'})_* [( C_{ C_{B'/A'}\times_{B'} D' / C_{B'/A'} } )].
\end{align*}
\end{proof}

\section{Virtual fundamental classes via semi-perfect obstruction theories and Torus localization}\label{sec:virtfund}

Using the results in the previous section, we define and show basic properties of virtual fundamental classes via semi-perfect obstruction theories. 

Consider a DM-type morphism $u : X \to \cM$ from a DM stack $X$ to $\cM$ which is a smooth Artin stack with pure dimension. We assume that there is a semi-perfect obstruction theory $\phi$ on $X/Y$ with $Ob_{\phi} = \cE$. Then, we define a virtual fundamental class via $\phi$ to be the following :

\begin{defi}\cite[Definition-Theorem 3.1]{CL11semi}
We define a virtual fundamental class via $\phi$ to be :
\[ [X,\phi]\virt := \Gysin{\cE}(\eta_{\phi})_*[C_{X/\cM}] = u^!_{\cE}[\cM].
\]
\end{defi}

The next theorem is a slight generalization of Proposition \cite[Proposition 3.1]{CL11semi}. Consider a morphism of Artin stacks :
\[ \xymatrix{X \ar[r]^u & Y \ar[r]^v & Z}.
\]
where $X,Y$ are DM stacks and $u,v$ are DM-type morphisms
Assume that there are semi-perfect obstruction theories $\phi$ on $X/\cM$, $\phi'$ on $Y/\cM$, and $\phi''$ on $X/Y$ where $\phi$, $\phi'$, $\phi''$ are compatible such that $Ob_{\phi} = \cE$, $Ob_{\phi'} = \cF$, $Ob_{\phi''} = \cG$. Then we have the following.

\begin{theo}\label{deformation invariance} In the above setting, we have :
\[u^!_{\cG}[Y,\phi']\virt = [X,\phi]\virt.
\]
\end{theo}
\begin{proof}
Since $[Y,\phi']\virt = v^!_{\cF}[\cM]$ by definition, the theorem directly follows from Proposition \ref{functoriality}.
\end{proof}

And from Proposition \ref{comparison}, we directly obtain the following theorem.

Let $u : X \to \cM$ be a DM-type morphism from a DM stack $X$ to $\cM$ which is smooth Artin stack with pure dimension, and let $\phi : (E_{X/\cM})\hseq \to L_{X/\cM}$ be a perfect obstruction theory with $Ob_{\phi} = \cE$. Since $\phi$ is also a semi-perfect obstruction theory, there are two kinds of virtual fundamental classes. Let $[X,E_{X/\cM}]\virt$ be a virtual fundamental class defined via the perfect obstruction theory, and let $[X,\phi]\virt$ be a virtual fundamental class defined via the semi-perfect obstruction theory. Then, we have :

\begin{theo}
\[ [X,E_{X/\cM}]\virt = [X,\phi]\virt
\]
\end{theo}
\begin{proof}
$[X,E_{X/\cM}]\virt = u^!_{\EE_{X/\cM}}[\cM] = u^!_{\cE}[\cM] = [X,\phi]\virt$ by Proposition \ref{comparison}.
\end{proof}





\subsection{Torus localization} In this section, we recover the result in \cite[Theorem 4.5]{Kie16locsemi} using semi-virtual pull-back.
Let $X$ be a DM stack equipped with a $\CC^*$-action. Let $X^F$ be a $\CC^*$-fixed locus.

\begin{defi}\cite[Definition 4.1]{Kie16locsemi}
A semi-perfect obstruction theory $\phi = ( \{X_{\alpha}\}_{\alpha \in \cI}, \{\phi_{\alpha}\}, \{\psi_{\alpha \beta}\}, Ob_{\phi} )$ is called $\CC^*$-equivariant if it satisfies the followings:
\begin{enumerate}
\item 
\'Etale charts $\{X_{\alpha} \to X\}$ are $\CC^*$-equivariant.
\item
Perfect equivariant theories $\phi_{\alpha} : E_{\alpha} \to L_{X_{\alpha}}$ are objects of $D([X_{\alpha}/\CC^*])$, which is the derived category of $\CC^*$-equivariant quasi-coherent sheaves on $X_{\alpha}$.
\end{enumerate}
\end{defi}

Let us assume that there is a $\CC^*$-equivariant semi-perfect obstruction theory $\phi$ on $X$.

Let $X^{F}_{\alpha} := X^{F}\times_X X_{\alpha}$. Then we can decompose since there is a $E_{\alpha}|_{X^F_{\alpha}}$ is $\CC^*$-equivariant, we can decompose it to the fixed part and the moving part by weights:
\begin{align*}
E_{\alpha}|_{X^F_{\alpha}} = E_{\alpha}|_{X^F_{\alpha}}^{F} \oplus E_{\alpha}|_{X^F_{\alpha}}^{M}
\end{align*}
Then there is an induced morphism
$\phi_{\alpha}^F : E^F_{\alpha} \stackrel{\phi_{\alpha}^{F}}{\lra} L_{X_{\alpha}}|_{X^F_{\alpha}}^{F} \to L_{X^F_{\alpha}}$, where $L_{X_{\alpha}}|_{X^F_{\alpha}}^{F}$ is the fixed part of $L_{X_{\alpha}}|_{X^F_{\alpha}}$.

\begin{lemm}\cite[Lemma 4.2]{Kie16locsemi} Each $\phi_{\alpha}^F$ are perfect obstruction theories and they compose a perfect obstruction theory $\phi^F = (\{X^F_{\alpha}\}_{\alpha \in \cI}, \{\phi^F_{\alpha}\}, \{\psi_{\alpha \beta}|_{X^F_{\alpha}}^{F}\},  Ob_{\phi^F})$ is a semi-perfect obstruction theory on $X^F$ where $\psi_{\alpha \beta}|_{X^F_{\alpha}}^{F} : Ob(\phi_{\alpha}^F)|_{X_{\alpha} \times_X X_{\beta}} = Ob(\phi_{\alpha})|_{X_{\alpha} \times_X X_{\beta}}^F \to Ob(\phi_{\beta})|_{X_{\alpha} \times_X X_{\beta} }^F = Ob(\phi_{\beta}^F)|_{X_{\alpha} \times_X X_{\beta} }$ are induced morphisms. 
\end{lemm}

Now we assume that the collection $\{(E_{\alpha}|_{X_{\alpha}^F}^{M})\dual\}_{\alpha\in\cI}$ glues to a two term complex $N^{vir} = [N^0 \to N^1]$. 
Then we can recover the following result using semi-virtual pull-back.

\begin{theo}\cite[Theorem 4.5]{Kie16locsemi}
\[
[X,\phi]\virt = \iota_* \frac{e(N^1)[X^F,\phi^F]\virt}{e(N^0)} \in A_*^{\CC^*}(X)\otimes_{\QQ[t]}\QQ[t,t^{-1}]
\]
where $\iota : X^F \hra X$ is an inclusion.
\end{theo}
\begin{proof}
We know that the natural map $\iota_* : A_*(X^F)\otimes_{\QQ[t]}\QQ[t,t^{-1}] \to A_*^{\CC}(X)\otimes_{\QQ[t]}\QQ[t,t^{-1}]$ is an isomorphism \cite{EG98loc}.
Therefore we can write $[X,\phi]\virt = \iota_*\xi$ for some $\xi \in A_*(X^F)\otimes_{\QQ[t]}\QQ[t,t^{-1}]$.
Let $[N^0 \to N^1]\dual =: [N_{-1} \to N_0]$.

Next we modify the semi-perfect obstruction theory $\phi^F$ by changing $\phi_{\alpha}^F : E_{\alpha}^F \to L_{X^F_{\alpha}}$ by $ \phi^F_{\alpha}\oplus 0 : E_{\alpha}^F\oplus [N_{-1} \to 0] \to L_{X^F_{\alpha}}$. Then it is clear that this is also a semi-perfect obstruction theory. We call this modified semi-perfect obstruction theory by $\wtil{\phi}^F$.

Since the mapping cone of the natural morphism $E_{\alpha}|_{X^F_{\alpha}} \to E_{\alpha}^F \oplus [N_{-1}][1]$ is $N_0[1]$, we can observe that there is an induced perfect obstruction theory $\phi_{\alpha}' : N_0[1] \to L_{X^F_{\alpha}/X_{\alpha}}$. In this case, obstruction bundles are just $N^0|_{X_{\alpha}}$, we can check that $(\phi^F)' = (\{X^F_{\alpha}\}_{\alpha\in \cI}, \phi_{\alpha}', N^0)$ is a relative semi-perfect obstruction theory over $X^F/X$.

Then we can observe that $\phi, \wtil{\phi}^F, (\phi^F)'$ are compatible semi-perfect obstruction theories. 
For the projection $p : X^F \to Spec(k)$, we have $[X^F,\wtil{\phi}^F]\virt = p^!_{Ob_{\wtil{\phi}^F}}[pt]$. Therefore, by functoriality of virtual pull-backs, we have $p^!_{Ob_{\wtil{\phi}^F}}[pt] = \iota_{N^0}^! \circ q_{Ob_{\phi}}^![pt]$ where $q : X \to pt$ is the projection. Thus we have
\[
[X^F,\phi^F]\virt \cap e(N^1) = [X^F,\wtil{\phi}^F]\virt = \iota_{N^0}^! \circ [X,\phi]\virt.
\]


Hence we have $[X^F,\phi^F]\virt \cap e(N^1) = \iota^!_{N^0}[X,\phi]\virt = \iota^!_{N^0}\iota_*\xi = e(N^0)\cap \xi.$ So that we have $\xi = \frac{e(N^1)[X^F,\phi^F]\virt}{e(N^0)}$.
\end{proof}

\bibliographystyle{plain}
\bibliography{library}
\end{document}